\documentclass[oneside,12pt,a4paper,reqno]{amsart}

\usepackage{makecell}
\usepackage{amsmath,amsthm,amssymb,enumerate}
\usepackage{graphicx}
\usepackage{color,floatrow}
\usepackage{tabularx}
\usepackage{tikz,hyperref}
\usetikzlibrary{calc,arrows.meta}

\usepackage{todonotes}

\renewenvironment{proof}{\proof}{\hfill $\Box$ \vskip 12pt}

\renewcommand{\proof}{\par\noindent{\it Proof.\ \ }}
\def\qed{\ifmmode\square\else\nolinebreak\hfill
$\Box$\fi\par\vskip12pt}

\def\ov{\overline} 
\def\l{\langle} \def\r{\rangle} 
\def\div{\,\big|\,} \def\notdiv{{\,\not\big|\,}}

\def\le{\leqslant}
\def\ge{\geqslant}
\def\leq{\leqslant}
\def\geq{\geqslant}

\def\mod{{\rm mod~}}

\def\FF{{\mathbb F}} \def\ZZ{{\mathbb Z}}
\def\SS{{\mathbb S}}

\def\calT{{\mathcal T}}

\def\calC{{\mathcal C}}

\def\calK{{\mathcal K}}
\def\cali{{\mathcal I}}
\def\calj{{\mathcal J}}

\def\calM{{\mathcal D}}
\def\calD{{\mathcal D}}
\def\calP{{\mathcal P}}

\def\SS{{\mathbb S}}
\def\ZZ{{\mathbb Z}}

\def\bfe{{\bf e}}
\def\bfi{{\bf i}}

\def\mod{{\rm mod~}}

\def\K{{\bf K}}

\def\Aut{{\rm Aut}} \def\Inn{{\rm Inn}}

\def\HA{{\rm HA}} \def\AS{{\rm AS}} \def\PA{{\rm PA}} \def\TW{{\rm TW}}

\def\Cos{{\sf Cos}}
\def\BiCos{{\sf Cos}}

\def\D{{\rm D}} \def\Q{{\rm Q}}
\def\S{{\rm S}} 
 
\def\soc{{\rm soc}} 

\def\C{{\bf C}}\def\N{{\bf N}}\def\Z{{\bf Z}}

\def\Ga{{\it \Gamma}}

\def\Ome{{\it \Omega}}
\def\Sig{{\it \Sigma}}
\def\Del{{\it \Delta}}

 \def\b{\beta} 
 \def\s{\sigma}

\def\o{\omega}

\def\AGammaL{{\rm A\Gamma   L}}

\def\PSL{{\rm PSL}}\def\PGL{{\rm PGL}}
\def\SigmaL{{\rm \Sigma L}}
\def\GL{{\rm GL}} \def\SL{{\rm SL}}
\def\AGL{{\rm AGL}}

 \def\F{{\rm F}} \def\D{{\rm D}}

\def\HS{{\rm HS}}\def\CD{{\rm CD}}\def\HC{{\rm HC}}\def\SD{{\rm SD}}

\def\calu{{\mathcal U}}
\def\calU{{\mathcal U}}

\def\ol#1{\overline{#1}}

\newtheorem{theorem}{Theorem}[section]
\newtheorem{lemma}[theorem]{Lemma}%
\newtheorem{corollary}[theorem]{Corollary}%
\newtheorem{proposition}[theorem]{Proposition}%

\theoremstyle{definition}
\newtheorem{problem}[theorem]{Problem}%
\newtheorem{remark}[theorem]{Remark}
\newtheorem{definition}[theorem]{Definition}%
\newtheorem{example}[theorem]{Example}%
\newtheorem{construction}[theorem]{Construction}%

\newcommand \be{\begin{equation}}
\newcommand \ee{\end{equation}}
\newcommand \bes{\begin{eqnarray}}
\newcommand \ees{\end{eqnarray}}
\newcommand \bess{\begin{eqnarray*}}
\newcommand \eess{\end{eqnarray*}}

\begin{document}

\title{Coverings of Groups, Regular Dessins, and Surfaces}
\thanks{This work was partially supported  by NSFC: 11931005, 12061092, 12101518;
Yunnan Applied Basic Research Projects: 202101AT070137; Fundamental Research Funds for the Central
Universities: 20720210036.}

\author{Jiyong Chen}
\address{$^{1}$ School of Mathematical Sciences\\
Xiamen University\\
Xiamen, Fujian 361005\\
P. R. China}
\email{chenjy1988@xmu.edu.cn}

\author{Wenwen Fan}
\address{$^{2}$ School of Mathematics\\
Yunnan Normal University\\
Kunming, Yunnan 650500\\
P. R. China}
\email{fww0871@163.com}

\author{Cai Heng Li}
\address{$^{3}$ SUSTech International Center for mathematics\\
Department of Mathematics\\
Southern University of Science and Technology \\
Shenzhen 518055\\
P. R. China}
\email{lich@sustech.edu.cn}

\author{Yan Zhou Zhu}
\address{$^{4}$
% SUSTech International Center for Mathematics \\
Department of Mathematics\\
Southern University of Science and Technology \\
Shenzhen 518055\\
P. R. China}
\email{zhuyz@mail.sustech.edu.cn}

\date\today

\begin{abstract}
A coset geometry representation of regular dessins is established, and employed to describe quotients and coverings of regular dessins and surfaces.
A characterization is then given of face-quasiprimitive regular dessins as coverings of unicellular regular dessins.
It shows that there are exactly three O'Nan-Scott-Praeger types of face-quasiprimitive regular dessins which are smooth coverings of unicellular regular dessins, leading to new constructions of interesting families of regular dessins.
Finally, a problem of determining smooth Schur covering of simple groups is initiated by studying coverings between $\SL(2,p)$ and $\PSL(2,p)$, giving rise to interesting regular dessins like Fibonacci coverings.
\end{abstract}

\maketitle

\section{Introduction}\label{introduction}

A {\it dessin} $\calD$ is a 2-cell embedding of a bipartite graph $(V,E)$ on an orientable closed surface $\SS$
such that the vertex set $V$ is partitioned into two parts $B$ and $W$
with vertices colored {\it black} and {\it white}, respectively, and
the supporting surface $\SS$ is partitioned into topological discs. The discs are called {\it faces}, and the set of faces is denoted by $F$.
A dessin $\calD$ is usually viewed as an incidence triple of vertices, edges and faces, and written as
$\calD=( B \cup W,E,F)$. The bipartite graph $\Ga=(B \cup W, E)$ is called the {\it underlying graph}.
In a dessin $\calD=(B\cup W, E, F)$, the number of edges incident with a vertex is called the {\it valency} at the vertex, and the number of edges on the boundary cycle of a face is called the {\it face length} of the face.

An {\it automorphism} of a dessin $\calD$ is a permutation on $(B\cup W)\cup E\cup F$
which preserves the sets $ B,W,E$ and $F$, and their incidence relations,
also preserves the orientation of the supporting surface.
Let $\Aut\calD$ be the group of automorphisms of $\calD$.
Then $\Aut\calD$ acts semi-regularly on the edge set $E$ of $\calD$. So, if $\Aut\calD$ is transitive on  $E$, then it is regular on $E$, and $\calD$ is thus called a {\it regular dessin}.

The importance of dessins has been well recognised, see \cite{Grothendieck,Jones,Jones-Korea,Jones-2013};
special classes of regular dessins have been studied, see \cite{CJSW,unique-embedding,pe-pf,Jones-2010,Jones-2007}.
In this paper, we systematically study regular dessins, and establish a theory regarding quotients and coverings of groups, regular dessins and surfaces.

\subsection{Coset geometry and quotients of dessins}\

It is well-known that each regular dessin $\calD$ can be identified with a group $G$ and two generators $b,w$, which is denoted by $\calD(G,b,w)$, refer to \cite[Chapter\,2]{Jones2016Dessins} or Lemma~\ref{G2}.
It follows that a group $G$ is the automorphism group of a regular dessin if and only if $G$ is 2-generated.
Studying 2-generated groups is an important topic in group theory, and has a long and rich history, see \cite{Burness} for references.
A particularly impressive result is that each finite simple group is 2-generated.
Obviously, any quotient group of a 2-generated group is also 2-generated, and this induces quotients of dessins and then induces quotients of surfaces, which we explain below.

Let $\calD=(B\cup W,E,F)$ and $\calD'=(B'\cup W',E',F')$, and let $\varphi$ be a mapping from $B\cup W\cup E\cup F$ to $B'\cup W'\cup E'\cup F'$ such that $\varphi(X)\subseteq X'$ for $X\in \{B,W,E,F\}$.
Then $\varphi$ is a {\it homomorphism} if $\varphi$ preserves the incidence relations.
If further $\varphi$ is a bijection, then $\calD$ is {\it isomorphic} to $\calD'$ and $\varphi$ is an {\it isomorphism}.

\begin{definition}\label{def-quotient}
{\rm
Let $\calD=\calD(G,b,w)=(V,E,F)$ be a regular dessin.
For a normal subgroup $N\lhd G$, we define
\begin{itemize}
\item[(i)] `{\sf geometric quotient}': $\calD_N=(V_N,E_N,F_N)$, where $V_N,E_N,F_N$ are the sets of $N$-orbits on $V,E,F$, respectively;

\item[(ii)] `{\sf algebraic quotient}': $\calD/N=\calD(\ov G,\ov b,\ov w)$, where $\ov G=G/N$, and $(\ov b,\ov w)=(bN,wN)$.
\end{itemize}
}
\end{definition}

It will be shown that $\calD_N$ is indeed a regular dessin with $\Aut\calD_N=G/N$, and
\[\calD_N\cong\calD/N,\]
see Theorem~\ref{quotient}.
The geometric quotient focuses on the actions of the automorphism groups of dessins, and so permutation group theory plays an important role, leading to a description of face quasiprimitive regular dessins given in Theorem~\ref{qp-types}.
On the other hand, the algebraic quotient emphasizes on the 2-generated groups, which was the principle motivation for Theorems~\ref{Fibonacci}.

A quotient of a dessin naturally induces a quotient of the supporting surface.

\begin{definition}\label{surface-coverings}
{\rm
For two orientable (closed) surfaces $\SS$ and $\SS'$,
if there is a positive integer $n$, a finite subset $\Del\subset\SS'$, and a continuous function $\varphi:\ \SS\to \SS'$ such that each point of $\SS'\setminus\Del$ has precisely $n$ preimages in $\SS$,
then $\SS'$ is called a {\it $n$-sheeted quotient} of $\SS$.
In this case, points in $\Del$ are called {\it ramification points} or {\it branched points}, and $\SS$ is called a (ramification) {\it covering} of $\SS'$.
If a covering has no ramification point, namely $\Del=\emptyset$, then $\SS$ is said to be a {\it smooth covering} of $\SS'$.
}
\end{definition}

We shall see that quotients and coverings of surfaces can be realized by quotient and coverings of regular dessins, defined below.
Note that, a bipartite graph is called {\it bi-regular} if the vertices in the same part have the same valency.
Moreover, a bi-regular bipartite graph is of bi-valency $(k_1,k_2)$ if the valencies of the two parts are $k_1$ and $k_2$, respectively.

\begin{definition}\label{def-cover}
{\rm
Let $\calD$ be a regular dessin, and let $N\lhd G=\Aut\calD$.
Then $\calD$ is called a {\it ramification covering} of $\calD_N$, and furthermore,
\begin{itemize}
\item[(a)] $\calD$ is called a {\it smooth covering} of $\calD_N$ if $\calD$ and $\calD_N$ have the same bi-valency and face length;
\item[(b)] $\calD$ is called a {\it quasi-smooth covering} of $\calD_N$ if $\calD$ and $\calD_N$ have the same bi-valency;
% \item[(c)] $\calD$ is called a {\it totally branched covering} of $\calD_N$ if $\calD$ and $\calD_N$ have the same number of vertices and the same number of faces. {\color{red} (This looks a bit strange !)}
\item[(c)] $\calD$ is called a {\it totally branched covering} of $\calD_N$ if $\calD$ and $\calD_N$ have the same number of vertices and the same number of faces, or equivalently, $N$ is contained in the intersection of all vertex stabilizers and all face stabilizers of $\calD$;
\item[(d)] if $N$ is a minimal normal subgroup of $G$, then $\calD$ is called a {\it minimal covering} of $\calD_N$.
\end{itemize}
Correspondingly, a quotient dessin $\calD_N$ is called a (\textit{quasi-})\textit{smooth quotient} of $\calD$ if $\calD$ is a (quasi-)smooth covering of $\calD_N$.
}
\end{definition}

{\noindent \bf Remarks on Definition~\ref{def-cover}:}

\begin{enumerate}[{\rm (i)}]
\item By definition, $\calD$ is a quasi-smooth covering of $\calD_N$ if and only if the underlying graph of $\calD$ is a covering of the underlying graph of $\calD_N$.

% \item

\item Smooth coverings and totally branched coverings represent two extreme cases.
If $\calD$ is a smooth covering of $\calD_N$, then the intersection of $N$ with any vertex stabilizer or face stabilizer is trivial.
In contrast, if $\calD$ is a totally branched covering of $\calD_N$, then $N$ is contained in every vertex stabilizer and face stabilizer.
Recall that the Euler characteristic of a dessin $\calD=(B\cup W,E,F)$ is
\[\chi(\calD)=|B\cup W|-|E|+|F|.\]
For regular dessins $\calD$ and $\calD_N$ with negative Euler characteristic,  Theorem~\ref{quotient-dessin} tells us that
\[|N|\leqslant \frac{\chi(\calD)}{\chi(\calD_N)}\leqslant 42|N|-41,\]
in particular, the first equality holds if and only if $\calD$ is a smooth covering of $\calD_N$, and
the second equality holds if and only if $\calD$ is a totally branched covering of $\calD_N$ and $\calD_N$ is a Hurwitz dessin, see Theorem~\ref{upbound}.

\item Coverings of regular dessins provide a simple way to understand surface coverings, for instance,
{\it each surface of positive genus has infinitely many smooth coverings which can be realized by smooth coverings of dessins.}
See Corollary~\ref{HA-inf} and Theorem~\ref{HA-classification}.

\item For the extremal case $N=G$, the quotient $\calD_N$ is $\K_2$ on a sphere, and
$\calM$ is a $|G|$-sheeted covering of $\calM_N$ with precisely 3 branched points,
which is so-called Bely\v{\i} covering \cite{Belyi}.
This is a `degenerate case' in the sense that the quotient map $\calM_N=\K_2$ does not have much structural information of the original map $\calM$.
This would be one of the reasons why it is very difficult to reconstruct Bely\v{\i} coverings.

\end{enumerate}

As noticed before, a group is the automorphism group of a regular dessin if and only if it is 2-generated,
the concept of coverings of regular dessins has a group version.

\begin{definition}\label{smooth-extension}
{\rm
Let $H$ be a 2-generated group, and let $G=N.H$ be an extension of $N$ by $H$.
For an element $g\in G$, let $\ov g$ be the image of $g$ in $H=\ov G=G/N$.
\begin{itemize}
\item[(a)] If $G=\l b,w\r$ such that $(|b|,|w|)=(|\ov b|,|\ov w|)$, then $G$ is called a {\it quasi-smooth covering} of $\ov G$, and $(G,b,w)$ is said to be a {\it quasi-smooth covering} of $(\ov G,\ov b,\ov w)$.

\item[(b)] If $G=\l b,w\r$ with $(|b|,|w|,|bw|)=(|\ov b|,|\ov w|,|\ov{bw}|)$, then $G$ is called a {\it smooth covering} of $\ov G$, and $(G,b,w)$ is said to be a {\it smooth covering} of $(\ov G,\ov b,\ov w)$.
\end{itemize}
Correspondingly, $H=\ov G$ is called a {\it (quasi-)smooth quotient} of $G$.
}
\end{definition}

We observe that there exists a regular dessin $\calD(G,b,w)$ which is a (quasi-)smooth covering of a regular dessin $\calD(\ov G,\ov b,\ov w)$ if and only if the group $G$ is a (quasi-)smooth covering of the factor group $\ov G$.
It is easily shown that a dihedral group does not have a smooth covering.
A natural problem then arises.

\begin{problem}\label{prob:smooth-quotient}
{\rm
Characterize regular dessins which have proper smooth quotients;
equivalently, characterize $2$-generated groups that have proper smooth quotient groups.
}
\end{problem}

Taking quotient of regular dessins suggests us to study regular dessins in two steps:
\begin{itemize}
\item[(a)] characterize certain `basic dessins' or `basic 2-generated groups', and
\item[(b)] determine (smooth) coverings of given regular dessins or 2-generated groups.
\end{itemize}

With respect to taking geometric quotient, an extremal type of basic dessin has a single face, called a {\it unicellular} dessin, which has a cyclic automorphism group.
So this piques our interest in investigating coverings of unicellular regular dessins and cyclic groups, which will be described in Section~\ref{sec:f-trans}.
On the other hand, with taking algebraic quotients, it is natural to study smooth coverings of simple groups, which led us to address smooth Schur coverings of simple groups described in Section~\ref{Schur covering}.

\subsection{Face-quasiprimitive regular dessins}\label{sec:f-trans}\

Let $\calD=(B\cup W,E,F)$ be a regular dessin, and let $G=\Aut\calD$.
Then $G$ is transitive on each of the four sets $B$, $W$, $E$ and $F$.
In this paper, we focus on the transitive action of $G$ on the face set $F$.

A permutation group $G$ on a set $\Omega$ is said to be {\it quasiprimitive} if each non-trivial normal subgroup of $G$ is transitive on $\Omega$.
A regular dessin $\calD$ is called {\it face-quasiprimitive} if $\Aut\calD$ acts quasiprimitively and faithfully on the face set.

A covering $\calD$ of $\calD_N$ is said to be {\em minimal} if $N$ is a minimal normal subgroup of $\Aut\calD$.
Then a face-quasiprimitive regular dessin is a minimal covering of a unicellular dessin.
Naturally, we have the following problem, which is a subproblem of Problem~\ref{prob:smooth-quotient}.

\begin{problem}\label{face-qp-coverings}
    {\rm
    Determine smooth coverings of unicellular regular dessins;
    equivalently, determine smooth coverings of cyclic groups.
    }
\end{problem}

The well-known O'Nan-Scott-Praeger theorem divides the quasiprimitive groups into eight types, see Section~\ref{qp-sec} for details.

Lemma~\ref{G-types} will show that only four of the eight types appear as automorphism groups of face-quasiprimitive regular dessins, and three of these types correspond to smooth coverings of unicellular regular dessins. The following theorem provides a characterization of face-quasiprimitive regular dessins through smooth coverings of unicellular regular dessins.

\begin{theorem}\label{qp-types}
Let $\calD$ be a regular dessin, and let $G=\Aut\calD$ be face-quasiprimitive and $N=\soc(G)$.
Assume that $\calD$ is a smooth covering of $\calD_N$.
Then $\calD_N$ is unicellular, $G/N$ is cyclic, and $G=N{:}\ZZ_\ell$ is of type $\HA$, $\TW$ or $\AS$.
Conversely, $G$ is a smooth covering of $\ZZ_\ell$ if one of the following holds:
\begin{enumerate}[{\rm(i)}]
\item $G=N{:}\ZZ_\ell$ is of type $\HA$, with $\ell\ge3$, or

\item $G=T\wr\ZZ_\ell=T^\ell{:}\ZZ_\ell$ is of type $\TW$, where $\ell\ge5$ and $T$ is nonabelian simple, or

\item $G=\SigmaL(2,2^\ell)=\SL(2,2^\ell){:}\ZZ_\ell$, where $\ell\ge5$ is a prime.
\end{enumerate}
\end{theorem}

Theorem~\ref{qp-types} offers a rich resource for examples of smooth coverings of unicellular regular dessins.
However, for part~(iii), we have been unable to prove that $\SigmaL(2,2^\ell)$ is a smooth covering of $\ZZ_\ell$ for each integer $\ell\ge5$, leading to the following problem.

\begin{problem}\label{prob:simple-out}
{\rm
Determine almost simple groups $G=T.\ZZ_\ell$ which is a smooth covering of $\ZZ_\ell$, where $T=\soc(G)$ and $\ell\ge5$.
}
\end{problem}

Unicellular regular dessins have been well-characterised, see \cite{1-face,Singerman}.
A unicellular regular dessin of face length $2\ell$ has underlying graph being a complete bipartite graph $\K_{m,n}^{(\lambda)}$ with $\ell=mn\lambda$ such that the triple $(m,n,\lambda)\in T_\ell$, where
\[
T_\ell=\{(m,n,\lambda)\mid \ell=mn\lambda, \gcd(m,n)=1,\mbox{ and }\lambda_2<\max\{\ell_2,2\}\}, \eqno(1)
\]
where $\lambda_2$ is the 2-part of $\lambda$ and $\ell_2$ is the $2$-part of $\ell$.
For $(m,n,\lambda)\in T_{\ell}$, let
\[\begin{array}{rcl}
\calu_\ell&=&\{\mbox{unicellular regular dessins of face length $2\ell$}\};\\

\calu_\ell^{(\lambda)}&=&\{\mbox{unicellular regular dessins of face length $2\ell$, edge-mulitplicity $\lambda$}\};\\

\calK_{m,n}^{(\lambda)}&=&\{\mbox{unicellular regular dessins with underlying graph $\K_{m,n}^{(\lambda)}$}\}.
\end{array}\]
Then $\calu_\ell$ is a disjoint union of $\calu_\ell^{(\lambda)}$ with suitable $\lambda$, and $\calu_\ell^{(\lambda)}$ is the disjoint union of $\calK_{m,n}^{(\lambda)}$ with $(m,n,\lambda)\in T_\ell$.
Recalling that the vertices of a dessin are colored black and white, so each triple $(m,n,\lambda)$ uniquely determines a colored graph $\K_{m,n}^{(\lambda)}$,
in particular, if $m\not=n$, we have $\K_{m,n}^{(\lambda)}\not\cong\K_{n,m}^{(\lambda)}$, which implies that $|T_\ell|$ equals the cardinality of
\[\{\mbox{non-isomorphic colored graphs $\K_{m,n}^{(\lambda)}$ which underlies a dessin in $\calu_\ell$}\}.\]
%We remark that as colored graphs, if $m\not=n$, then $\K_{m,n}^{(\lambda)}$ and $\K_{n,m}^{(\lambda)}$ are not isomorphic.

As usual, for any positive integer $n$, the set of prime divisors of $n$ is denoted by $\pi(n)$, and the Euler totient function is denoted by $\phi$.
The following theorem counts regular unicellular dessins.
We remark that determining the cardinality $|\calu_\ell|$ was an unsettled  problem posed in \cite[Problem~B]{1-face}.

\begin{theorem}\label{Counting}
Let $\ell$ be a positive integer, and let $(m,n, \lambda)\in T_\ell$ as defined in $(1)$.
% $\ell=mn\lambda$ satisfy $(1)$.
Then the following statements hold.
\begin{enumerate}[{\rm(1)}]
\item $|\calu_\ell|=\ell$, namely, there are exactly $\ell$ non-isomorphic unicellular regular dessins with face length $2\ell$;

\item $|\calK_{m,n}^{(\lambda)}|=\phi(\lambda) \prod\limits_{p\in \pi}\frac{p-2}{p-1}$, and
$|\calu_\ell^{(\lambda)}|=2^{|\sigma|}\phi(\lambda) \prod\limits_{p\in \pi}\frac{p-2}{p-1}$,
where $\sigma=\pi(\ell/\lambda)$ and $\pi=\pi(\lambda)\setminus\sigma$;

\item $|T_\ell|=(2e_1+\delta)(2e_2+1)\dots(2e_s+1)$, where $\ell=p_1^{e_1}\cdots p_s^{e_s}$ is the prime factorization with $p_1<p_2<\dots<p_s$, and $\delta=0$ or $1$ for $\ell$ even or odd, respectively.
\end{enumerate}
\end{theorem}

The equality $|\calu_\ell|=\sum_{(m,n,\lambda)\in T_{\ell}}|\calK_{m,n}^{(\lambda)}|$ implies an interesting decomposition for integers.

\begin{corollary}\label{cor:ell-docm}
Let $\ell$ be a positive integer.
For each divisor $\lambda\div\ell$, let $\pi_\lambda=\pi(\lambda)\setminus \pi(\ell/\lambda)$ and $\pi'_\lambda=\pi(\ell/\lambda)$.
Then $\ell$ has a decomposition
\begin{equation*}
\ell=\sum_{\substack{\lambda\mid \ell \mbox{\ \footnotesize and} \\ \lambda_2<\max\{\ell_2,2\}}}2^{|\pi'_\lambda|}\phi(\lambda)\prod_{p \in \pi_\lambda}\frac{p-2}{p-1}.
\end{equation*}
\end{corollary}

{\noindent \bf Remark:}\
\begin{itemize}
 \item[(a)] Let $\calP_{2\ell}$ be a $2\ell$-gon with edges labeled $e_1,\dots,e_{2\ell}$.
Then a unicellular dessin with face length $2\ell$ can be formed by partitioning the $2\ell$ edges of $\calP_{2\ell}$ into $\ell$ pairs and identifying the two edges in each pair such that the resulted surface is orientable.
A remarkable result \cite[Theorem\,2]{HZ} establishes a formula for the number of dessins formed from a $2\ell$-gon and indexed by genus.
By Theorem~\ref{Counting}, the number of non-isomorphic unicellular regular dessins of face length $2\ell$ is equal to $\ell$.
\item[(b)] By part~(2) of Theorem~\ref{Counting}, it is easy to count the number of  non-isomorphic uncolored graphs $\K_{m,n}^{(\lambda)}$ underlies a dessin in $\calu_\ell$.
As uncolored graphs, $\K_{m_1,n_1}^{(\lambda_1)}\cong\K_{m_2,n_2}^{(\lambda_2)}$ if and only if $(m_1,n_1,\lambda_1)$ equals $(m_2,n_2,\lambda_2)$ or $(n_2,m_2,\lambda_2)$.
Note that $(m,n,\lambda)\in T_\ell$ if and only if $(n,m, \lambda)\in T_\ell$,  and $(m,n, \lambda)=(n, m, \lambda)\in T_\ell$ if and only if $(m,n,\lambda)=(1,1,\ell)$ with $\ell$ odd.
Hence the number of  non-isomorphic uncolored graphs $\K_{m,n}^{(\lambda)}$ underlies a dessin in $\calu_\ell$ equals
% $\frac{|T_\ell|+\delta}{2}$+
${(|T_\ell|+\delta)}/{2}$.
\end{itemize}

\subsection{Schur coverings}\label{Schur covering}\

Recall that a group $G$ is called {\it quasi-simple} if $G$ is a perfect group (namely, $G=G'$) and $G/\Z(G)$ is simple.
For a nonabelian simple group $S$, a group $G$ is called a {\it covering group} of $S$ if $G$ is perfect and $G/\Z(G)\cong S$; in this case, $\Z(G)$ is a factor group of the {\it Schur multiplier} of $S$.
It is well-known that each finite quasi-simple group is $2$-generated, and so acts edge-regularly on a dessin.

Very recently, Chen, Lubotzky and Tiep \cite[Theorem B]{Tiep-2024+} proves that,
{\it any finite quais-simple group $S$ with $\Z(S)\not=1$ is a smooth covering of $S/\Z(S)$.}
This solves a conjecture posed in the previous version of this paper, and motivates us to study the following Problem~\ref{conj:Schur-covering}.

For a regular dessin $\calD=\calD(G,b,w)$, we say $\calD$ is of \textit{type} $(\ell,m,n)$ if $(|b|,|w|,|bw|)=(\ell,m,n)$.
Then $\calD$ is a smooth covering of $\calD_N$ for $N\lhd G$ if and only if $\calD_N$ is also of type $(\ell,m,n)$.
A group $G$ is said to be a \textit{$(\ell,m,n)$-group} if there exists a regular dessin $\calD$ of type $(\ell,m,n)$ with $\Aut\calD\cong G$.
It is easy to see that a $(\ell,m,n)$-group is also a $(u,v,w)$-group, where $(u,v,w)$ is any permutation of $(\ell,m,n)$.
For a quasi-simple group $G=\langle b,w\rangle$ with $N=\Z(G)$, it is natural to ask whether $\calD=\calD(G,b,w)$ is a smooth covering of $\calD_N$, leading to the following problem:

\begin{problem}\label{conj:Schur-covering}
    {\rm
    Determine finite simple groups $S$ and $(\ell,m,n)$ such that there exists a regular dessin $\calD$ of type $(\ell,m,n)$ with $\Aut\calD\cong S$ which has a non-trivial smooth Schur covering.
    }
\end{problem}

Denote by $\mathrm{Spec}(G)$ the \textit{spectrum} of $G$, that is, the set of orders of elements in $G$.
We address Problem~\ref{conj:Schur-covering} for $\PSL(2,q)$.
Remark that $\PSL(2,2^f)$ with $f\geqslant 3$ has a trivial Schur multiplier; the only non-trivial Schur covering of $\PSL(2,q)$ is $\SL(2,q)$ for odd $q\geqslant 5$ except for $\PSL(2,9)\cong A_6$ (whose Schur multiplier is $6$).
The result of smooth Schur coverings of regular dessins of $\PSL(2,9)$ will be given in Lemma~\ref{lem:schurpsl29}.

\begin{theorem}\label{thm:sl2q}
    Let $S=\PSL(2,q)$ with $q=p^f\geqslant 5$ and $q\neq 9$ for odd prime $p$, and let $\ell,m,n\in\mathrm{Spec}(S)$ such that $1<\ell\leqslant m\leqslant n$.
    Then there exists a regular dessin $\calD$ of type $(\ell,m,n)$ of $S$ such that $\calD$ has a non-trivial smooth Schur covering if and only if
    \begin{enumerate}[{\rm(1)}]
        \item $\ell mn$ is odd with $(\ell,m,n)\neq (3,3,3)$ and $(\ell,m,n,q)\notin\{(3,3,p,p),(p,p,p,p)\}$; and
        \item $\{\ell,m,n\}\not\subset\mathrm{Spec}(\PSL(2,p^e))$ for any proper divisor $e$ of $f$.
    \end{enumerate}
\end{theorem}

% \begin{theorem}\label{thm:sl2q}
%     Let $q=p^f\geqslant 5$ with odd prime $p$, and let $S=\PSL(2,q)$.
%     Assume that $\ell\leqslant m\leqslant n$ are odd integers such that $(\ell,m,n)\neq (3,3,3)$, $\{\ell,m,n\}\subset\mathrm{Spec}(S)$ and $\{\ell,m,n\}\not\subset\mathrm{Spec}(\PSL(2,p^e))$ for any proper divisor $e$ of $f$.
%     Then there exists a regular dessin $\calD$ of type $(\ell,m,n)$ with $\Aut\calD\cong S$ such that $\calD$ has a non-trivial smooth Schur covering, except for either
%     \begin{enumerate}[{\rm(1)}]
%         %\item $\ell mn$ is even;
%         \item $q=9$, and $(\ell,m,n)=(3,5,5)$; or
%         \item $q=p$ is a prime, and $(\ell,m,n)=(3,3,p)$ or $(p,p,p)$.
%     \end{enumerate}
% \end{theorem}
Obviously, $\PSL(2,p)$ is a quasi-smooth quotient of $\SL(2,p)$ when $p$ is odd. Next we will investigate the smooth covering of $\PSL(2,p)$ by considering the generating pairs of $\SL(2,p)$
Let $G=\SL(2,p)$ with $p$ an odd prime, and let
\[
b=
\begin{pmatrix}
1 & 0\\
1 & 1
\end{pmatrix},\
w=\begin{pmatrix}
1 & 1\\
0 & 1
\end{pmatrix}.
\]
Then $|b|=|w|=p$, and $\l b, w\r=G$.
%Obviously, $\PSL(2,p)$ is a quasi-smooth quotient of $G=\SL(2,p)$ since the center $\Z(G)=\ZZ_2$ and $p$ is odd.
Let $\overline{G}=G/\Z(G)$ and let $\overline{g}$ be images of $g\in G$ in $\overline{G}$.
The following theorem shows the existence of regular dessins $\calD(G,b,w^i)$ which are smooth coverings of $\calD(\ov G,\ov b,\ov w^i)$.

\begin{theorem}\label{Fibonacci}
Let $G=\SL(2,p)$ and $\ov G=\PSL(2,p)$ with $p\ge5$ prime, and let $b,w$ be as above.
Then there are exactly $\frac{(p+1)_{2'}+(p-1)_{2'}}{2}-1$ different values of $i$ such that $\calD(G,b,w^i)$ is a smooth covering of $\calD(\ov G,\ov b,\ov w^i)$.
\end{theorem}

However, it is unknown which values of $i$ are such that $\calD(G,b,w^i)$ is a smooth covering of $\calD(\ov G,\ov b,\ov w^i)$.
We next consider the case that $i=1$ as an example.

\begin{example}\label{Fibonacci-1}
Let $G=\SL(2,p)$ with $p$ prime, and $b,w\in G$ be defined above.
\begin{itemize}
\item[(1)] If $p\equiv 11$ or $19$ $(\mod 20)$, then $\calD(G,b,w)$ is a smooth covering of $\calD(\ov G,\ov b,\ov w)$.
\item[(2)] If $p\equiv2$ or $3$ $(\mod 5)$, then $\calD(G,b,w)$ is not a smooth covering of $\calD(\ov G,\ov b,\ov w)$.
\item[(3)] For the case $p\equiv 1$ $(\mod 20)$, we only have the following conclusion obtained by computation in Magma~\cite{magma}:
\begin{itemize}
    \item[(i)] if $p\in\{101,181,461,521,541,941\}$, then $\calD(G,b,w)$ is a smooth covering of $\calD(\ov G,\ov b,\ov w)$;
    \item[(ii)] if $p\in\{41,61,241,281,401,421,601,641,661,701,761,821,881\}$, then $\calD(G,b,w)$ is not a smooth covering of $\calD(\ov G,\ov b,\ov w)$.
\end{itemize}
\end{itemize}
\end{example}

This paper is organized as follows.
An explicit construction of regular dessins $\calD(G,b,w)$ is given in Section~\ref{coset regular dessins-sec}, and then the quotient and ramification phenomenon under this construction are studied in Section~\ref{quotient-sec}.
In Section~\ref{qp-sec}, a classification is obtained for the face quasiprimitive case. From Section~\ref{abel-sec} to Section~\ref{Schur-sec}, some special cases of regular dessins and their coverings are considered.

\section{Coset regular dessins}\label{coset regular dessins-sec}

It is well-known that a regular dessin is uniquely determined by its automorphism group and a pair of generators.
This is easily proved in the theory of monodromy groups.
In this section, we explain this important fact in a combinatorial way, see Theorem~\ref{G2}.

Let $G$ be a group such that  $G=\l b,w\r$.
Let $\Ga$ be the graph with vertex set $V$ and edge set $E$, where
\[\left\{
\begin{array}{l}
V=[G:\l b\r]\cup[G:\l w\r],\\
E=G,
\end{array}\right.
\]
such that the end vertices of an `edge' $g\in G$ are $\l b\r g$ and $\l w\r g$.
This graph is denoted by $\BiCos(G,\l b\r,\l w\r)$,  and called a {\it bi-coset graph}.
We remark that, in Dessin Theory, vertices in the set $[G:\l b\r]$ are colored {\bf b}lack, and vertices in the set $[G:\l w\r]$ are colored by {\bf w}hite.
We make the following observations about $\Ga=\Cos(G,\l b\r,\l w\r)$:
\begin{itemize}
\item[(i)] Two vertices $\l b\r x$ and $\l w\r y$ are adjacent if and only if $xy^{-1} \in\l b\r\l w\r$.

\item[(ii)] $\Ga$ may be not a simple graph, and the edge multiplicity of $\Ga$ equals $|\l b\r\cap\l w\r|$.

\item[(iii)] If $b=1$ or $w=1$, then $\Ga$ is a star $\K_{1,|G|}$.

\item[(iv)] $\Ga$ is a complete bipartite multigraph if and only if $G=\l b\r\l w\r$.

\item[(v)] For any element $g\in G$, the right multiplication of $g$ on set $V\cup E$ is an automorphism of $\Ga$. Moreover, the group $G$ can be viewed naturally as an automorphism subgroup of $\Ga$ in this way.
\end{itemize}

A directed walk of length $\ell$ in a graph $\Ga$ is an alternating sequence of vertices and edges, say $(v_0, e_1,v_1,e_2, v_2,\dots, e_\ell, v_\ell)$, such that  $v_{i-1}, v_i$ are the two endvertices of $e_i$ for $1\le i\le \ell$.  The vertices $v_0, v_\ell$ are called the endvertices of this directed walk. A directed walk of length $1$ is also called an arc of $\Ga$.
% An arc is said to lie on a directed walk if it is a subsequence
A directed cycle in a graph $\Ga$ is a directed walk $(v_0, e_1,v_1,e_2, v_2,\dots, e_\ell, v_\ell)$ with $v_0=v_\ell$.
More precisely, a directed cycle is a closed directed walk without distinguished endvertices,
that is, the sequences $(v_0, e_1,v_1,e_2, v_2,\dots, e_\ell, v_\ell)$ and $(v_1,e_2, v_2,\dots, e_\ell, v_\ell, e_1,v_1)$ represent the same directed cycle.
For a directed cycle, we may omit the vertices in these sequences if there is no ambiguity.

For a given $2$-generated group $G$ and one of its ordered generating pairs $(b,w)$, define {\it the boundary cycle $C(b,w)$ generated by $(b,w)$} in the bi-coset graph $\BiCos(G, \l b\r, \l w\r)$ to be the directed cycle.
% \begin{align}\label{cycle-0}
% C(b,w)=&(\l w\r, 1, \l b\r, b^{-1}, \l w\r b^{-1}, \dots, \\
% &\l w\r (bw)^{-i}, (bw)^{-i}, \l b\r (bw)^{-i}, b^{-1}(bw)^{-i}, \l w\r (bw)^{-i-1}, \dots, \\
% &\l w\r  ),
% \end{align}
% or simply
\begin{eqnarray}\label{cycle-1}
C(b,w)=( 1,  b^{-1}, \dots,  (bw)^{-i},  b^{-1}(bw)^{-i}, \dots, (bw)^{-(\ell-1)},  b^{-1}(bw)^{-(\ell-1)}=w ),
\end{eqnarray}
or more precisely, with the incident vertices included,
\begin{eqnarray}\label{cycle-1-precise}
C(b,w)=(\l w\r,  1, \l b\r,  b^{-1},\l w\r b^{-1}, \dots,\l w\r(bw)^{-i},(bw)^{-i},\l b\r(bw)^{-i}  ,\dots, \l b\r w  , w , \l w\r),
\end{eqnarray}
 where $\ell=|bw|$.
 % See Figure~\ref{local of c(bw)} for a part of this cycle.
 Notice that the cycle $C(b,w)$ is obtained by spinning the $2$-arc $(1,b^{-1})$ by the cyclic group $\l bw\r$.
Since the right multiplication of the group $G$ induced an automorphism subgroups of $\Ga$,
for any element $g\in G$, the image $C(b,w)g$  of $C(b,w)$ under $g$ is also a directed cycle of $\Ga$, which is
\[( g,  b^{-1}g, \dots,  (bw)^{-i}g,  b^{-1}(bw)^{-i}g, \dots, bwg, w g).\]

\begin{lemma}\label{face-stabilizer}
    For any two elements $g_1, g_2\in G$,  two cycles $C(b,w)g_1$ and $C(b,w)g_2$ are identical if and only if $g_1g_2^{-1}\in \l bw\r$.
    %, and each cycle $C(b,w)g$ is uniquely determined by the coset $\l bw\r g$.
\end{lemma}

\begin{proof}
    Obviously, $C(b,w)g_1=C(b,w)g_2$ is equivalent to $C(b,w)g_1g_2^{-1}=C(b,w)$.
    We only need to show that $C(b,w)g=C(b,w)$ if and only if $g\in \l bw\r$.

    Note that $(\l w\r,  1, \l b\r)$ is an arc started from a white vertex in $C(b,w)$.
    Hence, the arc $(\l w\r,  1, \l b\r)^g=(\l w\r g,g,\l b\r g)$ is started from a white vertex and lies on $C(b,w)^g=C(b,w)$.
    By the definition of $C(b,w)$, the arcs started from white vertices are
    \[(\l w\r (bw)^{-i},(bw)^{-i},\l b\r (bw)^{-i})\mbox{ for }i=0,1,...,\ell-1.\]
    This implies $g\in \l bw\r$.
    Therefore, $C(b,w)g_1=C(b,w)g_2$ if and only if $g_1g_2^{-1}\in \l bw\r$.
 \end{proof}

\vskip0.1in

Let $e$ be an edge of a regular dessin $\calD$ with two ends: the black vertex $\b$ and the white vertex  $\o$.
Then the automorphism group $G=\Aut\calD$ is generated by two elements $b,w$ such that
 $G_\b=\l b\r$ and $G_\o=\l w\r$.
In the following, we give an incidence configuration $(V,E,F)$ to identify this regular dessin.
% analog to the procedure on bi-rotary maps in~\cite[Construction 4.4]{rotary}.

% The dessin may be represented as an incidence configuration $(V,E,F)$ defined as follows.

% \begin{definition}\label{def-maps}
% {\rm
% Given an abstract group $G=\l b,w\r$, define a configuration $(V,E,F)$:
% \[\left\{
% \begin{array}{l}
% V=[G:\l b\r]\cup[G:\l w\r],\\
% E=G,\\
% F=\{C(b,w)g \mid g\in G\},
% \end{array}\right.\]
% {\color{blue} where the incidence relation between elements in $V\cup E$ gives rise to the bi-coset graph $\Cos(G,\l b\r, \l w\r)$, and the incidence relation between elements in $F$ and elements in  $V\cup E$ is given by the incident relation of the corresponding directed cycles with the vertices and edges in $\Cos(G,\l b\r, \l w\r)$.
% }
% % such that an `edge' $g\in E$ is incident with two `vertices' $\l b\r g$ and $\l w\r g$, and with
% % two `faces' $(\l bw\r\cup b^{-1}\l bw\r)g$ and $(\l bw\r\cup b^{-1}\l bw\r)w^{-1}g$.
% This incidence configuration is denoted by $\calD(G,b,w)$.
% {\color{blue}
% For each cycle $f\in F$, let $\hat{f}$ be a closed Euclidean disc
% with the boundary identified with the cycle $f$.
% Gluing all of these discs along common edges on their boundaries produces a topological space $\SS$. The graph $\Cos(G,\l b\r, \l w\r)$ is the skeleton of this topological space $\SS$. Hence, this incidence configuration $\calD(G,b,w)$ gives an embedding of $\Cos(G,\l b\r, \l w\r)$ on $\SS$.
% }
% }
% \end{definition}

\begin{definition}\label{def-maps}
{\rm
Given an abstract group $G=\l b,w\r$, define a configuration $(V,E,F)$:
\[\left\{
\begin{array}{l}
V=[G:\l b\r]\cup[G:\l w\r],\\
E=G,\\
F=\{C(b,w)g \mid g\in G\},
\end{array}\right.\]
such that an `edge' $g\in E$ is incident with two `vertices' $\l b\r g$ and $\l w\r g$, and with two `faces' $C(b,w)g$ and $C(b,w)w^{-1}g$.
% {\color{blue} where the incidence relation between elements in $V\cup E$ give rise to the bi-coset graph $\Cos(G,\l b\r, \l w\r)$, and the incidence relation between elements in $F$ and elements in  $V\cup E$ is given by the incident relation of the corresponding directed cycles with the vertices and edges in $\Cos(G,\l b\r, \l w\r)$.
% }
%
This incidence configuration is denoted by $\calD(G,b,w)$.}
\end{definition}

Since $G$ is a group of automorphisms of $\Cos(G,\l b\r,\l w\r)$,
an element $g\in G$ maps the cycle $C(b,w)$ to another cycle in $F$ by right multiplication. By the setting of the set $F$, we have the following lemma immediately.

\begin{lemma}\label{F-trans}
The group $G$ acts on the set $F$ transitively by right multiplication.
\end{lemma}

As $G$ also acting transitively on the edge set $E$, each edge must lie on some directed cycles in $F$.
The following lemma give a more clear relation between arcs with cycles in $F$.

\begin{lemma}\label{arc-unique}
Each arc of $\calD=\calD(G,b,w)$ lies on a unique cycle in $F$, and each cycle contains no repeated arcs.
Moreover, for $C=C(b,w)$ and $g\in G$, the arc $(\l w\r g,g,\l b\r g)$ lies on $Cg$, and the paired arc $(\l b\r g,g,\l w\r g)$ lies on $Cbg$.
\end{lemma}

\begin{proof}
As $(\l w\r, 1, \l b\r)$ lies on the directed cycle $C$, the arc $(\l w\r g, g, \l b\r g)$ lies on the cycle $Cg$, and the paired arc $(\l b\r g,g,\l w\r g)$ lies on $Cbg$.

Suppose that the arc $(\l w\r g, g, \l b\r g)$ also lies on $Cg_1$.
Then the arc $(\l w\r gg_1^{-1}, gg_1^{-1}, \l b\r gg_1^{-1})$ lies on $C$.
By the definition of the boundary cycle $C$, the arcs in $C$ with orientation from white vertex to black vertex are precisely those arcs $(\l w\r h, h, \l b\r h)$, where $h\in \l bw\r$.
This gives $gg_1^{-1}\in \l bw\r$, and so $Cg=Cg_1$ by Lemma~\ref{face-stabilizer}. Thus, $Cg$ is the only cycle in $F$ which contains the arc $(\l w\r g, g, \l b\r g)$.
By the same argument, the only cycle in $F$ which contains the arc $(\l b\r g, g, \l w\r g)$ is $Cbg$.
Thus, each arc lies on a unique cycle.

Suppose that a cycle $C'$ in $F$ contains a repeated arc $\alpha$.
Since $G$ is transitive on edges of $\calD$, there exists $g\in G$ such that $\alpha^g=(\l w\r, 1, \l b\r)$ or $(\l b\r, 1, \l w\r)$.
Without loss of generality, we assume that $\alpha^g=(\l w\r, 1, \l b\r)$.
Then $C=(C')^g$ as both $C$ and $(C')^g$ contain the arc $\alpha^g=(\l w\r, 1, \l b\r)$,
which yields that the arc $(\l w\r, 1, \l b\r)$ repeats on the cycle $C$, which contradict with the definition of $C$.
Therefore, every cycle contains no repeated arcs.
\end{proof}

Note that each edge contains exactly two arcs, and the above lemma deduces that every edge repeated at most twice in a cycle.
The following lemma shows that if an edge repeats in a cycle, then the dessin is unicellular.

\begin{lemma}\label{1-cycle}
Let $C=C(b,w)$.
Then the following statements are equivalent:
\begin{itemize}
\item[(i)] The edge $1$ appears at least two times on the cycle $C$.

\item[(ii)] $G=\l bw\r\cong\ZZ_\ell$.

\item[(iii)] Every edge of $C$ appears exactly two times on the cycle $C$.

\item[(iv)] $Cb=C$.

\item[(v)] $|F|=1$.
\end{itemize}

\end{lemma}
\proof
Assume first that statement (i) holds. By Lemma~\ref{arc-unique}, both the arcs  $(\l w\r,1,\l b\r)$ and $(\l b\r, 1, \l w\r)$ appear on the cycle $C$.
Then  we have that $1=b^{-1}(bw)^{-i}$ for some $0\le i \le \ell -1$.
Thus $b=(bw)^{-i}$, and $w=b^{-1}(bw)=(bw)^{i+1}$.
Therefore, $b,w\in\l bw\r$, and $G=\l bw\r$ is cyclic, as in part (ii).
So (i) implies (ii).

Suppose that part~(ii) holds, namely, $G=\l bw\r$ is cyclic.
Then $b=(bw)^{-j}$ for some integer $j$ with $0\le j\le \ell-1$.
Thus $1=b^{-1}(bw)^{-j}$ and $b^{-1}=(bw)^{\ell-j}$.
It follows that both the edges $b^{-1}$ and 1 appear exactly twice on $C$.
By the definition of $C$, each edge on $C$ appears exactly twice.
So (ii) implies (iii).

It is obvious that part~(iii) implies part~(i).

Part~(ii) implies part~(iv) since $G=\l bw\r$ implies $b=(bw)^j$ for some integer $j$ and so $Cb=C(bw)^j=C$.

Assume that $Cb=C$, as in part~(iv).
By Lemma~\ref{face-stabilizer},  we have $b\in \l bw\r,$
and so $w=b^{-1}(bw)\in \l bw\r.$
Therefore, $G=\l b,w\r=\l bw\r,$ and part~(iv) implies part~(ii).

Part~(v) implies part~(iv) since there is only one cycle $C$ which is the boundary cycle of the unique face.

Finally, assume that part~(iv) holds. Then part (ii) hods, and so
$Cg=C$ for any $g\in G$ because $g=(bw)^k$ for some integer $k$.
Since $G$ is transitive on the set of cycles, it follows that $C$ is the unique cycle and so $|F|=1$, as in part~(v).
\qed

A collection $\calC$ of cycles of a graph $\Ga$ is called a {\it cycle-double-covering} if
each edge of the graph $\Ga$ appears on a cycle at most once and lies on exactly two cycles in $\calC$.
The case where $|F|=1$ gives rise to regular dessins with a single face, which is characterised in
\cite{1-face}.
Thus we next assume that $F$ contains at least two cycles, and so $|b|>1, |w|>1.$

\begin{lemma}\label{reg-dessin}
The incidence triple $\calM(G,b,w)$ defined in Definition~$\ref{def-maps}$ is a regular dessin.
\end{lemma}
\proof
By Definition~\ref{def-maps}, $\calD(G,b,w)$ gives an embedding of the graph $\Cos(G, \l b\r, \l w\r)$ on the topological space $\SS$.
We need to prove that $\SS$ is indeed an orientable closed surface, that is, a sufficiently small neighborhood of each point in $\SS$ is an open disc and $\SS$ is orientable.

First, for any $f\in F$, each interior point of $\widehat{f}$ clearly has open disc neighborhoods.
Then, by Lemma~\ref{arc-unique}, each edge lies on exactly two cycles, and
so each interior point of an edge in $\SS$ is contained in a larger disc as an interior point.

We finally look at the point corresponding to the vertex $\l b\r$.
Let $|b|=m$, $C=C(b,w)$, and let $C_i=Cb^i$, where $0\le i\le m-1$.
Then the $\widehat C_i$ are all the discs
which are incident with  $\l b\r$.
Further, $\widehat C_i$ and $\widehat C_{i+1}$ share a unique common edge $b^{i}$, and
as $F$ is a cycle-double-covering of $(V,E)$, the discs $\widehat C_0,\widehat C_1,\dots,\widehat C_i,\dots,\widehat C_{m-1}$  glued together gives rise to a larger disc $D=\widehat C_0\cup\widehat C_1\cup\dots\cup\widehat C_{m-1}$
which contains the vertex $\l b\r$ as an interior point.
Similarly, the vertex $\l w\r$ is an interior point of some disc in $\SS$.

Since $G$ is transitive on the sets $[G:\l b\r]$, $[G:\l w\r]$, $E$ and $F$,
it follows that each point of $\SS$ has a neighbourhood which is  an open disc in $\SS$.
Therefore, $\SS$ is a surface.
 Moreover, for each disc $\hat{f}$, define a local orientation by the orientation of the directed cycle $f$. As each arc appears exactly once in $F$, these local orientations are compatible with each other. Hence $\SS$ is orientable.

Now the surface $\SS$ minus the edge set $E$ becomes a collection of open discs:
$\widehat{Cg}\setminus Cg$ with $g\in G$, and so the graph $\Ga$ is (2-cell) embedded in $\SS$.
By definition, $\calM(G,b,w)$ is a dessin, and $G$ preserves the orientation of~$\SS$.
As $G$ is regular on the edge set, $\calM(G,b,w)$ is a regular dessin.
\qed

The next lemma shows that any regular dessin has the form $\calM(G,b,w)$.

\begin{lemma}\label{G2}
Let $\calM$ be a regular dessin, and let $G=\Aut\calM$.
Then $\calM=\calM(G,b,w)$ for some elements $b,w\in G$.
\end{lemma}
\proof
Since $G$ is regular on the edges, we identify the edges of $\calD$ with the elements of $G$
such that $G$ acts on the edge set by right multiplication.
Let $e$ be the edge represented by the identity  $1\in G$ with black end $\beta$ and white end $\omega$, and
let $f,f'$ be the two faces that are incident with the edge $e$.

Let $(e_1,\o,e,\b,e_2)$ be the directed $3$-walk on the boundary of $f$, and
$(e_1',\o,e,\b,e_2')$ be the directed $3$-walk on the boundary of $f'$.
Since $G_\o$ is regular on $E(\o)$ and preserves the supporting surface,
there exists an element $w\in G_\o$ which sends $e$ to $e_1$, and sends $e_1'$ to $e$.
Similarly, there exists an element $b\in G_\b$ which sends $e$ to $e_2'$, and sends $e_2$ to $e$.
Thus, $e_1=w$, $e_1'=w^{-1}$, $e_2=b^{-1}$ and $e_2'=b$.
Noticing that the rotation $w$ sends the face $f'$ to $f$ and the rotation $b$ sends the face $f$ to $f'$, we conclude that the element $bw$ fixes the face $f$ and so does the subgroup $\l bw\r$.
It follows that all the images of $(e_1,\o,e,\b,e_2)=(w, 1,b^{-1})$ lie on the boundary of $f$. This gives the boundary cycle of $f$:
\[C:=(1,b^{-1},\dots,  (bw)^{-i},b^{-1}(bw)^{-i},\dots,(bw)^{-(\ell-1)}), b^{-1}(bw)^{-(\ell-1)}),\]
% \[C:=(b^{-1},1,w,bw,\dots, b^{-1}(bw)^i,(bw)^i,w(bw)^i,\dots,(bw)^{\ell-1})\]
where $\ell=|bw|$.
Since $G_\o=\l w\r$ is transitive on $E(\o)$, it is transitive on the faces of $\calM$ that are incident with $\o$.
Thus $G$ is transitive on the face set $F$, and so $F=\{Cg\mid g\in G\}$.
Therefore, $\calD=\calD(G, b,w)$.
\qed

Let $G_1=\l b_1,w_1\r$ and $G_2=\l b_2,w_2\r$.
If there exists a group isomorphism $\sigma:G_1\rightarrow G_2$ such that $(b_1,w_1)^\sigma=(b_2,w_2)$, then
$\sigma$ clearly induces an isomorphism between $\calD(G_1, b_1, w_1)$ and $\calD(G_2, b_2, w_2)$.
The converse part is a well-known result by considering the $2$-generated Free group (see \cite{Jones2014}).
We give a simple proof under our terminology of incidence configuration for this link between isomorphisms of groups and dessins.

\begin{lemma}\label{iso}
Two regular dessins $\calD_i=\calD(G_i, b_i, w_i)$ $(i=1,2)$ are isomorphic if and only if there is a group automorphism $\sigma$ from $G_1$ to $G_2$ such that $(b_1, w_1)^\sigma=(b_2,w_2)$.
\end{lemma}

\begin{proof}
    The sufficiency is clear. Note that $G_i\cong\Aut\calD_i$ by Lemma~\ref{G2}. An isomorphism from $\calD_1$ to $\calD_2$ induces a group isomorphism $\sigma:G_1\rightarrow G_2$.
    As $G_i$ acts regularly on the edge set of $\calD_i$ and $\sigma$ is color-preserving, we may assume that $\sigma$ maps the arc $(\l b_1\r, 1_{G_1}, \l w_1\r)$ to the arc $(\l b_2\r, 1_{G_2}, \l w_2\r)$. By Lemma~\ref{arc-unique}, the  boundary cycles contain these two arcs are $C(b_1,w_1)$ and $C(b_2,w_2)$, respectively. Hence $C(b_1,w_1)^\sigma=C(b_2,w_2)$. This follows
    \[(\l b_1\r w_1, w_1, \l w_1\r, 1, \l b_1\r, b_1^{-1},\l w_1\r b_1^{-1})^\sigma=(\l b_2\r w_2, w_2, \l w_2\r, 1, \l b_2\r, b_2^{-1},\l w_2\r b_2^{-1}).\]
    That gives $(b_1,w_1)^\sigma=(b_2,w_2)$.
\end{proof}

\section{Quotients and coverings: smoothness and ramification}
\label{quotient-sec}

In this section, we will discuss the quotients and coverings of dessins from two different perspectives: geometric and algebraic. Naturally, there arises a phenomenon of ramification during the discussion. Through the exploration of this phenomenon, we got a connection between the Euler characteristics of dessins and their coverings.

Let $\calD=\calD(G,b,w)=(V,E,F)$ be a $G$-regular dessin, and
let $C=C(b,w)$, as defined in (\ref{cycle-1}).
Let $N\lhd G$, and $\ov G=G/N$. We discuss the quotients of the dessin $\calD(G,b,w)$ with respect to the normal subgroup $N$ in the following subsections.

\subsection{Geometric quotient}\label{subsec:geo-quo}\

The `geometric quotient' of $\mathcal{D}$ induced by $N$, denoted as $\mathcal{D}_N$, is obtained by contracting $N$-orbits.
As $N\lhd G$ preserves the sets $V,E,F$, we can consider the actions on these sets separately.
Let $V_N,E_N,F_N$ be the sets of $N$-orbits on $V,E,F$, respectively.
 Moreover, there is a natural projection
 \def\phigeo{\varphi_{\mbox{\scalebox{.4}{Geo}}}}
  \def\phialg{\varphi_{\mbox{\scalebox{.4}{Alg}}}}

 \[\phigeo:\ \ x\mapsto x^N,\ \mbox{where $x\in V\cup E\cup F$}.\]
 Then the quotient $\calD_N$ is just the incidence triple $(V_N,E_N,F_N)$ which gives rise to a quotient dessin. We give a specific description of this quotient dessin in the following.

 Let $\beta$ denote the black vertex corresponding to the coset $\langle b \rangle$ in $\mathcal{D}$. The orbit of $\beta$ under $N$ is $\beta^N = \{ \langle b \rangle x | x \in N \} = \{ \langle b \rangle x | x \in \langle b \rangle N \}$. Consequently, the number of preimages of $\varphi_{\text{geo}}(\beta)$ is $\frac{|\langle b \rangle N|}{|\langle b \rangle|} = \frac{|N|}{|\langle b \rangle \cap N|}$. This holds true for other black vertices in $\mathcal{D}$ as well.
Similarly, for the white vertex $\omega$ corresponding to the coset $\langle w \rangle$ in $\mathcal{D}$, the number of preimages of $\varphi_{\text{geo}}(\omega)$ is $\frac{|\langle w \rangle N|}{|\langle w \rangle|} = \frac{|N|}{|\langle w \rangle \cap N|}$, and the same applies to other white vertices.
For any edge $e$ of $\mathcal{D}$, which is also an element of the group $G$, the orbit $e^N$ has cardinality $|N{:}N_ e|=|N|$. Thus, the number of preimages of $\varphi_{\text{geo}}(e)$ is $|N|$.

The contraction of a face orbit is somewhat complex and requires a bit more caution.
Consider the face orbit $C^N=C(b,w)^N=\{C(b,w)x| x\in N\}$. Since the stabilizer of $C$ is $\l bw\r$, the number of preimages of $\varphi_{\text{geo}}(C)$ is $\frac{|N|}{|\langle bw \rangle \cap N|}$.
Note that the intersection $N\cap\l bw\r$ fixes the cycle $C$. Let $|bw|=\ell$ and $|N\cap\l bw\r|=m$.
Then $N\cap\l bw\r$ wraps the cycle $C$ by $m$ times, contracting $C$ into a cycle of length $2\ell/m$.
Similarly, each cycle in the orbit $C^N$ becomes a cycle of length $2\ell/m$.
Then contract the $\frac{|N|}{|N\cap\l bw\r|}$ cycles of length $2\ell/m$ into a cycle of the same length.
Refer to the demonstration shown in Figure~\ref{contracting-face}.
 By the same way, for any $g\in G$, the face orbit $C^{gN}=C^{Ng}=(C^N)^g$ i{}s also viewed as the contraction of $2\ell/m$ cycles in $F$.
This gives the directed cycles $C^{Ng}\in F_N$.
So we obtain an incidence triple $(V_N,E_N,F_N)$, and it will be shown to be a regular dessin by Theorem~\ref{quotient},
which could be regarded as a `geometric quotient' of $\calM$.

 \begin{figure}
     \centering
     \begin{tikzpicture}[scale=0.5]
      \coordinate (v1) at (3.4cm, 0);
      \coordinate (v2) at (0, 1cm);
      \coordinate (v4) at (0, -1cm);
      \coordinate (v3) at (-3.4cm, 0);
      \coordinate (v5) at (-2cm,-1.5cm);
      \coordinate (v0) at (0,-1.4cm);
      \coordinate (v6) at ($(v4)+(v0)$);
      \coordinate (v7) at ($(v1)+(v0)$);
      \coordinate (v8) at ($(v2)+(v0)$);
      \coordinate (v9) at ($(v3)+(v0)$);
      \coordinate (v10) at (-0.3cm,-0.4cm);
      \coordinate (v11) at (-10cm,-0.7cm);

      \foreach \j in {0cm, 7cm, 13cm} {
          \filldraw[fill =gray!20, line width=1pt] ($(v11)+(0, \j)$) circle (2.5cm);
          \foreach \i in {1,..., 4} {
          \filldraw[fill=white] ($(v11)+(\i*90: 2.5cm)+(0, \j)$) circle (4pt);
          \filldraw[fill=black] ($(v11)+(\i*90+45: 2.5cm)+(0, \j)$) circle (4pt);
          };

          \fill[fill= gray!20] ($(v5)+(0, \j)$) .. controls +(-55: 1cm) and +(180: 0.4cm) .. ($(v6)+(0, \j)$) .. controls +(0: 2cm) and +(-90: 0.6cm) .. ($(v7)+(0, \j)$) .. controls +(90: 0.6cm) and +(0: 2cm) .. ($(v8)+(0, \j)$) .. controls +(180: 2cm) and +(90: 0.6cm) ..  ($(v9)+(0, \j)$) .. controls +(-90: 0.4cm) and +(215: 1cm) ..($(v5)+(0, \j)$) ;
          \shade [bottom color=gray, top color=gray!0, opacity=0.6, shading angle=-170] ($(v5)+(0, \j)$) .. controls +(-55:1cm) and +(180: 0.4cm) .. ($(v6)+(0, \j)$) .. controls +(0: 2cm) and +(-70: 0.2cm) ..  ($(v4)+(0, \j)$) .. controls +(190: 0.2cm) and +(40:1cm) .. ($(v5)+(0, \j)$);
            \draw[line width=1pt] ($(v5)+(0, \j)$) .. controls +(-55: 1cm) and +(180: 0.4cm) .. ($(v6)+(0, \j)$) .. controls +(0: 2cm) and +(-90: 0.6cm) .. ($(v7)+(0, \j)$) .. controls +(90: 0.6cm) and +(0: 2cm) .. ($(v8)+(0, \j)$) .. controls +(180: 2cm) and +(90: 0.6cm) ..  ($(v9)+(0, \j)$) .. controls +(-90: 0.4cm) and +(215: 1cm) ..($(v5)+(0, \j)$) ;
          \fill[fill=gray!20] ($(v5)+(0, \j)$) .. controls +(40: 1cm) and +(190: 0.2cm) .. ($(v4)+(0, \j)$) .. controls +(10: 2cm) and +(-90: 0.6cm) .. ($(v1)+(0, \j)$) .. controls +(90: 0.6cm) and +(0: 2cm) .. ($(v2)+(0, \j)$) .. controls +(180: 2cm) and +(90: 0.6cm) ..  ($(v3)+(0, \j)$) .. controls +(-90: 0.4cm) and +(125: 1cm) ..($(v5)+(0, \j)$) ;
          \shade [bottom color=gray, top color=gray!0, opacity=0.6, shading angle=10] ($(v5)+(0, \j)$) .. controls +(40:1cm) and +(-150: 1cm) .. ($(v10)+(0, \j)$) .. controls +(125: 2cm) and +(90: 0.6cm) ..  ($(v3)+(0, \j)$) .. controls +(-90: 0.4cm) and +(125:1cm) .. ($(v5)+(0, \j)$);
          \draw[line width=1pt] ($(v5)+(0, \j)$) .. controls +(40: 1cm) and +(190: 0.2cm) .. ($(v4)+(0, \j)$) .. controls +(10: 2cm) and +(-90: 0.6cm) .. ($(v1)+(0, \j)$) .. controls +(90: 0.6cm) and +(0: 2cm) .. ($(v2)+(0, \j)$) .. controls +(180: 2cm) and +(90: 0.6cm) ..  ($(v3)+(0, \j)$) .. controls +(-90: 0.4cm) and +(125: 1cm) ..($(v5)+(0, \j)$) ;

          \foreach \i in {1,3,7,9} {
          \filldraw[fill=black] ($(v\i)+(0, \j)$) circle (4pt);
          };
           \foreach \i in {2,4,6} {
          \filldraw[fill=white] ($(v\i)+(0, \j)$) circle (4pt);
          };

           \draw [line width=0.5pt, double distance=3pt,  arrows = {-Latex[length=0pt 3 0]}] ($(-7cm,-0.7cm)+(0, \j)$) -- ($(-4cm,-0.7cm)+(0, \j)$) ;

             \draw [line width=0.5pt, double distance=3pt,  arrows = {-Latex[length=0pt 3 0]}] ($(0,-0.7cm +\j)! .35!(11cm, 6.3cm) $) -- ($(0,-0.7cm+\j)! .7!(11cm, 6.3cm) $) ;
      };

      \node[anchor= south ] at (-5.5cm, 6.3cm) {Wraping};
      \node[anchor= south ] at (5.8cm, 6.3cm) {Contraction};
      \node[anchor= south ] at (11cm, 9cm) {$C^N$};

      \node at ($(v11)+(0, 3.5cm)$) {$\vdots$};
      \node at ($(0,-0.7cm)+(0, 3.5cm)$)  {$\vdots$};

      \filldraw[fill =gray!20, line width=1pt] (11cm, 6.3cm) circle (2.5cm);
       \foreach \i in {1,2} {
          \filldraw[fill=white] ($(11cm, 6.3cm)+(\i*180+90: 2.5cm)$) circle (4pt);
          \filldraw[fill=black] ($(11cm, 6.3cm)+(\i*180: 2.5cm)$) circle (4pt);
    };

    \node at (-15cm, -0.7cm) {$Cn_{\frac{|N|}{m}}$};
    \node at (-14.5cm, 6.3cm) {$Cn_2$};
    \node at (-14.5cm, 2.8cm) {$\vdots$};
    \node at (-14.5cm, 12.3cm) {$C=Cn_1$};
     \end{tikzpicture}
          \caption{Contraction of the face orbit $C^N$}
     \label{contracting-face}
     % \caption{Part of the boundary cycle $C(b,w)$}
     % \label{local of c(bw)}
 \end{figure}
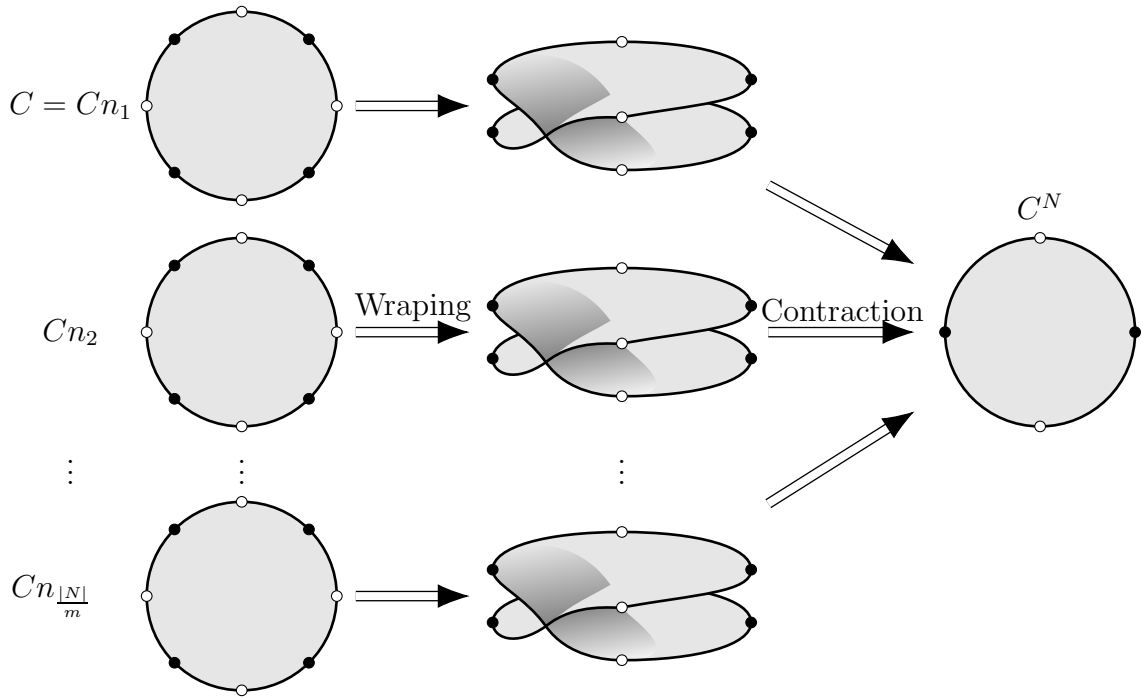

 % \begin{figure}
 %     \centering
 %     \includegraphics[width=12cm]{contracting-face.jpg}
 %     \caption{Contraction of the face orbit $C^N$}
 %     \label{contracting-face}
 % \end{figure}

\subsection{Algebraic quotient}

For each element $g\in G$, let $\ov g$ be the image of $g$ under $G\to\ov G=G/N$.
Then the quotient group $\ov G$ is generated by $\ov b$ and $\ov w$.
Let $\l bw \r\cap N=\l (bw)^k \r\cong \ZZ_m$ such that $km=|bw|=\ell.$ Then $C(\overline{b},\overline{w})$ is a cycle with length $2k$ as follows
\begin{eqnarray}\label{quo-cycle}
C(\ov b, \ov w)=(\overline{1},\overline{b}^{-1},\dots,({\ov b\ov w})^{-i},\overline{b}^{-1}({\ov b\ov w})^{-i}, \dots,({\ov b\ov w})^{-(k-1)}, \overline{w}).
\end{eqnarray}
Let
\[\calM/N=\calM(\ov G,\ov b,\ov w),\]
called a {\it normal quotient dessin} of $\calM$ induced by $N$,
which we regard as the `algebraic quotient' of $\calD$ and satisfies
\[\begin{array}{l}
V/N=[\ov G:\l\ov b\r]\cup[\ov G:\l\ov w\r],\\
E/N=[\ov G:\l\ov 1\r]
% =\{[\l\ov w\r\ov g,\ov g,\l\ov b\r\ov g]\mid \ov g\in G/N\}
,\\
F/N=\{C(\overline{b},\overline{w})\overline{g}\mid \ov g\in G/N\}.
\end{array}\]

The following theorem shows that `geometric quotient' and `algebraic quotient' are
the same.

\begin{theorem}\label{quotient}
The incidence triple $(V_N,E_N,F_N)$ is a regular dessin and
isomorphic to $\calD(\ov G,\ov b,\ov w)$.
\end{theorem}
\proof
We note that the action of $G$ on $\calD$ is by right multiplication.
%on $V,E$ and $F$.
For an element $g\in G$, the orbit of the vertex $\l w\r g$ under $N$ is equal to
\[(\l w\r g)^N=\{\l w\r gx\mid x\in N\}=(\l w\r N)(Ng)=\l \ov w\r\ov g;\]
similarly, the orbit of the vertex $\l b\r g$ under $N$ is $\l\ov b\r\ov g$.
Thus the vertex set $V_N$ equals $[\ov G:\l\ov w\r]\cup[\ov G:\l\ov b\r]$.

The orbit of the edge $g$ under $N$ is equal to $gN=\ov g\in \ov G$, and
the edge set $E_N=[\ov G:\ov 1]$.
Now we consider the orbit of the face $C(b,w)$ under $N$. Recall that
\[\begin{array}{l}
C(b,w)=( 1,  b^{-1}, \dots,  (bw)^{-i},  b^{-1}(bw)^{-i}, \dots, (bw)^{-(\ell-1)},  b^{-1}(bw)^{-(\ell-1)} ),
\end{array}\]
Let $\l bw \r\cap N=\l (bw)^k \r\cong \ZZ_m$ such that $km=|bw|=\ell.$
% In the cycle $C(b,w),$ the edges $b^{-1}, b^{-1}(bw)^k, b^{-1}(bw)^{2k},\cdots, b^{-1}(bw)^{(m-1)k}$ form an orbit of $b^{-1}$ under $N$,
% that is $b^{-1}N,$ which is the second edge in $C(\ov b, \ov w)$. Similarly, $\{1,(bw)^k, (bw)^{2k},\cdots, (bw)^{(m-1)k}\}=1N$ is the first edge in $C(\ov w, \ov b)$.
In the cycle $C(b,w),$ the edges $1,(bw)^k, (bw)^{2k},\cdots, (bw)^{(m-1)k}$ belong to an orbit of $1$ under $N$,
that is $1N,$ which is the first edge in $C(\ov b, \ov w)$. Similarly, $\{b^{-1}, b^{-1}(bw)^k, b^{-1}(bw)^{2k},\cdots, b^{-1}(bw)^{(m-1)k}\} \subset b^{-1}N$ is the second edge in $C(\ov b, \ov w)$.
And so on, $\{b^{-1}(bw)^{k-1}, b^{-1}(bw)^{2k-1},\cdots, b^{-1}(bw)^{mk-1}\}=b^{-1}(bw)^{k-1}N$ is the $2k$-th edge in $C(\ov b, \ov w)$.
Thus, the orbit of the face $C(b,w)$ under $N$ is $C(\ov b, \ov w)$. By the transitivity, the face set $F_N=\{C(\overline{b},\overline{w})\overline{g}\mid \ov g\in G/N\}.$
We therefore conclude that $(V_N,E_N,F_N)=\calD(\ov G,\ov b,\ov w)$, and
so $\calD_N=(V_N,E_N,F_N)$ is a regular dessin.
\qed

Notice that if $G=\l b,w\r$ then $\l b\r\cap\l w\r\lhd G$ and the edge-multiplicity of $\Cos(G,\l b\r,\l w\r)$ equals $|\l b\r\cap\l w\r|$, led to the following corollary.

\begin{corollary}\label{q-multi-edges}
Each regular dessin has a normal quotient which is a simple dessin.
%
%Let $\calM=\calM(G,b,w)$, and $N$ a normal subgroup of $G$.
%If $N=\l b\r\cap\l w\r$, then $\calM_N$ is a simple dessin.
\end{corollary}

%\proof {\color{red} The proof may be deleted.)}
%The quotient dessin $\MM_N=\MM(\ov{G},\ov b,\ov{w})$, where $\ov{G}=G/N,$
%and $\ov b, \ov{w}$ are the images of $b,w$ under $G\rightarrow \ov{G}=G/N$.
%If $N=\l b\r\cap\l w\r$, then \[\l \ov b\r \cap\l \ov{w}\r=(\l b\r N/N) \cap (\l w\r N/N)\cong(\l b\r \cap \l w\r) N/N=\ov{1}, \]
%and so $\BiCos(\ov G,\l\ov b\r,\l\ov w\r)$ is simple. \qed

\begin{example}\label{quaternion}
{\rm
Let $G=\Q_{4m}=\l x,y\r$, where $|x|=2m\geq 4$, $|y|=4$, $x^y=x^{-1}$ and $y^2=x^m$, a {\it generalized quaternion} group.
Let $b=xy$ and $w=y^{-1}$.
Then $|b|=|w|=4$, $G=\l b,w\r$, $\l b\r\cap\l w\r=\ZZ_2$ and $\l bw\r=\l x\r=\ZZ_{2m}$.
Let $\calD=\calD(G,b,w)$ and let
\[\begin{aligned}
    C&=(1,b^{-1},\dots, (bw)^{-i},b^{-1}(bw)^{-i},\cdots,(bw)^{-(2m-1)},b^{-1}(bw)^{-(2m-1)})\\
    &=(1,x^{m+1}y,\dots,x^{-i}, x^{m+1+i}y,\dots,x^{-(2m-1)}, x^{m}y ).
\end{aligned} \]
Then $\calD=(V,E,F)$, where $V=[G:\l b\r]\cup[G:\l w\r]$, $E=G$ and $F=\{C,C y\}$;
further,
\begin{itemize}
\item[(i)] $\calD$ has valency 4, and face length $2|bw|=4m$;
\item[(ii)] the underlying graph of $\calD$ is $\C_{2m}^{(2)}$, a cycle with edge multiplicity 2;
\item[(iii)] $\chi(\calD)=|V|-|E|+|F|=4m(\frac{2}{4}-1+\frac{1}{ 2m})=-2(m-1)$.
\end{itemize}

We consider the quotient of $\calD$ induced by the center $N=\Z(G)=\l y^2\r\cong\ZZ_2$.
First, notice that $G/N=\D_{2m}=\l\ov x\r{:}\l\ov y\r=\l\ov b,\ov w\r$ is a dihedral group where $|\ov x|=m$ and $|\ov y|=|\ov b|=|\ov w|=2$, and
the quotient graph of $\C_{2m}^{(2)}$ is a simple cycle $\C_{2m}$.
\begin{itemize}
\item[(a)] The geometric quotient $\calD_N=(V/N,E/N,F/N)$, with $V/N=[\ov G:\l\ov b\r]\cup[\ov G:\l\ov w\r]$, $E/N=\ov G\cong \D_{2m}$, and $F/N=\{\overline{C},\overline{C} \overline{y}\}$, where
\[ \overline{C}= (\overline{1},\overline{x}\overline{y},\dots,\overline{x}^{-i}, \overline{x}^{1+i}\overline{y},\dots,\overline{x}^{-(m-1)}, \overline{y} ).\]
\item[(b)] The algebraic quotient $\calD/N=\calD(\ov G,\ov b,\ov w)$, where $\ov G\cong \D_{2m}$ is generated by the involutions $\ov b$ and $\ov w$.
    Clearly, $\calD/N$ is an embedding of a simple cycle $\C_{2m}$ in a sphere.
\end{itemize}

Thus the double covering between $\calD$ and $\calD/N$ induces a covering between a surface with Euler characteristic $\chi(\calD)=-2(m-1)$ and a sphere.
% , and the covering has $|\ov G|(\frac{1}{2}+{\frac{1}{2}}+{\frac{1}{m}})=2m+2$ ramification points.
\qed
}
\end{example}

\subsection{Ramification points}\label{rami-sec} \

%To preserve the unity of the symbols, we denote $\calD_N$ and $\calD/N$ by $\ov{\calD}$ from now on.
Let $\varphi$ be the natural homomorphism from $G$ to $\ov G=G/N$:
\[\varphi:\ \ g\mapsto \ov g=gN,\ \ \mbox{where $g\in G$.}\]
Recall that the underlying graphs of $\calD(G,b,w)$ and $\calD(\ov G,\ov b,\ov w)$ are
$\Cos(G,\l b\r,\l w\r)$ and $\Cos(\ov G,\l\ov b\r,\l\ov w\r)$, respectively.
Then $\varphi$ induces a mapping from vertices and edges  of $\calM$ to the vertices and edges of $\calM_N$,
respectively, as below
\[\varphi:\ \ \ \begin{array}{rcl}
\l b\r g &\mapsto& \l\ov b\r\ov g,\\
\l w\r g &\mapsto& \l\ov w\r\ov g,\\
% \left[\l b\r g,g,\l w\r g\right] &\mapsto& \left[\l\ov b\r\ov g,\ov g,\l\ov w\r\ov g\right].
g &\mapsto& \overline{g}.
\end{array}\]

\begin{lemma}\label{graph-homo}
The mapping $\varphi$ is a graph homomorphism from $\Cos(G,b,w)$ to $\Cos(\ov G,\ov b,\ov w)$, and further, the following hold.
\begin{itemize}
\item[(1)] Each  vertex in $[\ov G:\l\ov b\r]$ has exactly $\frac{|N|}{|N\cap\l b\r|}$ preimages in $[G:\l b\r]$,
and each vertex in $[\ov G:\l\ov w\r]$ has exactly $\frac{|N|}{|N\cap\l w\r|}$ preimages in $[G:\l w\r]$.

\item[(2)] Each edge of $\Cos(\ov G,\l\ov b\r,\l\ov w\r)$ has exactly $|N|$ preimages in $\Cos(G,\l b\r,\l w\r)$.
\end{itemize}
\end{lemma}
\proof
{\rm (1).}
A vertex in $[\ov G:\l\ov b\r]$ has the form $\l\ov b\r\ov g=\l bN\r gN$, where $g\in G$.
Since $b^iNgN=b^igN$, we have  $\l bN\r gN=\{gN, bgN,\dots,b^igN,\dots\}=\l b\r gN$,
which is an orbit of the vertex $\l b\r g\in[G:\l b\r]$ under $N$ by right multiplication.
For $x_1,x_2\in N$, the vertices $\l b\r gx_1$ and $\l b\r gx_2$ are equal if and only if
$b^igx_1=b^jgx_2$ for some integers $i,j$, that is, $b^{j-i}=(x_1x_2^{-1})^{g^{-1}}\in\l b\r\cap N$.
Thus the cardinality $|\{\l b\r gx\mid x\in N\}|$ equals $\frac{|N|}{|N\cap \l b\r|}$.

Similarly, the vertex $\l \overline{w}\r \overline{g}=\l wN\r gN=\l w\r gN$ is an orbit of $\l w\r g$ under $N$, and
it follows that $|\{\l w\r gx\mid x\in N\}|=\frac{|N|}{|N\cap \l w\r|}$.

{\rm (2).} Let $\ov e = \overline{1_G}=1_{G/N}=N$ be the  edge between $\l\ov b\r$ and $\l\ov w\r$.
Then for any $g\in G$, the edge $g$ is a preimage of $\overline{e}$ if and only if
 $\varphi(g)=N=e$, which is equivalent to $g\in N$. Hence, the set of all preimages of $\overline{e}$  is $N$, which is also an orbit of $N$ on the edges in $E$.
Finally, since $G$ is regular on the edge set, $N$ is semiregular, and so $\ov e$ has exactly
$|N|$ preimages.
So is each edge in $E_N$ because $G/N$ is transitive on $E_N$.
\qed

Let $C=C(b,w)$ and $\ov C=C(\ov b,\ov w)$, as defined in (\ref{cycle-1}) and (\ref{quo-cycle}).
Let $\calD=\calM(G,b,w)$ and $\calD_N=\calM(\ov G,\ov b,\ov w)$.

\begin{lemma}\label{phi}
The graph homomorphism $\varphi$ induces a homomorphism from the dessin $\calD$ to the quotient dessin $\calD_N$, and further,
each face of $\calD_N$ has exactly $\frac{|N|}{|N\cap\l bw\r|}$ preimages, which form an orbit of $N$ acting on $F$.
\end{lemma}
\proof
Consider the two cycles
\[\begin{array}{l}
C=C(b,w)=( 1,  b^{-1}, \dots,  (bw)^{-i},  b^{-1}(bw)^{-i}, \dots, (bw)^{-(\ell-1)},  b^{-1}(bw)^{-(\ell-1)} ),\\

\ov C=C(\ov b, \ov w)=(\overline{1},\overline{b}^{-1},\dots,({\ov b\ov w})^{-i},\overline{b}^{-1}({\ov b\ov w})^{-i}, \dots,({\ov b\ov w})^{-(\overline{\ell}-1)}, \overline{w}),
\end{array}\]

 where $\ell=|bw|$, and $\overline{\ell}=|\overline{b}\overline{w}|$.
 Obviously, the image $\varphi(C)$ is $\ov C$.
 Let $m=|\l bw\r\cap N|$.
 Then $m$ is a divisor of $\ell$, and $(bw)^{\ell/m}\in N$.
Thus $(\ov b\ov w)^{\ell/m}=\ov 1$, and for  integers $j$ and $i$ with $0\leq j\leq m-1$,
\[\varphi:\
\begin{array}{rll}
b^{-1}(bw)^{j\frac{\ell}{m}+i} &\mapsto& \ov b^{-1}(\ov b\ov w)^i,\\
(bw)^{j\frac{\ell}{m}+i} &\mapsto& (\ov b\ov w)^i,\\
w(bw)^{j\frac{\ell}{m}+i} &\mapsto& \ov w(\ov b\ov w)^i.
\end{array}
\]

It follows that each arc on $\ov C$ has precisely $m$ preimages in $C$, namely,
the cycle $C$ is mapped to $\ov C$ by wrapping $m$ times.
% Furthermore, $\overline{C}$ possesses additional preimages.

Let $g\in G$ be such that $C^g$ is a preimage of $\overline{C}$ under the map $\varphi$.
Then $\ov C^{\ov g}=\ov C$, and so $\ov g\in \ov G_{\ov C}=\l\ov b\ov w\r=\l bwN\r$.
Thus $g\in\l bw\r N$, namely, $g=(bw)^ix$ for some element $x\in N$, and so  $C^g=C^x$.
The cardinality $|\{C^x\mid x\in N\}|=\frac{|N|}{|N\cap \l bw\r|}$, that is,
the cycle $\ov C$ has $\frac{|N|}{|N\cap\l bw\r|}$ preimages.
\qed

Next, we consider the relation between $\calM$ and $\calD_N$.
Let $\widehat C$ be a unit Euclidean disc
\[\widehat C=\{\rho\bfe^{\theta\bfi}\mid 0\le\rho\le 1,\ 0\le\theta<2\pi\}\]
with boundary cycle $C$, namely,
\[C=\{\bfe^{\theta\bfi}\mid 0\le\theta<2\pi\}.\]
Let $r=\ell/\ov\ell$, where $\ell=|bw|$ and $\overline{\ell}=|\overline{b}\overline{w}|$.
Define the following function
\[\Psi:\ \rho\bfe^{j\theta\bfi}\mapsto (\rho\bfe^{j\theta\bfi})^r,\ \ \mbox{where $0\le\rho\le 1$ and $0\le\theta<2\pi$.}\]

\begin{lemma}\label{C-cover-C}
The mapping $\Psi$ is a $|N\cap\l bw\r|$-sheeted covering of $\widehat{\ov C}$ by $\widehat C$,
with the origin being a ramification point if $|bw|\neq |\ov b\ov w|$.
Moreover, $\Psi$ and $\varphi$ coincide at vertices and edges on the cycle $C=C(b,w)$.
\end{lemma}
\proof
By definition, $\Psi$ is a $r$-to-1 mapping from $C\setminus\{0\}$ to $\ov C\setminus\{0\}$,
and fixes the origin $0$.
That is to say, $\Psi$ wraps the disc $C$ around the origin $r$ times to give rise to the disc $\ov C$.
Thus $C$ is an $r$-sheeted covering of $\ov C$,
where $r=\frac{|\l bw\r|}{|\l \overline{b}\ov w\r|}=\frac{|\l bw\r|}{|\l bw\r N/N|}=|\l bw\r\cap N|$,
with the origin being one ramification point if $|bw|\neq |\overline{b}\overline{w}|$.
\qed

\subsection{Euler characteristics}\

In this section, we investigate the relations of the Euler characteristics when taking normal quotient. Let $\calD=\calD(G, b,w)$ be a regular dessin, and let $N$ be a normal subgroup of $G$.

\begin{lemma}\label{cover-genus}
$\chi(\calD)\leq|N|\chi(\calD_N)$, and the equality holds if and only if $N$ is semiregular on
the vertex set and the face set of $\calD$.
\end{lemma}
\proof
The characteristic $\chi(\calD)=|V|-|E|+|F|=|G|({\frac{1}{|b|}}+{\frac{1}{|w|}}-1+{\frac{1}{|bw|}})$, and
\[\begin{array}{lll}
\chi(\calD_N) &=&|V_N|-|E_N|+|F_N|=|\ov G|({\frac{1}{|\ov b|}}+{\frac{1}{|\ov w|}}-1+{\frac{1}{|\ov b\ov w|}})\\
&\geq&{\frac{|G| }{|N| }}({\frac{1 }{ |b|}}+{\frac{1}{|w|}}-1+{\frac{1}{|bw|}})\\
&=& {\frac{1}{|N|}}\chi(\calD).
\end{array}\]
The equality holds if  and only if ${\frac{1}{|b|}}+{\frac{1}{|w|}}-1+{\frac{1}{|bw|}}={\frac{1}{|\ov b|}}+{\frac{1}{|\ov w|}}-1+{\frac{1}{|\ov b\ov w|}}$, and if and only if $|b|=|\ov b|$, $|w|=|\ov w|$ and $|bw|=|\ov b\ov w|$.
It follows that $\chi(\calD)=|N|\chi(\calD_N)$ if and only if $N$ is semiregular on the vertex set and the face set of $\calD$.
\qed

We may identify a dessin with the supporting surface so that the normal quotient of a map corresponds to surface covering.
The following theorem characterises branched points of a covering of regular dessins.

\begin{theorem}\label{quotient-dessin}
Let $N\lhd G=\l b,w\r$, and let $\ov G=G/N$, $\ov b=bN, \ov w=wN$.
Then $\calD=\calD(G,b,w)$ is a $|N|$-sheeted covering of $\calD_N=\calD(\ov G,\ov b,\ov w)$ with
$|\ov G|({\frac{i}{|\ov b|} }+{\frac{j}{|\ov w|} }+{\frac{k}{|\ov b\ov w|} })$
ramification points with $i,j,k\in\{0,1\}$ such that
\[\mbox{$i=0\Longleftrightarrow\l w\r\cap N=1$, and $j=0\Longleftrightarrow\l b\r\cap N=1$, and
$k=0\Longleftrightarrow\l bw\r\cap N=1$.}\]
Furthermore, $\chi(\calM)\leq|N|\chi(\calM_N)$, and the equality $\chi(\calM)=|N|\chi(\calM_N)$ holds if and only if
%The genera $g$ and $g_N$ of $\calD$ and $\calD_N$ satisfy $g-1\geq|N|(g_N-1)$, and further $g-1=|N|(g_N-1)$
$\calM$ is a smooth covering of $\calM_N$, and if and only if $N$ is semiregular on $V\cup F$.
\end{theorem}
\proof
We use the notation defined above.
Let $\varphi$ be the homomorphism from $\calM$ to $\calM_N$.
By Lemmas~\ref{phi}, each interior point on an edge of $\calM_N$ has exactly $|N|$ preimages.
By Lemmas~\ref{phi} and \ref{C-cover-C}, an interior point of a face $\ov C$ has exactly
$${\frac{|N|}{|N\cap\l bw\r|}}\cdot|N\cap\l bw\r|=|N|$$
preimages with one possible exception of a distinguished interior point.
Therefore, $\calD$ is a $|N|$-sheeted covering of $\calD_N$, and
the only possible ramification points are some vertices and some distinguished interior points of faces.
We next compute the ramification points.

The black vertex $\l \ov b\r$ of $\calM_N$ has precisely ${\frac{|N|}{|N\cap\l b\r|}}$ preimages $\{\l b\r x\mid x\in N\}$.
Thus $\l \ov b\r$ is not a ramification point if and only if ${\frac{|N|}{|N\cap\l b\r|}}=|N|$,
which is equivalent to $|N\cap\l b\r|=1$, namely, $N$ is semiregular on $[G:\l b\r]$.
By the transitivity, either no vertex in $[\ov G:\l\ov b\r]$ is ramification point, or
all of the vertices in $[\ov G:\l\ov b\r]$ are ramification points.
Similarly, the white vertex $\l\ov w\r$ is not a ramification point if and only if $N$ is semiregular on $[G:\l w\r]$, and either no white vertex is ramification point, or each white vertex is a ramification point.
Arguing similarly shows that the face $\overline{C}=\l \ov b\ov w\r$ contains no ramification point if and only if
$|\ov b\ov w|=|bw|$, and since $G$ is transitive on the face set, each face in $F_N=[\ov G:\l\ov b\ov w\r]$
contains $k$ ramification point, where $k=0$ or 1.
We thus conclude that  there are
\[i|\ov G:\l\ov b\r|+j|\ov G:\l\ov w\r|+k|\ov G:\l\ov b\ov w\r|=|\ov G|({\frac{i}{|\ov b|} }+{\frac{j}{|\ov w|} }+{\frac{k}{|\ov b\ov w|} })\]
ramification points, where $i,j,k=0$ or $1$, such that
$i=0$ if and only if $\l b\r\cap N=1$, $j=0$ if and only if $\l w\r\cap N=1$, and $k=0$ if and only if
$\l bw\r\cap N=1$.

The conclusion for the Euler characteristics of $\calD$ and $\calD_N$ is justified by Lemma~\ref{cover-genus}.
This completes the proof of Theorem~\ref{quotient-dessin}.
\qed

Suppose that the Euler characteristic $\chi(\calD_N)$ of the quotient dessin is negative.
By Lemma~\ref{cover-genus}, the ratio $\frac{\chi(\calD)}{\chi(\calD_N)}$ has a lower bound $|N|$, and the lower bound is reached if and only if $\calD$ is a smooth covering of $\calD_N$. To analyze the upper bound of $\frac{\chi(\calD)}{\chi(\calD_N)}$, we give a Hurwitz-like bound for regular dessin first in the following lemma.

\begin{lemma}\label{h-bound-dessin}
 Let $\calD=\calD(G, b,w)$ be a regular dessin with negative Euler characteristic $\chi(\calD)$. Then
 \[|G|\le 42|\chi(\calD)|,\]
 and the equality holds if and only if
 \[\{|b|, |w|, |bw|\}=\{2,3,7\}. \]
\end{lemma}
\proof
 Since the characteristic
 \[\chi(D)=|V|-|E|+|F|=|G|(\frac{1}{|b|}+\frac{1}{|w|}+\frac{1}{|bw|}-1)<0,\]
 we have $\frac{1}{|b|}+\frac{1}{|w|}+\frac{1}{|bw|}-1<0$ and
 \[\frac{|G|}{|\chi(D)|}=\frac{1}{1-(\frac{1}{|b|}+\frac{1}{|w|}+\frac{1}{|bw|})}.\]
Without loss of generality, assume that $|b|\le |w|\le |bw|$.
If $|b|\ge 3$, then $\frac{1}{|b|}+\frac{1}{|w|}+\frac{1}{|bw|}<1$ implies that $|bw|\ge 4$, and hence $\frac{|G|}{|\chi(\calD)|}\le \frac{1}{1-(\frac{1}{3}+\frac{1}{3}+\frac{1}{4})}=12$.
Note that $(|b|,|w|)\neq (2,2)$.
If $|b|=2$ and $|w|=3$, then $|bw|\ge 7$, and hence
\begin{equation*}%\label{H-ineqn-2}
    \frac{|G|}{|\chi(D)|}=\frac{1}{1-(\frac{1}{|b|}+\frac{1}{|w|}+\frac{1}{|bw|})}\le \frac{1}{1-(\frac{1}{2}+\frac{1}{3}+\frac{1}{7})}=42.
\end{equation*}
If $|b|=2$ and $|w|>3$, then $|bw|\ge 5$, and hence $\frac{|G|}{|\chi(D)|}\le \frac{1}{1-(\frac{1}{2}+\frac{1}{4}+\frac{1}{5})}=20$.
We conclude that \[|G|\le 42|\chi(\calD)|.\]
  Moreover, the equality holds if and only if $(|b|, |w|, |bw|)=(2,3,7)$.
\qed

A \emph{Hurwitz group} is a finite group which can be generated by two elements $x,y$ such that $(|x|, |y|, |xy|)=(2,3,7)$.
It was proved by Hurwitz that a compact Riemann surface with Euler characteristic $\chi$ admits at most $84|\chi|$ conformal automorphisms. The upper bound is reached only when the full automorphism group is a Hurwitz group. For more about the Hurwitz group, see \cite{Conder1990}. By Lemma~\ref{h-bound-dessin}, a similar phenomenon happens for regular dessins.
A regular dessin $\calD(G,b,w)$ is called a Hurwitz dessin if $\{|b|, |w|, |bw|\}=\{2,3,7\}$.

\begin{theorem}\label{upbound}
 Let $N\lhd G=\l b,w\r$, and let $\overline{G}=G/N$, $\ov b=bN, \ov w=wN$.
Set $\calD=\calD(G,b,w)$ and $\calD_N=\calD(\ov G,\ov b,\ov w)$. Suppose that $\chi(\calD_N)<0$. Then
\[\frac{\chi(\calD)}{\chi(\calD_N)}\le 42|N|-41.\]
Moreover, the equality holds if and only if $\calD_N$ is a Hurwitz dessin and $\calD$ is a totally branched covering of $\calD_N$.
\end{theorem}

\proof
  Set $n=|N|$. Noting that $\ov b$ is the image of $b$ in $\ov G$, we have $|b|\le n |\ov b|$. Hence, $ -\frac{1}{|b|} \le -\frac{1}{n|\ov b|}$. Similarly, $-\frac{1}{|w|} \le -\frac{1}{n|\ov w|}$ and $-\frac{1}{|bw|} \le -\frac{1}{n|\ov{bw}|}$. Thus
  \begin{align}
   \frac{\chi(\calD)}{\chi(\calD_N)}&=\frac{-\chi(\calD)}{-\chi(\calD_N)}=\frac{|G|(1-\frac{1}{|b|}-\frac{1}{|w|}-\frac{1}{|bw|})}{-\chi(\calD_N)}
   % \notag \\
   % &
   \le \frac{|G|(1-\frac{1}{n|\ov b|}-\frac{1}{n|\ov w|}-\frac{1}{n|\ov{bw}|})}{-\chi(\calD_N)}\label{up-ineqn-1}\\
   &= \frac{|G/N|\big((1-\frac{1}{|\ov b|}-\frac{1}{|\ov w|}-\frac{1}{|\ov{bw}|}) +(n-1)\big)}{-\chi(\calD_N)}=1+(n-1)\frac{|G/N|}{-\chi(\calD_N)}\notag \\
   &\le 1+42(n-1) =42n-41=42|N|-41. \label{up-ineqn-2}
  \end{align}
  	The equality in (\ref{up-ineqn-1}) holds if and only if $\frac{|b|}{|\ov b|}=\frac{|w|}{|\ov w|}=\frac{|bw|}{|\ov{bw}|}=n$. That is the covering $\calD$ of $\calD_N$ is totally branched.
  	The equality in (\ref{up-ineqn-2}) holds if and only if $\calD_N$ is a Hurwitz dessin. This completes our proof.
\qed

\section{Unicellular regular dessins}\label{abel-sec}

In this section, we will numerate the unicellular regular dessins formed from a polygon of length $2\ell$,
and we will give a number theoretic decomposition of an integer $\ell$ by counting the number of unicellular regular dessins.

Unicellular regular dessins are characterized by the following proposition, which was obtained in \cite{1-face}.

\begin{proposition}\label{uni-face}
\begin{itemize}
\item[(1)]A bipartite graph $\Gamma$ underlies a unicellular regular dessin if and only if $\Gamma=\K_{m,n}^{(\lambda)}$ such that $\gcd(m,n)=1$ and $mn$ is even whenever $\lambda$ is even.
\item[(2)]A regular dessin $\calD(G,b,w)$ is unicellular if and only if $G=\l bw\r$ is cyclic.
\end{itemize}
\end{proposition}

Let $\calD$ be a unicellular regular dessin with face length $2\ell$, and let $H=\Aut\calD$.
Then each edge appears exactly two times on the boundary cycle of the unique face of $\calD$.
It follows from Lemma~\ref{1-cycle} that $H=\l h\r$ is a cyclic group of order $\ell$.
By Lemma~\ref{G2}, we have $\calD=\calD(\l h\r,b,w)$ with $\l b,w\r=\l h\r$,
and so $b=h^k$ and $w=h^{k'}$ for some integers $k,k'$. As $\calD(\l h\r,b,w)$ is unicellular, by Proposition~\ref{uni-face}(2),  $H=\l h\r=\l bw\r$.
Without loss of generality, we may assume that $h=bw=h^kh^{k'}=h^{k+k'}$, and hence $k'\equiv1-k$ $(\mod \ell)$.
Therefore,
\[\mbox{$\calD=\calD(\l h\r,h^k,h^{1-k})$, where $0\leq k\leq \ell-1$}.\]
\vskip0.1in

The following lemma gives a numeration of unicellular regular dessins which are formed by a polygon of length $2\ell$, and proves the statement in part~(1) of Theorem~\ref{Counting}.
\vskip0.1in

\begin{lemma}\label{number of uniface dessins}
There are exactly $\ell$ non-isomorphic unicellular regular dessins with face length $2\ell$, and the underlying graphs of these dessins are complete bipartite multigraphs.
\end{lemma}
\proof Let $\calD$ be a unicellular regular dessin with face length $2\ell$. According to the above analysis we have \[\mbox{$\calD=\calD(\l h\r,h^k,h^{1-k})$, where $0\leq k\leq \ell-1$}.\]
By Definition~\ref{def-maps}, the underlying graphs of these dessins are $\Ga=\Cos(\l h\r,\l h^k\r, \l h^{1-k}\r)$, which are complete bipartite graphs because $\l h\r=\l h^k, h^{1-k}\r=\l h^k\r\l h^{1-k}\r$,
and the edge multiplicity equals $\l h^k\r\cap\l h^{1-k}\r$.

Suppose that $\calD(\l h\r,h^{k_1},h^{1-k_1})\cong\calD(\l h\r,h^{k_2},h^{1-k_2})$ and $k_1\neq k_2$.
By Lemma~\ref{iso}, there exists $1\neq\sigma\in \Aut(\l h\r)$ such that $(h^{k_1},h^{1-k_1})^\sigma=(h^{k_2},h^{1-k_2})$,
and so $$h^\sigma=(h^{k_1}h^{1-k_1})^\sigma=h^{k_2}h^{1-k_2}=h,$$
that is, $\sigma=1$.
This contradiction implies that $\calD(\l h\r,h^{k_1},h^{1-k_1})\cong\calD(\l h\r,h^{k_2},h^{1-k_2})$ if and only if $k_1=k_2$.
Hence, there are exactly $\ell$ non-isomorphic unicellular regular dessins with face length $2\ell$.
\qed

\vskip0.1in
 Next, we count the unicellular regular dessins of face length $2\ell$ from another perspective.
 Recall the definitions of the sets $T_\ell, \calU_\ell, \calU_\ell^{(\lambda)}$ and $\calK_{m,n}^{(\lambda)}$ in Subsection~\ref{sec:f-trans}.
To compute $|\calK_{m,n}^{(\lambda)}|$ (see Lemma~\ref{lem:(m,n,lambda)}), we introduce `direct product' of regular dessins.

\begin{definition}\label{x-def}
{\rm Given $s$ regular dessins $\calD_i=\calD(G_i,b_i,w_i), i=1,2,\cdots,s$ such that $|b_1|, \dots, |b_s|$ are pairwise coprime, and so are $|w_1|, \dots, |w_s|$.
 The map $$\calD(G_1\times \dots \times G_s,b_1\dots b_s,w_1\dots w_s)$$ is called the  {\it direct product} of $\calD_1, \dots, \calD_s$,
denoted by $\calD_1\times \dots\times \calD_s$ or simply $\prod \limits_{i=1}^s \calD_i$. }
\end{definition}

The following lemma gives the factorization of $|\calK_{m,n}^{(\lambda)}|$.
\begin{lemma}\label{lem:decompose}
Let $(m,n,\lambda)\in T_{\ell}$ and let $\pi(\ell)=\{p_1,...,p_s\}$.
Then we have \[|\calK_{m,n}^{(\lambda)}|=\prod_{i=1}^s|\calK_{m_{p_i},n_{p_i}}^{(\lambda_{p_i})}|.\]
\end{lemma}
\proof
   For any $\calD_i\in \calK_{m_{p_i},n_{p_i}}^{(\lambda_{p_i})}$ with $1\leq i\leq s$, it is a routine to verify that the direct product $\calD=\prod \limits_{i}^s \calD_i \in \calK_{m,n}^{(\lambda)}$.
   To establish the lemma's equality, it suffices to demonstrate that for each $\calD\in \calK_{m,n}^{(\lambda)}$, there exists a unique decomposition $\calD=\prod\limits_{i=1}^s \calD_i$ with $\calD_i\in \calK_{m_{p_i},n_{p_i}}^{(\lambda_{p_i})}$.

   Consider $\calD\in \calK_{m,n}^{(\lambda)}$, an arbitrary unicellular regular dessin with the underlying graph $\K_{m,n}^{(\lambda)}$.
Then $(m,n,\lambda)\in T_{\ell}$ and according to the previous analysis we have $$\calD=\calD(\l h\r, h^k, h^{1-k}),$$
 where $\l h\r\cong \ZZ_{mn\lambda}$, $|\l h\r{:}\l h^k\r|=\gcd (k, \ell)=m,$ $|\l h\r{:}\l h^{1-k}\r|=\gcd(1-k, \ell)=n$ and $|\l h^k\r\cap \l h^{1-k}\r|=\lambda$.
As $\pi(\ell)=\{p_1,...,p_s\},$ the cyclic group has a natural decomposition
 $$\l h\r=\l h_1\r\times\l h_2\r\times\cdots\times\l h_s\r\cong\ZZ_{\ell_{p_1}}\times\ZZ_{\ell_{p_2}}\times\cdots\times\ZZ_{\ell_{p_s}},$$ where $h=h_1h_2\dots h_s$.
Set $\calD_i=\calD(\l h_i\r , h_i^k, h_i^{1-k})$ for $1\leq i\leq s$. Then $\calD=\prod\limits_{i}^s \calD_i$. Furthermore,
\begin{align*}
|\l h_i\r {:}\l h_i^k\r|&= \gcd (k, |h_i|) =\gcd(k, \ell_{p_i})=\gcd(k, \ell)_{p_i}=m_{p_i},\\
|\l h_i\r {:}\l h_i^{1-k}\r|&= \gcd (1-k, |h_i|) =\gcd(1-k, \ell_{p_i})=\gcd(1-k, \ell)_{p_i}=n_{p_i},\\
|\l h_i^k\r \cap \l h_i^{1-k}\r|&= \frac{| \l h_i^k\r| |\l h_i^{1-k}\r|}{|\l h_i^k\r \l h_i^{1-k}\r|}= \frac{| \l h_i^k\r| |\l h_i^{1-k}\r|}{|\l h_i\r|}=|\l h_i\r| \frac{| \l h_i^k\r| }{|\l h_i\r|}\frac{|\l h_i^{1-k}\r|}{|\l h_i\r|}=\frac{\ell_{p_i}}{m_{p_i}n_{p_i}}=\lambda_{p_i}.
\end{align*}
This gives $\calD_i\in \calK_{m_{p_i},n_{p_i}}^{(\lambda_{p_i})}$ and $\calD=\prod\limits_{i=1}^s \calD_i$ is indeed a decomposition we need.

To verify the uniqueness, let us assume $\calD_i'=\calD(\l h_i\r, h_i^{k_i}, h_i^{1-k_i})\in \calK_{m_{p_i},n_{p_i}}^{(\lambda_{p_i})}$ for $1\leq i\leq s$, such that $\calD=\prod\limits_{i}^s \calD_i' $. By Lemma~\ref{iso}, there exists a group automorphism $\sigma$ that maps $h^k=h_1^k\dots h_s^k$ to $h_1^{k_1}\dots h_s^{k_s}$ and  $h^{1-k}=h_1^{1-k}\dots h_s^{1-k}$ to $h_1^{1-k_1}\dots h_s^{1-k_s}$. Consequently,
\[h^\sigma=(h^k h^{1-k})^\sigma=h_1^{k_1}\dots h_s^{k_s}h_1^{1-k_1}\dots h_s^{1-k_s}=h_1\dots h_s=h.\]
This compels $\sigma$ to be the identity, yielding  $h_1^k\dots h_s^k=h_1^{k_1}\dots h_s^{k_s}$. Thus $h_i^k=h_i^{k_i}$ and $\calD_i=\calD_i'$. This concludes the proof.
\qed

This result tells us that, to calculate $|\calK_{m,n}^{(\lambda)}|$, we need only to calculate $|\calK_{m_p,n_p}^{(\lambda_p)}|$ for each prime divisor $p$ of $\ell=mn\lambda$.

\begin{lemma}\label{l is power}
  Let $\ell=p^e$ and $(m,n,\lambda)\in T_{\ell}$. Then
 \[
|\calK_{m,n}^{(\lambda)}|=\left\{\begin{array}{ll}
\phi(\lambda)\frac{p-2}{p-1} & \mbox{if $\lambda=\ell$},\\
\phi(\lambda) & \mbox{if $\lambda<\ell$. }
\end{array}\right.\]
\end{lemma}
\proof Let $\calD\in \calK_{m,n}^{(\lambda)}$, a unicellular regular dessin with underlying graph $\K_{m,n}^{(\lambda)}$.
Then $$\calD=\calD(\l h\r, h^k, h^{1-k})$$ for some integer $0\leq k\leq \ell-1$,
where $\l h\r\cong \ZZ_{\ell}$, $|\l h\r{:}\l h^k\r|=m, |\l h\r{:}\l h^{1-k}\r|=n$ and $|\l h^k\r\cap \l h^{1-k}\r|=\lambda$.
Since $(m,n,\lambda)\in T_{\ell}$, it follows that $\gcd(m,n)=1$ and $mn$ is even whenever $\lambda$ is even, and so
$m=1$ or $n=1$.

If $m=n=1$, then $\lambda=\ell=p^e$ and $p$ is odd. Thus, $h^k$ and $h^{1-k}$ are both generating elements of $\l h\r$,
and so $h^k\notin\langle h^p\rangle$ and $h^{1-k}\notin\langle h^p\rangle$.
Hence there are exactly $p^e-2p^{e-1}$ choices for $k$, and so
$|\calK_{1,1}^{(p^e)}|=p^e-2p^{e-1}=\phi(\lambda)\frac{p-2}{p-1}.$

If $n=1$ and $m=p^d>1$, then $\lambda=p^{e-d}$ and $|h^k|=p^{e-d}$, $|h^{1-k}|=p^e$. Which implies that $k=p^dk'$ with $1\leq k'\leq p^{e-d}$ and $\gcd(p,k')=1$.
Thus there are exactly $\phi(p^{e-d})=\phi(\lambda)$ possibilities for $k$, and so $|\calK_{m,1}^{(\lambda)}|=\phi(\lambda).$
Similarly, if $m=1$ and $n=p^d>1$, we also have $|\calK_{1,n}^{(\lambda)}|=\phi(\lambda).$ \qed

Combining Lemma~\ref{lem:decompose} and Lemma~\ref{l is power}, we have the following result which is the proof of part~(2) of Theorem~\ref{Counting}.

\begin{lemma}\label{lem:(m,n,lambda)}
    For each $(m,n,\lambda)\in T_{\ell}$, we have
    % \[|\calK_{m,n}^{(\lambda)}|=\phi(\lambda)\prod_{p\in \pi(\lambda)\setminus\pi(\ell/\lambda)}\frac{p-2}{p-1}.\]
    \[|\calK_{m,n}^{(\lambda)}|=\phi(\lambda)\prod_{p\in \pi}\frac{p-2}{p-1}\hspace*{5mm} \mbox{and}\hspace*{5mm}  |\calu_\ell^{(\lambda)}|=2^{|\sigma|}\phi(\lambda)\prod_{p\in \pi}\frac{p-2}{p-1},\]
    where $\sigma=\pi(\ell/\lambda), $ and $\pi=\pi(\lambda)\setminus\sigma$.
\end{lemma}
\begin{proof}
Let $\pi(\ell)=\{p_1,...,p_s\}$, a set of prime divisors of $\ell$.
By Lemmas~\ref{lem:decompose}-\ref{l is power}, we have
% \begin{align*}
% |\calK_{m,n}^{(\lambda)}|
% &=\prod_{i=1}^s|\calK_{m_{p_i},n_{p_i}}^{(\lambda_{p_i})}|=\left(\prod_{\lambda_{p_i}=\ell_{p_i}}\phi(\lambda_{p_i})\frac{p_i-2}{p_i-1}\right)\cdot\left(\prod_{\lambda_{p_j}<\ell_{p_j}}\phi(\lambda_{p_j})\right) \\
% &=\prod_{i=1}^s\phi(\lambda_{p_i})\cdot \left(\prod_{\lambda_{p_i}=\ell_{p_i}}\frac{p_i-2}{p_i-1}\right)=\phi(\lambda)\prod_{p\in\pi(\lambda)\setminus\pi(\ell/\lambda)}\frac{p-2}{p-1}.
% \end{align*}
\begin{align*}
|\calK_{m,n}^{(\lambda)}|
&=\prod_{i=1}^s|\calK_{m_{p_i},n_{p_i}}^{(\lambda_{p_i})}|=\left(\prod_{\lambda_{p_i}=\ell_{p_i}}\phi(\lambda_{p_i})\frac{p_i-2}{p_i-1}\right)\cdot\left(\prod_{\lambda_{p_j}<\ell_{p_j}}\phi(\lambda_{p_j})\right) \\
&=\prod_{i=1}^s\phi(\lambda_{p_i})\cdot \left(\prod_{\lambda_{p_i}=\ell_{p_i}}\frac{p_i-2}{p_i-1}\right)=\phi(\lambda)\prod_{p\in\pi}\frac{p-2}{p-1}.
\end{align*}
Note that when the numbers $\ell$ and $\lambda$ are fixed, the integer pair $(m,n)$ satisfies $(m,n,\lambda)\in T_\ell$ if and only if $mn=\ell/\lambda$ and $\gcd(m,n)=1$.
Then for each prime divisor $p$ of $\ell/\lambda$, we have either $(m_p,n_p)=(\ell_p/\lambda_p,1)$~or $(1,\ell_p/\lambda_p),$
which implies that the pair $(m_p,n_p)$ has exactly $2$ possibilities for each prime divisor $p$ of $\ell/\lambda$.
For each prime divisor $p\in \sigma$, we have $ \ell_p/\lambda_p=m_p$ or $n_p$.
Thus, there are exactly $2^{|\sigma|}$ pairs $(m,n)$ satisfying $mn=\ell/\lambda$ and $\gcd(m,n)=1$.
% There are $2^{|\sigma|}$ such pairs where $\sigma=\pi(\ell/\lambda)$.
Hence
\[|\calu_\ell^{(\lambda)}|=\sum_{\substack{mn=\ell/\lambda \\\gcd(m,n)=1 } } \phi(\lambda)\prod_{p\in \pi}\frac{p-2}{p-1}=2^{|\sigma|}\phi(\lambda)\prod_{p\in \pi}\frac{p-2}{p-1}. \]
This completes the proof.
\end{proof}

The combination of Lemma~\ref{number of uniface dessins} and Lemma~\ref{lem:(m,n,lambda)} yields an alternative counting expression for unicellular regular dessins with face length $2\ell$.

\begin{corollary}
Let $\calu_\ell$
% and $\calu_\ell^{(\lambda)}$
be the sets defined in Subsection~\ref{F-trans}.  Then
$$|\calu_\ell|=\sum_{(m,n,\lambda)\in T_{\ell}}\phi(\lambda)\prod_{p\in \pi(\lambda)\setminus\pi(\ell/\lambda)}\frac{p-2}{p-1}.$$
% \begin{eqnarray*}
%  |\calu_\ell^{(\lambda)}| &=&2^{|\sigma|}\phi(\lambda)\prod_{p\in \pi}\frac{p-2}{p-1}, \\
%  |\calu_\ell| &=&\sum_{(m,n,\lambda)\in T_{\ell}}\phi(\lambda)\prod_{p\in \pi(\lambda)\setminus\pi(\ell/\lambda)}\frac{p-2}{p-1}, \\
% \end{eqnarray*}
% where $\sigma=\pi(\ell/\lambda)$ and $\pi=\pi(\lambda)\setminus \sigma$.
\end{corollary}

 We observe that the cardinality of $\calu_\ell$ is equal to the sum of $|T_{\ell}|$ terms, however, each term is solely dependent on the factor $\lambda$ of $\ell$.
Based on this fact, a decomposition of the number theoretic form of $\ell$ will be given as follows.
\begin{lemma}
For any positive integer $\ell$, we have
\begin{equation*}
\ell=\sum_{\substack{\lambda\mid \ell \\ \lambda_2<\max\{\ell_2,2\}}}2^{|\pi'_\lambda|}\phi(\lambda)\prod_{p \in \pi_\lambda}\frac{p-2}{p-1},
\end{equation*}
    where $\pi_\lambda=\pi(\lambda)\setminus \pi(\ell/\lambda)$, and $\pi_\lambda'=\pi(\ell/\lambda)$.
\end{lemma}
% \proof
% Let $\lambda$ be a divisor of $\ell$ such that $\ell/\lambda=mn$ and $(m,n,\lambda)\in T_{\ell}$, that is,
% \[\gcd(m,n)=1~~\mbox{and}~~\lambda_2<\max\{\ell_2,2\}.\eqno(2)\]
% Then for each prime divisor $p$ of $\ell/\lambda$, we have either $(m_p,n_p)=(\ell_p/\lambda_p,1)$~or $(1,\ell_p/\lambda_p),$
% which implies that the pair $(m_p,n_p)$ has exactly $2$ possibilities for each prime divisor $p$ of $\ell/\lambda$.
% Let $\pi'_\lambda=\pi(\ell/\lambda)$ be the set of prime divisors of $\ell/\lambda$ .
% For each prime divisor $p\in \pi'_\lambda$, we have $ \ell_p/\lambda_p=m_p$ or $n_p$.
% Thus, there are exactly $2^{|\pi'_\lambda|}$ pairs $(m,n)$ satisfying the condition (2).
% By Lemma~\ref{lem:(m,n,lambda)}, there are
% $$\sum_{\substack{\lambda\mid \ell \\ \lambda_2<\max\{\ell_2,2\}}}2^{|\pi'_\lambda|}\phi(\lambda)\prod_{p \in \pi_\lambda}\frac{p-2}{p-1}$$
% non-isomorphic unicellular regular dessins with face length $2\ell$,
% where $\pi_\lambda=\pi(\lambda)\setminus \pi(\ell/\lambda)$.
% In conjunction with Lemma~\ref{number of uniface dessins}, we have $$\ell=|\calu_\ell|=\sum_{\substack{\lambda\mid \ell \\ \lambda_2<\max\{\ell_2,2\}}}2^{|\pi'_\lambda|}\phi(\lambda)\prod_{p \in \pi_\lambda}\frac{p-2}{p-1}.$$
% The proof is completed.\qed

\proof
Let $\lambda$ be a divisor of $\ell$. Then, by the definition of the set $T_\ell$,  there exist two integers $m,n$ such that $(m,n,\lambda)\in T_\ell$ if and only if $\lambda \mid\ell$ and $\lambda_2 <\max\{\ell_2, 2\}$. Therefore, by Lemma~\ref{lem:(m,n,lambda)},
$$\ell=|\calu_\ell|=\sum_{\substack{\lambda\mid \ell \\ \lambda_2<\max\{\ell_2,2\}}}|\calu_{\ell}^{(\lambda)}|=\sum_{\substack{\lambda\mid \ell \\ \lambda_2<\max\{\ell_2,2\}}}2^{|\pi'_\lambda|}\phi(\lambda)\prod_{p \in \pi_\lambda}\frac{p-2}{p-1}.$$
The proof is completed.\qed

At the end of this section, we consider the quantity of non-isomorphic complete bipartite graphs that underlie a unicellular regular dessin with face length $2\ell$,
which is a proof of part~(3) of Theorem~\ref{Counting}, and further, it offers a resolution to Problem B introduced in the initial section of \cite{1-face}.

\begin{lemma}\label{lem:|(m,n,lambda)|}
Let $\ell=p_1^{e_1}p_2^{e_2}\dots p_s^{e_s}$, where $p_1<p_2<\dots<p_s$ represent the prime divisors of $\ell$.
Then the number of non-isomorphic graphs that underlie a unicellular regular dessin with face length $2\ell$ is
% \[\frac{1}{2}\left((2e_1+\delta)(2e_2+1)\dots(2e_s+1)+\delta\right),\]
\[(2e_1+\delta)(2e_2+1)\dots(2e_s+1),\]
where $\delta=0$ if $\ell$ is even, and $\delta=1$ if $\ell$ is odd.
\end{lemma}
\begin{proof}
Note that a graph $\Gamma$ underlies a unicellular regular dessin with face length $2\ell$ if and only if $\Gamma\cong \K_{m,n}^{(\lambda)}$ as colored graphs for some $(m,n,\lambda)\in T_\ell$, where $T_\ell$ is defined as previously mentioned.
It follows that we only need to determine the cardinality of $T_\ell$.
% Since $\K_{m,n}^{(\lambda)}\cong \K_{n,m}^{(\lambda)}$, it follows that we only need to determine the cardinality of $T_+=\{(m,n,\lambda)\in T_\ell \mid m\geq n\}$.

% Let $T_-=\{(m,n,\lambda)\in T_\ell \mid m\leq n\}$.
% Then $T_\ell=T_+\cup T_-$ and $|T_+|=|T_-|$.
% As $\gcd(m,n)=1$, we infer that $T_+\cap T_-=\{(1,1,\ell)\}$ for odd $\ell$ and $T_+\cap T_-=\emptyset$ for even $\ell$.
% Hence $|T_+|=\frac{1}{2}(|T_\ell|+\delta)$, where $\delta=0$ if $\ell$ is even, and $\delta=1$ if $\ell$ is odd.

% Now we determine the cardinality of $T_\ell$.
It is apparent that $(m,n,\lambda)\in T_{\ell}$ if and only if $(m_p,n_p,\lambda_p)\in T_{\ell_p}$ for each prime divisor $p$ of $\ell$, leading to
$$|T_{\ell}|=\prod_{i=1}^s|T_{\ell_{p}}|.$$
Consequently, we need only to determine the size of $|T_{\ell}|$ when $\ell=p^e$ is a prime power.
For $(m,n,\lambda)\in T_{p^e}$, as $\gcd(m,n)=1$ and $mn$ is even whenever $\lambda$ is even,
\[(m,n,\lambda)=(1,p^{d_1},p^{e-{d_1}})~\mbox{or}~(p^{d_2},1,p^{e-{d_2}})\] where $d_i~(i=1,2)$ is an integer between $0$ and $e$.
Let $A=\{(m,n,\lambda)\in T_{p^e} | m=1\}$ and $B=\{(m,n,\lambda)\in T_{p^e} | n=1\}$.
Then $T_{p^e}=A\cup B,$ and so $$|T_{p^e}|=|A|+|B|-|A\cap B|.$$
If $p$ is odd, then $|A|=|B|=e+1$ and $A\cap B=\{(1,1,p^e)\}.$
So $|T_{p^e}|=2(e+1)-1=2e+1$.
If $p=2$, then $|A|=|B|=e$ and $A\cap B=\emptyset,$
and so $|T_{p^e}|=2e$.

In conclusion, if $\ell=p_1^{e_1}p_2^{e_2}\dots p_s^{e_s}$ with $p_1<p_2<\dots<p_s$, then \[|T_{\ell}|=\prod_{i=1}^s|T_{p_i^{e_i}}|=(2e_1+\delta)(2e_2+1)\dots(2e_s+1),\]
where $\delta=0$ if $\ell$ is even, and $\delta=1$ if $\ell$ is odd.
This completes the proof.
\end{proof}

\section{Minimal coverings of unicellular regular dessins}\label{qp-sec}

We in this section prove Theorem~\ref{qp-types}.
We first introduce O'Nan-Scott types for quasiprimitive groups.
Let $G$ be a quasiprimitive permutation group on $\Ome$, and
let $N=\soc(G)$, the socle of $G$ generated by all minimal normal subgroups.
Then $N=T^n$, where $T$ is simple and $n$ is a positive integer.
Praeger \cite{Praeger-qp} shows that there are the following eight types:
\begin{itemize}
\item[($\HS$)] {\it Holomorph Simple}: $N$ is a product of two minimal subgroups which are nonabelian simple;

\item[($\HC$)] {\it Holomorph Compound}: $N$ is a product of two minimal normal subgroups which are not simple;

\item[($\HA$)] {\it Holomorph Affine}: $N$ is abelian;

\item[($\AS$)] {\it Almost Simple}: $N$ is simple, and $\C_G(N)=1$;

\item[($\TW$)] {\it Twisted Wreath product}: $N$ is nonabelian, non-simple, and regular;

\item[($\SD$)] {\it Simple Diagonal}: the point stabilizer $N_\o$ is simple and isomorphic to $T$;

\item[($\CD$)] {\it Compound Diagonal}: the point stabilizer $N_\o\cong T^k$ with $k\ge2$;

\item[($\PA$)] {\it Product Action}: $N$ has no normal subgroup which is regular on $\Ome$.
\end{itemize}

\vskip0.08in
Let $\calM=(V,E,F)$ be a regular dessin with face length $\ell$ and let $G=\Aut\calM$, the automorphism group of $\calM$. In this section, we always suppose that the action of $G$ on the face set $F$ is faithful and quasiprimitive. Then we have the following interesting lemmas.
 % consider the action of $G$ on the face set $F$ .

\begin{lemma}\label{Frattini-argument-quasiprimitive}
% Assume that $G$ is faithful and quasiprimitive on $F$.
Let $N$ be a minimal normal subgroup of $G$.
Then there exists an element $g\in G$ such that $G=N\l g\r$, and there exists an element $x\in N$ such that
$G=\l g^{1-i}x^{-1},xg^i\r$, where $0\leq i\leq \ell-1$.
\end{lemma}
\proof
Suppose that $\calD=\calD(G,b,w)$ and set $g=bw$. Then $\l g\r$ is a face stabilizer. Since $G$ is quasiprimitive, $N$ is transitive on the face set $F$. Hence $G=N\l g\r$ and there exists an element $x\in N$ and an integer $0\leq i\leq \ell-1$ such that $w=xg^i$. The element $b=bww^{-1}=gw^{-1}=g^{1-i}x^{-1}$. Moreover $G=\l b,w\r= \l g^{1-i}x^{-1},xg^i\r$.
\qed

The types of the quasiprimitive group $G^F$ are determined in the following lemma.

\begin{lemma}\label{G-types}
The quasiprimitive group $G^F$ is of type $\HA$, $\AS$, $\TW$ or $\PA$,
and is of type $\HA$, $\AS$, $\TW$ when the dessin $\calD$ is a smooth covering of $\calD_N$.
Moreover, the quotient map $\calD_N$ is a unicellular dessin of face length equal to $2|G/N|$.
\end{lemma}
\proof
We know that, for each of the four types $\HS$, $\HC$, $\SD$, and $\CD$, the stabilizer is insoluble.
But the face stabilizer of a regular dessin is cyclic. We conclude that $G^F$ is of type $\HA$, $\AS$, $\TW$, or $\PA$.
Furthermore, if $\calD$ is a smooth covering of $\calD/N$, then $N$ acts regularly on the face set by Theorem~\ref{quotient-dessin}.
Hence, it is possible only when $G^F$ is of type $\HA$, $\AS$, $\TW$.
Since $N=\soc(G)$ is transitive on $F$, the quotient $\calM_N$ is a unicellular dessin.
So $G/N$ is cyclic and the face length of this quotient dessin is $2|G/N|$.
\qed

In the rest of this section, we prove that each quasiprimitive group of the four types indeed appears as the automorphism group of a regular dessin.

\subsection{Affine groups}\label{sec:HA}

Assume that $G$ is an affine quasiprimitive group, and let $N=\ZZ_p^d$ be the socle of $G$ such that $G/N$ is a cyclic group of order $\ell$.
Then $G=N{:}H\cong\ZZ_p^d{:}\ZZ_\ell$, and $G$ is actually primitive on face set.
Since $H\cong\ZZ_\ell$ acts irreducibly on $\ZZ_p^d$,  $\ell$ is a primitive divisor of $p^d-1$, namely, $\ell$ divides $p^d-1$ but does not divide $p^i-1$ for any $i<d$.
We conclude the properties of $G$ in the following lemma.
\begin{lemma}\label{affine-frobenius}
	Let $\calD$ be a regular dessin, and let $G=\Aut\calD$.
    Suppose that $G$ is faithful and quasiprimitive on the face set of $\calD$ of type $\HA$.
    Then $G=N{:}H\cong\ZZ_p^d{:}\ZZ_\ell\lesssim\AGL(1,p^d)$ is a Frobenius group, where $N$ is the minimal normal subgroup of $G$, $H$ is a face stabilizer and $\ell$ is a primitive divisor of $p^d-1$.
\end{lemma}
\proof
    By Lemma~\ref{Frattini-argument-quasiprimitive}, $G=N\l g\r$, where $N$ is the minimal normal subgroup of $G$ and $g$ is a generator of a face stabilizer.
    Let $H=\l g\r$.
    We have $G=N{:}H\lesssim\AGL(d,p)$ as $G$ is of affine type, and hence $N{:}H\cong\ZZ_{p}^d{:}\ZZ_\ell$ with $H\cong\ZZ_\ell$ acting irreducibly on $\ZZ_p^d$.
    Thus, $H\lesssim\GL(1,p^d)$ is semiregular on $N\setminus\{1\}$,
    which yields that $G\lesssim\AGL(1,p^d)$ is a Frobenius group and $\ell=|g|$ is a primitive divisor of $p^d-1$.
\qed

Now we give a general construction of regular dessins whose automorphism group is faithful and quasiprimitive on the face set of type $\HA$.

\begin{construction}\label{cons-HA}
{\rm
Let $\ell$ be a primitive divisor of $p^d-1$ where $p$ is a prime and $d$ is a positive integer.
Set $G=N{:}H\leqslant \AGL(1,p^d)$ where $N\cong \ZZ_p^d$ and $H=\langle h\rangle\leqslant\GL(1,p^d)$ of order $\ell$.
For any $x\in N\setminus\{1\}$ and integers $0\leqslant i,j\leqslant \ell-1$ with $\gcd(j,\ell)=1$, define
\[\calD=\calD(G,b,w)\mbox{ with $b=h^{i}x$ and $w=x^{-1}h^{j-i}$.}\]
}
\end{construction}

\begin{lemma}
    Let $\calD$ be a regular dessin such that $\Aut\calD$ is faithful and quasiprimitive on the face set of $\calD$ of type $\HA$.
    Then $\calD$ is isomorphic to some dessins given in Construction~\ref{cons-HA}.
\end{lemma}
\proof
    Assume that $\calD$ has $p^d$ faces with face length $2\ell$ where $p$ is a prime.
    Lemma~\ref{affine-frobenius} shows that $G=\Aut\calD\cong N{:}H$ is a Frobenius group with $N\cong\ZZ_{p}^d$ being the minimal normal subgroup and $H=\langle h\rangle\lesssim\GL(1,p^d)$ of order $\ell$.
    Hence, $\calD\cong \calD(G,b,w)$ for some $b,w\in G$.
    Notice that $\langle bw\rangle$ is a face stabilizer, then $\langle bw\rangle$ is conjugate to $H$.
    By Lemma~\ref{iso}, we may assume that $\langle bw\rangle=H$.
    Hence, $bw=h^j$ for some $0\leqslant j\leqslant \ell-1$ such that $\gcd(j,\ell)=1$.
    Then $b=h^ix$ and $w=x^{-1}h^{j-i}$ for some $x\in N$ and $0\leqslant i\leqslant \ell-1$.
    Notice that $\langle b,w\rangle=G$ if and only if $x\neq 1$, and hence the lemma holds.
\qed

It is known that the complete bipartite graph $\K_{2,2g+1}$ underlies a unicellular regular dessin with face length $8g+4$ and genus $g$, see \cite{1-face}.
Applying Construction~\ref{cons-HA}, we obtain the following corollary.
\begin{corollary}\label{smooth-coverings-2}
    For each non-negative integer $g$ and a prime power $p^d$, an orientable surface of genus $g$ has $p^d$-sheeted coverings if $4g+2$ is a primitive divisor of $p^d-1$.
\end{corollary}

In Construction~\ref{cons-HA}, we observe that $\calD_N\cong\calD(H,h^i,h^{j-i})$ is the unicellular regular dessin with face length $2\ell$.
As $G=N{:}H$ is a Frobenius group, it is not hard to see that $\calD$ is a smooth covering of $\calD_N$ if and only if $i\neq 0$ and $i\neq j$.
\begin{corollary}\label{HA-inf}
    A unicellular regular dessin with positive genus has infinite face-primitive smooth coverings of $\HA$ type.
\end{corollary}
\proof
    Let $\calu_{\ell,i}=\calD(H,h^i,h^{1-i})$ be a unicellular regular dessin of positive genus where $H=\langle h\rangle$ has order $\ell$.
    It follows that $i\ne 0,1$ and $\ell >2$.
    Then, by the Dirichlet theorem, there are infinitely many primes of the form $k\ell - 1$ where $k\in \mathbb{Z}^+$.
    In addition, if $p=k\ell -1$ is a prime, then $\ell \mid p^2-1$ and $\ell\nmid p-1$.
    This yields that $\ell$ is a primitive divisor of $p^2-1$ for prime $p\equiv -1\pmod \ell$.
    We can construct face-primitive regular dessin $\calD=\calD(G,h^ix,x^{-1}h^{1-i})$ with $G=N{:}H\cong \ZZ_p^2{:}\ZZ_\ell$ as in Construction~\ref{cons-HA}.
    Then $\calD$ is a smooth covering of $\calD_N\cong\calu_{\ell,i}$ since $|h^ix|=|h^i|$ and $|x^{-1}h^{1-i}|=|h^{1-i}|$ when $i\neq 0,1$.
\qed

We next study some properties of face-primitive regular dessins of $\HA$ type.
\begin{lemma}\label{HA}
    Using definitions in Construction~$\ref{cons-HA}$ with $i\neq 0,j$.
    Then
    \begin{enumerate}[\rm (i)]
    %\item $G$ is primitive on both of biparts of $\calM$ for $\gcd(i,\ell)=\gcd(j-i,\ell)=1$;

    %\item $G$ is primitive on $B=[G:\l b\r]$, and imprimitive on $W=[G:\l w\r]$ for $\gcd(i,\ell)=1$ and $\gcd(j-i,\ell)\neq 1$;

    %\item $G$ is imprimitive on $W=[G:\l b\r]$, and primitive on $W=[G:\l w\r]$ for $\gcd(i,\ell)\neq 1$ and $\gcd(j-i,\ell)=1$;

    %\item $G$ is quasiprimitive on none of the biparts for $\gcd(i,\ell)\not=1$ and $\gcd(j-i,\ell)\not=1$.
    \item $G$ is primitive on $B=[G:\l b\r]$ if and only if $\gcd(i,\ell)=1$; and
    \item $G$ is primitive on $W=[G:\l b\r]$ if and only if $\gcd(j-i,\ell)=1$.

    \end{enumerate}
\end{lemma}
\proof
    Note that $G$ is primitive on $B=[G:\l b\r]$ ($W=[G:\l w\r]$) if and only if $\langle b\rangle$ ($\langle w\rangle$) is a maximal subgroup of $G$.
    As $b=h^ix$ has order $|h^i|$ and $w=x^{-1}h^{j-i}$ has order $|h^{j-i}|$.
    It follows that $G$ is primitive on $B=[G:\l b\r]$ ($W=[G:\l w\r]$) if and only if $\gcd(i,\ell)=1$ ($\gcd(j-i,\ell)=1$).
\qed

The following theorem provides a criterion for determining isomorphisms among face-primitive regular dessins of $\HA$ type.

\begin{theorem}\label{HA-classification}
    Let $G=N{:}H\cong\ZZ_{p}^d{:}\ZZ_\ell\lesssim \AGL(1,p^d)$ such that $\ell$ is a primitive divisor of $p^d-1$.
    Assume that $H=\langle h\rangle$, $x_1,x_2\in N\setminus\{1\}$, $0\leqslant i_1,i_2,j_1,j_2\leqslant \ell-1$ such that $\gcd(j_1,\ell)=\gcd(j_2,\ell)=1$.
    Then the following are equivalent:
    \begin{enumerate}[\rm (i)]
        \item $\calD(G,b_1,w_1)\cong\calD(G,b_2,w_2)$ where $b_t=h^{i_t}x_t$ and $w_t=x_t^{-1}h^{j_t-i_t}$ for $t=1,2$.
        \item $i_2\equiv i_1p^k\pmod{\ell}$ and $j_2\equiv j_1p^k\pmod{\ell}$ for some $0\le k\le d-1$.
    \end{enumerate}
\end{theorem}
\proof
    Lemma~\ref{iso} shows that $\calD(G,b_1,w_1)\cong\calD(G,b_2,w_2)$ if and only if there exists $\sigma\in \Aut(G)$ such that $b_1^\sigma=b_2$ and $w_1^\sigma=w_2$.
    Since $G$ is Frobenius, we have
    \[G \lesssim \Aut(G)\cong \AGammaL(1,p^d)=\FF_{p^d}^+{:}\FF_{p^d}^\times {:}\l \phi\r=\FF_{p^d}^+{:}\l\mu \r {:}\l \phi\r,\]
    where $\mu$ is a generator of the cyclic group $\FF_{p^d}^\times$ and $\phi$ is the Frobenius automorphism.
    We may assume that $G\leqslant \Aut(G)$ with $N=\FF_{p^d}^+$, $H\leqslant \FF_{p^d}^\times$ and $h=\mu^{(p^d-1)/\ell}$.

    First, we assume that part~(i) holds.
    Then $b_1^\sigma=b_2$ and $w_1^\sigma=w_2$ for some $\sigma\in\Aut(G)$.
    Note that $\sigma=\varphi \phi^k$ for some $0\le k\le d-1$ and $\varphi\in\AGL(1,p^d)=\FF_{p^d}^+{:}\FF_{p^d}^\times$.
    Since $(Nh)^\varphi=Nh$ and $(Nh)^\phi=Nh^{p}$, we obtain that
    \[(Nb_1)^\sigma=(Nh^{i_1})^{\phi^k}=Nh^{i_1p^k}\mbox{ and }(Nw_1)^\sigma=(Nh^{j_1-i_1})^{\phi^k}=Nh^{(j_1-i_1)p^k}.\]
    Recall that $Nb_2=Nh^{i_2}$ and $Nw_2=Nh^{j_2-i_2}$, it follows that
    $h^{i_1p^k}=h^{i_2}$ and $h^{(j_1-i_1)p^k}=h^{j_2-i_2}$.
    Thus $i_2\equiv i_1p^k\pmod{\ell}$ and $j_2\equiv j_1p^k\pmod{\ell}$ as in part~(ii).

    Now, we assume that part~(ii) holds.
    Then
    \[\begin{aligned}
        \calD(G,b_1,w_1)&\cong \calD(G,b_1^{\phi^{k}},w_1^{\phi^k})=\calD(G,h^{i_1p^k}x_1^{\phi^{k}},(x_1^{\phi^{k}})^{-1}h^{(j_1-i_1)p^k})\\
        &=\calD(G,h^{i_2}x_1^{\phi^{k}},(x_1^{\phi^{k}})^{-1}h^{j_2-i_2}).
    \end{aligned}\]
    Note that $\langle\mu\rangle$ is transitive on $N\setminus\{1\}$, there exists $\varphi\in\langle \mu\rangle$ mapping $x_1^{\phi^{k}}$ to $x_2$.
    Thus, we have
    \[\begin{aligned}
        \calD(G,b_1,w_1)&\cong \calD(G,b_1^{\phi^{k}\varphi},w_1^{\phi^k\varphi})=\calD(G,(h^{i_2} x_1^{\phi^{k}})^\varphi,((x_1^{\phi^{k}})^{-1}h^{j_2-i_2})^\varphi)\\
        &=\calD(G,h^{i_2} x_2,x_2^{-1} h^{j_2-i_2})=\calD(G,b_2,w_2).
    \end{aligned}\]
    This yields part~(i), and we complete the proof.
\qed

Recall that the Euler's phi-function $\phi(n)$ is the number of positive integers that is not greater than $n$ and coprime to $n$.
We enumerate face-primitive regular dessins of $\HA$ type by Theorem~\ref{HA-classification}.

\begin{lemma}\label{HA-type}
    Let $\calD(p^d,\ell)$ be the set of regular dessins $\calD$ such that $\Aut\calD$ is faithful and primitive on face set $F$ of $\HA$ type with face length $2\ell$ and $|F|=p^d$.
    Then $|\calD(p^d,\ell)|=\frac{\phi(\ell)\ell}{d}$.
\end{lemma}
\proof
    By Lemma~\ref{affine-frobenius}, we have $G=\Aut\calD=N{:}H\cong\ZZ_p^d{:}\ZZ_\ell\lesssim\AGL(1,p^d)$ for $\calD\in\calD(p^d,\ell)$.
    Let $x\in N\setminus\{1\}$, and let $h$ be a generator of $H$.
    By Theorem~\ref{HA-classification}, we have
    \[\calD(p^d,\ell)=\{\calD_{i,j}\mid \mbox{$0\leqslant i,j\leqslant \ell-1$ with $\gcd(j,\ell)=1$}\},\]
    where $\calD_{i,j}=\calD(G,h^{i}x,x^{-1}h^{j-i})$.
    Notice that there are $\ell\phi(\ell)$ pairs of $(i,j)$ such that $0\leqslant i,j\leqslant \ell-1$ with $\gcd(j,\ell)=1$.
    Theorem~\ref{HA-classification} also deduces that $\calD_{i_1,j_1}\cong \calD_{i_2,j_2}$ if and only if $i_2\equiv i_1p^k\pmod{\ell}$ and $j_2\equiv j_1p^k\pmod{\ell}$ for some $0\le k\le d-1$.
    Immediately, we obtain that $|\calD(p^d,\ell)|=\frac{\phi(\ell)\ell}{d}$ as desired.
\qed

\subsection{TW type and PA type}\

Let $\calD=\calD(G,b,w)$ be a regular dessin, and suppose that $G\cong G^F$ is a quasiprimitive permutation group of type $\TW$ or $\PA$.
By definition, $G$ has a unique minimal normal subgroup
\[N=T_1\times T_2\times\dots\times T_k\cong T^k,\]
where $k\geq2$, and $T_1\cong\dots\cong T_k\cong T$ are nonabelian simple.
Further, $G=N\l g\r$ by Lemma~\ref{Frattini-argument-quasiprimitive}.
We identify $N$ with $T^k$ for convenience.

\begin{lemma}\label{TW-PA-g-form}
Let $G=N\l g\r\cong T^k\l g\r$ be such that $N$ is the unique minimal normal subgroup.
Then $G$ can be identified with a subgroup of $\Aut(N)= \Aut(T)^k{:}\S_k$ such that
\[\mbox{$g^y=(1,1,\dots, 1,a)\pi$, where $y\in \Aut(N)$, $a\in \Aut(T)$, and $\pi=(12\dots k)\in \S_k$.}\]
\end{lemma}
\proof
Let $\calT=\{T_1,T_2,\dots,T_k\}$.
Then $\l g\r$ acts by conjugation on $\calT$ transitively, and $g\in\Aut(N)= \Aut(T)^k{:}\S_k$ as $N$ is a unique minimal normal subgroup of $G$.
Thus
\[\mbox{$g=(a_1, a_2, \dots, a_k)\pi\in \Aut(N)$, where $a_i\in \Aut(T)$ for $1\le i\le k$, and $\pi\in\S_k$.}\]
Relabeling if necessary, we may assume that $\pi=(12\dots k)$.
Let $t_i=a_i a_{i+1}\dots a_k\in \Aut(T)$, and let $y=(t_1, t_2, \dots, t_k)$ where $1\le i\le k$.
Then
\[\mbox{$t_i^{-1}a_it_{i+1}=a_k^{-1}\dots a_{i+1}^{-1}a_i^{-1}a_ia_{i+1}\dots a_k=1$ for $1\le i\le k-1$,}\]
and so
	\begin{align*}
	 g^y
	 &=(t_1^{-1}, t_2^{-1}, \dots, t_k^{-1})(a_1,a_2,\dots,a_k ) \pi (t_1, t_2, \dots, t_k) \\
	 &=(t_1^{-1}, t_2^{-1}, \dots, t_k^{-1})(a_1,a_2,\dots,a_k )  (t_1, t_2, \dots, t_k)^{\pi^{-1}} \pi \\
	 &=(t_1^{-1}, t_2^{-1}, \dots, t_k^{-1})(a_1,a_2,\dots,a_k ) (t_2, t_3,\dots, t_k, t_1) \pi \\
	 &=(t_1^{-1} a_1 t_2, t_2^{-1} a_2t_3, \dots, t_k^{-1}a_kt_1) \pi \\
	 &=(1, 1, \dots,1, t_1) \pi,
	\end{align*}
as stated in the lemma with $a=t_1=a_1a_2\dots a_k \in \Aut(T)$.
\qed

\begin{lemma}\label{TW-PA-generators}
Let $G=N\l g\r=T^k\l g\r$ be such that $N$ is the unique minimal normal subgroup.
Then there exist elements $x\in N$ such that $\l x,g\r=G$.
\end{lemma}
\begin{proof}
By Lemma~\ref{TW-PA-g-form}, we may assume that $g=(1,\dots, 1,a)\pi$, where $a\in\Aut(T)$ and $\pi=(12\dots k)\in\S_k$.
Let $s,t\in T$ generate $T$ such that $|t|=2$, see \cite{King} for the existence of such a pair $(s,t)$.
Then $|s|\ge3$ because $T$  is not a dihedral group.

First, assume that $k=2$.
Let $x=(s,t)\in N=T^2$, and let $M:=\l x,x^g,x^{g^2}\r=\l (s,t),(t^a,s), (s^a, t^a)\r$.
Define $\varphi_i$ as the projection of $M$ to $T_i$ for $i=1,2$.
Then
\[\begin{aligned}
    \varphi_1(M)&\ge\langle\varphi_1(x^g),\varphi_1(x^{g^2})\rangle=\langle t^a,s^a\rangle\cong T\mbox{ and }\\
    \varphi_2(M)&\ge\langle\varphi_2(x),\varphi_2(x^{g})\rangle=\langle t,s\rangle\cong T.
\end{aligned}\]
Hence, both $\varphi_1$ and $\varphi_2$ are surjections.
Then the kernels of $\varphi_1$ and $\varphi_2$ are normal subgroups of $T$.
Note that $x^2=(s^2,1)\neq (1,1)$ lies in the kernel of $\varphi_2$, it follows that $\ker\varphi_2=T_1$.
Since $T_1^g=T_2$, we have that $M=T_1\times T_2=N$.
Then $\langle x,g\rangle=\langle M,g\rangle=\langle N,g\rangle=G$.

Next, assume that $k\ge3$.
Let $x=(s,t,1,\dots,1)$, and let $M=\langle x^{\langle g\rangle}\rangle$.
Define $\varphi_i$ as the projection of $M$ to $T_i$ for $i=1,...,k$.
Note that
\[\varphi_2(M)\ge \varphi_2\langle x,x^g\rangle=\varphi_2\langle (s,t,1,\dots,1),(1,s,t,\dots,1)\rangle\cong T.\]
Since $\l g\r$ is transitive by conjugation on $\{T_1,T_2,\dots,T_k\}$, each $\varphi_i$ is a surjection.
Note that $\ker\varphi_2\cap\cdots\cap \ker\varphi_k\lhd T_1$ contains $x^2=(s^2,1,...,1)$, it follows that $T_1\lhd M$.
By the transitivity of $\langle g\rangle$ on $\{T_1,T_2,\dots,T_k\}$, we obtain that $M=N$, and therefore $\langle x,g\rangle=\l M,g\r=\langle N,g\rangle=G$.
\end{proof}

\begin{corollary}\label{coro:PA-example}
 There are infinitely many face-quasiprimitive regular dessins of type $\PA$.
\end{corollary}

\begin{proof}
 For any nonabelian simple group $T$, let $g=(1,\dots, 1,a)\pi \in \Aut(T)^k{:}\S_k$, where $1\neq a\in \Inn(T)$ and $\pi=(12\dots k)\in \S_k$.
 Let $G=\Inn(T)^k \l g\r=N\l g\r$, where $N=\Inn(T)^k \cong T^k$ is the unique minimal normal subgroup of $G$. As $\l g\r\cap N=\l (a,a,\dots, a)\r\neq 1$,
 the action of $G$ on $[G{:}\l g\r]$ is quasiprimitive and of type $\PA$. By Lemma~\ref{TW-PA-generators}, there is an element $x\in N$ such that $\l x, g\r=G$.
 This give a face-quasiprimitive regular dessin $\calD(G,g^2x, x^{-1}g^{-1})$ of type $\PA$.
 As there are infinitely many nonabelian simple groups, we have infinitely many examples of type $\PA$ as desired.
\end{proof}

Note that the supporting surface of a regular dessin is orientable, and hence with even Euler characteristic.
% {\color{blue}
Now we construct face-quasiprimitive regular dessins of type $\TW$ which are smooth coverings of a unicellular dessin.

\begin{construction}\label{cons-TW}
{\rm
Let $k\ge5$ be an odd integer, and let $T$ be a nonabelian simple group.
Take a permutation $g=(12\dots k)\in\S_k$, and elements $s,t\in T$ such that $T=\l s,t\r$,
see \cite{King} for the existence of such elements.
Let $G=T^k{:}\l g\r=(T_1\times \dots\times T_k){:}\l g\r$, where $g$ acts on $N=T_1\times\dots\times T_k$ by
\[g:\ (t_1,t_2,\dots,t_k)\mapsto (t_k,t_1,\dots,t_{k-1}).\]
%Then $N$ is the unique minimal normal subgroup of $G$.
Let $x=(s, 1, t, 1,(ts)^{-1}, 1,\dots,1)\in N$, and let
\[\mbox{$b=g^2x^{-1}$ and $w=xg^{-1}$.}\]
Let $\calD=\calD(G,b,w)$ and $\ov \calD=\calD(\l g\r,g^2,g^{-1})$.
}
\end{construction}

\begin{lemma}\label{lem:TW}
Let $G,N,\calD$ and $\ov\calD$ be as defined in Construction~$\ref{cons-TW}$.
Then
\begin{enumerate}[{\rm(i)}]
\item $G$ is quasiprimitive of type $\TW$ on the face set of $\calD$;
\item $\ov\calD\in \calK_{1,1}^{(k)}$, of Euler characteristic equal to $3-k$;
\item $\calD$ is of Euler characteristic $(3-k)|T|^k$, and a smooth covering of $\ov\calD$.
\end{enumerate}
\end{lemma}
\begin{proof}
Let $X=\l x,x^{g^2}\r$.
Observe that $T_3<X<T_1\times\dots\times T_k$.
Since $\l g\r$ is transitive on $\{T_1,\dots,T_k\}$, we obtain that $\l x^{\l g\r}\r=T_1\times\dots\times T_k$, and so $\l x,g\r=N{:}\l g\r$.
Let $\calD=\calD(G,b,w)$ be a regular dessin.
Then $\l g\r=\l bw\r$ is a stabilizer of a face in $\calD$, and $N$ is the unique minimal normal subgroup of $G$.
Hence, $N$ is regular on the face set $[G:\l g\r]$.
So $\calD$ is a face-quasiprimitive regular dessin of type $\TW$, as in part~(i).

Since $k$ is odd, it is easily shown that $\ov\calD\cong \calD(\l g\r, g^2,g^{-1})\in \calK_{1,1}^{(k)}$, which has Euler characteristic $1+1-k+1=3-k$, as in part~(ii).

Finally, it is easily shown that
\[\begin{array}{l}
w^k=(xg^{-1})^k=xx^gx^{g^2}\dots x^{g^{k-1}} g^k=xx^gx^{g^2}\dots x^{g^{k-1}},\\
b^k=(g^2x^{-1})^k=g^{2k}(xx^{g^2}\dots x^{g^{2k-2}})^{-1}=(xx^{g^2}\dots x^{g^{2k-2}})^{-1}.
\end{array}\]
Calculation shows that
\[\begin{array}{rcl}
x&=&(s, 1, t, 1,(ts)^{-1}, 1,\dots,1),\\
x^g&=&(1,s, 1, t, 1,(ts)^{-1}, 1,\dots,1),\\
x^{g^2}&=&(1,1,s, 1, t, 1,(ts)^{-1}, 1,\dots,1),\\
%x^{g^3}&=&(1,1,1,(st)^{-1},t,s,1,\dots,1)\\
 \dots&=& \dots\\
x^{g^{k-4}}&=&((ts)^{-1},1,\dots,1,t,1),\\
x^{g^{k-3}}&=&(1,(ts)^{-1},1,\dots,1,t),\\
x^{g^{k-2}}&=&(t,1,(ts)^{-1},1,\dots,s,1),\\
x^{g^{k-1}}&=&(1,t,1,(ts)^{-1}, 1,\dots,s).\\
\end{array}\]
Then $xx^gx^{g^2}\dots x^{g^{k-1}}=1$, and thus $|w|=|g^{-1}|=k$.
Similarly, $xx^{g^2}\dots x^{g^{2k-2}}=1$, and $|b|=|g^2|=k$.
Furthermore, since $bw=g$ has order $k$, the dessin $\calD=\calD(G,b,w)$ is a smooth covering of $\calD(\l g\r, g^2,g^{-1})$.
Since $\calD$ is a smooth covering of $\ov\calD$ with normal subgroup $N=T^k$, the Euler characteristic of $\calD$ is equal to $(3-k)|N|$ by Theorem~\ref{quotient-dessin}, as in part~(iii).
\end{proof}

The following proposition tells us that each cyclic group of odd order greater than 5 has infinitely many smooth coverings which are automorphism groups of some face-quasiprimitive regular dessins of type $\TW$.

\begin{proposition}\label{prop:chi<-4}
Let $k\ge5$ be an odd integer, and $T$ a nonabelian simple group.
Then $G=T\wr\ZZ_k$ is a smooth covering of $\ZZ_k$.
\end{proposition}

\subsection{Almost simple groups}\

For a prime $r$,  a unicellular regular dessin in $\calK_{1,1}^{(r)}$ has automorphism group $\ZZ_r$.
We shall construct smooth coverings of such unicellular dessins with  transformation groups being  simple groups.

Let $r\ge5$ be a prime, and let $T=\SL(2,2^r)$.
Let $\phi$ be the automorphism of $T$ induced by the field automorphism of $\FF_{2^r}$.
Then $T$ is a simple group, and $\phi$ is of order $r$.
Let
\[G=T{:}\l\phi\r=\SigmaL(2, 2^r).\]
Let $\calD=\calD(G,b,w)$ be a regular dessin which has a smooth quotient $\calD/T\in \calK_{1,1}^{(r)}$.
Then $\SigmaL(2,2^r)$ is a smooth covering of $\l\phi\r=G/T\cong\ZZ_r$, and thus $|b|=|w|=|bw|=r$.

We need to study the elements $t\in T$ such that $\l t,\phi\r=G$.
Let
\[\cali=\{t\in T\mid \l t,\phi\r=G\},\ \mbox{and}\ \calj=\{t\in T\mid \l t,\phi\r<G\}.\]
Then $T=\cali \cup \calj$, and $|\cali|=|T|-|\calj|$.

\begin{lemma}\label{lem:SL-phi}
The field automorphism $\phi$ normalizes exactly $3$ subgroups isomorphic to $\ZZ_2^r{:}\ZZ_{2^r-1}$, $3$ subgroups  isomorphic to  $\D_{2(2^r-1)}$ and $1$ subgroup isomorphic to $\D_{2(2^r+1)}$.
\end{lemma}
\begin{proof}
Let $L$ be a subgroup that normalized by $\phi$ and isomorphic to $\ZZ_2^r{:}\ZZ_{2^r-1}$, $\D_{2(2^r-1)}$ or $\D_{2(2^r+1)}$.
Then $L$ is maximal in $T$ and $L{:}\langle\phi\rangle$ is maximal in $G=T{:}\langle\phi\rangle$.
Note that each subgroup isomorphic to $L$ of $T$ is conjugate to $L$ in $T$.
Assume that $L^t$ is normalized by $\phi$ for some $t\in T$.
Then
\[\phi\in \N_G(L^t) \iff L^{t\phi}=L^t \iff \phi^{t^{-1}}\in \N_G(L)=L{:}\langle\phi\rangle.\]
Since $\langle\phi\rangle$ is a Sylow $r$-subgroup of $L{:}\langle\phi\rangle$, there exists $\ell \in L$ such that $\l\phi^{t^{-1}}\r^{\ell}=\l\phi\r^{t^{-1}\ell}=\l\phi\r$ when $\l\phi^{t^{-1}}\r\le L{:}\langle\phi\rangle$.
Thus, $L^t$ is normalized by $\phi$ if and only if there exists $c\in\N_T(\l\phi\r)=\C_T(\phi)=\SL(2,2)$ and $\ell\in L$ such that $t=\ell c$.
Note that $L^{\ell c}=L^c$, it follows that
\[\begin{aligned}
\left|\left\{L^t:t \in T\mid \left(L^t\right)^\phi=L^t\right\}\right|&=\left|\left\{L^c:c \in \SL(2,2)\right\}\right|\\
&=\frac{|\SL(2,2)|}{|\SL(2,2)\cap L|}=\frac{6}{|L\cap \SL(2,2)|}.
\end{aligned}\]

Note that $\N_{L{:}\l \phi\r}(\l \phi\r)=\N_L(\l \phi\r){:}\l \phi\r =\C_L(\l\phi\r)\times \l \phi\r=(L\cap \SL(2,2)){\times}\l \phi\r$ and the number of Sylow $r$-subgroups in $L{:}\l \phi\r$ is $\frac{|L{:} \l\phi\r|}{|\N_{L{:}\l \phi\r}(\l \phi\r)|}$. By the Sylow theorem,
\[\frac{|L{:} \l\phi\r|}{|\N_{L{:}\l \phi\r}(\l \phi\r)|}
=\frac{|L|}{|L\cap \SL(2,2)|}\equiv 1 \pmod{r}.\]
Hence $|L\cap \SL(2,2)|\equiv |L|\pmod{r}$. For $r\ge 7$, as $2^r\equiv 2\pmod{r}$ and $|L\cap \SL(2,2)|$ is a divisor of $6=|\SL(2,2)|$, we have
$|L\cap \SL(2,2)|= 2$, $2$, or $6$  if $L\cong \ZZ_2^r{:}\ZZ_{2^r-1}$, $\D_{2(2^r-1)}$ or $\D_{2(2^r+1)}$, respectively. This is also true for $r=5$ by computations using Magma. Thus, there are $3=\frac{6}{2}$ subgroups isomorphic to $\ZZ_2^r{:}\ZZ_{2^r-1}$; $3=\frac{6}{2}$ subgroups  isomorphic to  $\D_{2(2^r-1)}$; and $1=\frac{6}{6}$ subgroup isomorphic to $\D_{2(2^r+1)}$ which are normalized by $\phi$.
\end{proof}

The next lemma estimate the size of $\calj$.
\begin{lemma}\label{lem:J<T}
    $|\calj|<\frac{5}{2^r}|T|$.
\end{lemma}
\begin{proof}
    % \proof
    Suppose that $t\in\calj$.
    Then $\langle t^{\langle\phi\rangle}\rangle =\l t,\phi\r \cap T < T=\SL(2,2^r)$, and it is normalized by $\phi$.
    Thus, $t$ lies in some maximal subgroups $S$ of $\SL(2,2^r)$ normalized by $\phi$.
    By~\cite[Table 8.1]{low-dim-book}, $S$ is isomorphic to $\ZZ_2^r{:}\ZZ_{2^r-1}$, $\D_{2(2^r-1)}$ or $\D_{2(2^r+1)}$.
    Lemma~\ref{lem:SL-phi} shows that
    \[\begin{aligned}
        |\calj|&<3|\ZZ_2^r{:}\ZZ_{2^r-1}|+3|\D_{2(2^r-1)}|+|\D_{2(2^r+1)}|\\
        &=3\frac{|T|}{2^r+1}+3\frac{|T|}{2^{r-1}(2^r+1)}+\frac{|T|}{2^{r-1}(2^r-1)}\\
        &< 3\frac{|T|}{2^r}+\frac{|T|}{2^r}+\frac{|T|}{2^r}=\frac{5}{2^r}|T|.
    \end{aligned}\]
    This completes the proof.
\end{proof}

For $1\le i\le r-1$, let $D_i=\{[\phi^i,t]\mid t\in T\}$.
\begin{lemma}\label{lem:Di}
$|D_1|=|D_2|=\dots=|D_{r-1}|=\frac{1}{6}|T|$.
\end{lemma}
\begin{proof}
Notice that, for any elements $s,t\in T$ and any integers $i,j\in\{1,\dots,r-1\}$,
\[ [\phi^i,s]=[\phi^i,t]\Longleftrightarrow s^{-1}\phi^{i}s=t^{-1}\phi^{i}t\Longleftrightarrow s^{-1}\phi^{j}s=t^{-1}\phi^{j}t.\]
Since $r$ is a prime, it follows  that $|D_1|=|D_2|=\dots=|D_{r-1}|$.
Furthermore,
\[ [\phi,s]=[\phi,t]\Longleftrightarrow s^{-1}\phi s=t^{-1}\phi t\Longleftrightarrow (st^{-1})^{-1}\phi (st^{-1})=\phi\Longleftrightarrow st^{-1}\in\C_T(\phi)=\SL(2,2).\]
Therefore, we have that
\[|D_1|=|\{[\phi,t]\mid t\in T\}|=|\{\SL(2,2)s\mid s\in T\}|=\frac{|T|}{|\SL(2,2)|}=\frac{|T|}{6}.\]
 This completes the proof.
\end{proof}

\begin{lemma}\label{lem:getx}
    There exist $1\le i<j\le r-1$ such that $D_i\cap D_j\cap \cali\not=\emptyset$.
\end{lemma}
\begin{proof}
    For $r\le 7$, we can verify the lemma using Magma.

    Now, let's assume that $r\ge 11$. For each $1\le k\le r-1$, let $\overline{D_k}=D_k\cap \cali$. By Lemma~\ref{lem:J<T}, we have  $|\overline{D_k}|\ge |D_k|-|\calj|=(\frac{1}{6}-\frac{5}{2^r})|T|$, and
    \[\sum_{i=1}^{r-1} |\overline{D_i}|\ge (r-1) (\frac{1}{6}-\frac{5}{2^r})|T|.\]
    When $r\ge 11$, the function $(r-1) (\frac{1}{6}-\frac{5}{2^r})$ increases as $r$ increases. Therefore
	\[\sum_{i=1}^{r-1} |\overline{D_i}|\ge (11-1) (\frac{1}{6}-\frac{5}{2^{11}})|T|>|T|.\]
	This implies that there exist $1\le i<j\le r-1$ such that $\overline{D_i}\cap \overline{D_j} =D_i\cap D_j\cap \cali\not=\emptyset$.
\end{proof}

In Lemma~\ref{lem:getx}, as $r$ is a prime, without loss of generality, we may assume that $i=1$, so that $D_1\cap D_j\cap\cali\not=\emptyset$.

\begin{construction}\label{Cons:Simple-Gp}
Assume that $D_1\cap D_j\cap \cali\not=\emptyset$ with $j>1$.
Let $x=[\phi,s]=[\phi^j,t]\in D_1\cap D_j$, where $s,t\in T$ such that $\l x,\phi\r=G$, and let
\[b=\phi^{j-1},\ w=\phi x.\]
\end{construction}

\begin{lemma}\label{lem:Cons:Simple-Gp}
With the group $G=\SigmaL(2,2^r)$ for $r\ge11$ and the pair $(b,w)$ produced in Construction~$\ref{Cons:Simple-Gp}$, the regular dessin $\calD(G,b,w)$ is a smooth covering of a unicellular regular dessin in $\calK_{1,1}^{(r)}$.
\end{lemma}
\begin{proof}
By definition, we have that $\l b,w\r=\l \phi^{j-1},\phi x\r=\l \phi, x\r=G$, since $1<j\leq r-1$ and $r$ is a prime.
Further,
\[\begin{array}{rl}
|b|&=|\phi^{j-1}|=r,\\
|w|&=|\phi x|=|\phi[\phi,s]|=|\phi^s|=r,\\
|bw|&=|\phi^{j-1}\phi[\phi^j,t]|=|(\phi^j)^t|=r.
\end{array}\]
Thus $\calD(G,b,w)$ is indeed a smooth covering of $\calD(\l\phi\r, \phi^{j-1}, \phi)\in\calK_{1,1}^{(r)}$.
\end{proof}

Now we are ready to state and prove the main result of this subsection.

\begin{proposition}\label{prop:AS-example}
For each prime $r\ge5$, the group $\SigmaL(2,2^r)$ is a smooth covering of $\ZZ_r$.
\end{proposition}
\begin{proof}
If $r\ge11$, the proof follows from Lemma~\ref{lem:Cons:Simple-Gp}.
For $r=5$ or $7$, computation in Magma shows that the conclusion of the proposition is true.
\end{proof}

\begin{remark}
When $r=3$, all smooth coverings of unicellular dessins in $\calK_{1,1}^{(3)}$ are on the torus whose transformation groups are abelian, and hence their automorphism groups cannot be of type $\AS$.
\end{remark}

We now summarise the arguments for proving Theorem~\ref{qp-types}.

\vskip0.1in
\noindent{\bf Proof of Theorem~\ref{qp-types}:}
Since $G$ is face-quasiprimitive, it follows that the quotient $\calD_N$ is unicellular, and so $\Aut\calD_N\cong G/N$ is cyclic.
By Lemma~\ref{G-types}, $G^F$ is of types $\HA$, $\TW$, $\AS$ and $\PA$, and in particular, if $\calD$ is a smooth covering of $\calD_N$, then $G^F$ is of types $\HA$, $\TW$ and $\AS$.

Face-quasiprimitive regular dessins of type $\HA$ are classified in Section~\ref{sec:HA}, which shows that, for each primitive divisor $\ell$ of $p^d-1$ with $p$ being a prime and $d$ a positive integer,
a cyclic group of order $\ell$ has a smooth covering $\ZZ_p^d{:}\ZZ_\ell\le\AGL(1,p^d)$.
Face-quasiprimitive regular dessins of types $\TW$ and $\AS$ are described in section 5.2 and 5.3, respectively,
and the rest statements of Theorem~\ref{qp-types} then follow from Proposition~\ref{prop:chi<-4} and Proposition~\ref{prop:AS-example}, respectively.
\qed

\section{Schur coverings}\label{Schur-sec}

A {\it quasisimple group} $G$ is a perfect group, that is, $G=G'$ such that $G/\Z(G)$ is simple.
In this case, $G$ is a {\it covering group} of $S\cong G/\Z(G)$.
The {\it Schur multiplier} of a simple group $S$ is the center of the largest covering of $S$.
A covering $G$ of $S$ is called the {\it Schur covering} of $S$ if $\Z(G)$ is the Schur Multiplier of $S$.
For instance, $S=\PSL(4,5)$ has two covering groups $\SL(4,5)$ and $\SL(4,5)/\ZZ_2$, and the first one is the Schur covering.

%\subsection{Basic properties}

For a covering group $G$ of a nonabelian simple group $S$, we notice that each generating pair of $S$ can be lifted to be a generating pair of $G$.

\begin{lemma}\label{Schur-covering}
    Let $G$ be a quasisimple group with $S\cong G/\Z(G)$, and let $\ol b,\ol w\in S$.
    Then $\l\ol b,\ol w\r= S$ if and only if $G=\l b,w\r$ for any pairs of preimages $(b,w)$ in $G$ of $(\ol b,\ol w)$.
\end{lemma}
\proof
Let $b,w\in G$ be preimages of $\ol b,\ol w$, respectively.
If $G=\l b,w\r$, then clearly $\l \ov b,\ov w\r=S$.

Assume that $\l \ov b,\ov w\r=S$.
Then $H=\l b, w\r\leqslant G$ such that $H\Z(G)/\Z(G)\cong S$.
Thus, we have $H\Z(G)=G$, and then $H\lhd H\Z(G)=G$.
Note that $G$ is a perfect group and $G/H=H\Z(G)/H\cong \Z(G)/(H\cap \Z(G))$ is abelian.
It follows that $G=H=\l b,w\r$.
\qed

Let $\ol b, \ol w$ be a generating pair of the simple group $S=G/\Z(G)$.
By the above lemma, $\calD(G,b,w)$ is always a covering of $\calD(S,\ol b,\ol w)$ for any pairs of preimages $(b,w)$ of $(\ol b,\ol w)$.

It would be interesting to determine whether $\calD(G, b, w)$ is a smooth covering of $\calD(S,\ov b,\ov w)$.
We observe that
\begin{align*}
	&\mbox{$\calD(G, b, w)$ is a smooth covering of $\calD(S,\ov b,\ov w)$}\\
	\iff& |b|=|\overline{b}|, |w|=|\overline{w}|\mbox{ and }|bw|=|\ov b\ov w|,\\
	\iff& \l b\r\cap\Z(G)=\l w\r\cap\Z(G)=\l bw\r\cap\Z(G)=1,
\end{align*}
led to the following simple lemma.

\begin{lemma}\label{smooth?}
Let $G=\l b,w\r$ be a quasisimple group with $S=G/\Z(G)=\l\ol b,\ol w\r$ where $\ol b,\ol w$ are images of $b,w$ in $S$, respectively.
Then $\calD(G,b,w)$ is a smooth covering of $\calD(S,\ov b,\ov w)$ if and only if $\l b\r\cap\Z(G)=\l w\r\cap\Z(G)=\l bw\r\cap\Z(G)=1$.
\end{lemma}

%It is a quick and easy way to determine whether $\calD(G,b,w)$ is a smooth covering of $\calD(S,\ol b,\ol w)$ by checking the conditions given in the above lemma.
%However, it is not a trivial work to determine all smooth coverings $\calD(G,b,w)$ of $\calD(S,\ol b,\ol w)$, even though it is known that the Schur multiplier of a simple group $S$ is relatively very small.
In the rest of this section, we study Schur coverings of simple groups $\PSL(2,q)$.

\subsection{Smooth Schur coverings of \texorpdfstring{$\PSL(2,q)$}{PSL(2,q)}}\label{sec:sl2q}

This section we will focus on the covering between $\SL(2,q)$ and $\PSL(2,q)$, where $q=p^f\geqslant 5$ with odd prime $p$.
For convenience, let %$\overline{\cdot}$ be the quotient map from $G$ to $S$, where
\[G=\SL(2,q),\ N=\Z(G)\mbox{ and } S=G/N\cong \PSL(2,q).\]

A pair $(b,w)$ in $G$ is called an \textit{$(\ell,m,n)$-pair} if $(|b|,|w|,|bw|)=(\ell,m,n)$.
If $g\in G$ has even order, then $\overline{g}$ has order $|g|/2$ as $G$ has a unique involution.
Let $(b,w)$ be an $(\ell,m,n)$-pair of $G$.
Then $(|\overline{b}|,|\overline{w}|,|\overline{bw}|)=(|b|,|w|,|bw|)$ if and only if $\ell mn$ is odd.
In this section, we will give a criterion (Theorem~\ref{thm:generater}) for odd integers $\ell,m,n$ such that $G$ (or $S$) is an $(\ell,m,n)$-group, which can immediately yield Theorem~\ref{thm:sl2q}.

Let $b,w\in G=\SL(2,q)$.
Then $\langle\overline{b},\overline{w}\rangle$ either equals $S$ or is a proper subgroup of $S$.
The subgroups of $S\cong\PSL(2,q)$ are well-known, and we state the classification given in~\cite[Page 19]{Mac} below, which is useful in the following discussions.
The proper subgroups of $S$ are one of the following three types:
\begin{enumerate}
    \item[(I)] \textit{`Finite triangular groups'}: $A_4$, $S_4$, $A_5$ and dihedral subgroups $D_{(q\pm 1)/k}$;.
    \item[(II)] \textit{`Affine subgroups'}: subgroups of $[q]{:}\ZZ_{(q-1)/2}$ and cyclic subgroups $\ZZ_{(q+1)/2k}$.
    \item[(III)] \textit{`Projective subgroups'}: $\PSL(2,p^{e})$ with $e\mid f$, and $\PGL(2,p^{e})$ with $2e\mid f$.
\end{enumerate}

Let $g\in G=\SL(2,q)$.
It is easy to see that $\mathrm{Tr}(g)=2$ (or $-2$, respectively) if and only if $|g|$ is in $\{1,p\}$ (or $\{2,2p\}$,respectively);
if $|g|\notin\{1,2,p,2p\}$, then $\mathrm{Tr}(g)=\lambda+\lambda^{-1}$ for some $\lambda\in\FF_{q^2}^\times$ with $|\lambda|=|g|$.
For any $\alpha,\beta,\gamma\in\FF_q$, define
\[E_q(\alpha,\beta,\gamma)=\{(b,w)\in \SL(2,q)^2: \mathrm{Tr}(b)=\alpha,\ \mathrm{Tr}(w)=\beta\mbox{ and } \mathrm{Tr}(bw)=\gamma\}.\]
We say $E_q(\alpha,\beta,\gamma)$ contains a commutative pair if there exists $(b,w)\in E_q(\alpha,\beta,\gamma)$ such that $bw=wb$.
The following lemma are given in~\cite{Mac}.

\begin{lemma}\label{lem:Mac}
    Let $\alpha,\beta,\gamma\in\FF_q$, and let $\ell,m,n\in\mathrm{Spec}(G)$ such that $\ell\leqslant m\leqslant n$ are odd.
    Then the following statements hold:
    \begin{enumerate}[{\rm(1)}]
        \item $E_q(\alpha,\beta,\gamma)\neq \emptyset$.
        % {\color{blue}\cite[Theorem 1]{Mac}}
        \item Let $H<G$ such that $\overline{H}$ is a finite triangular subgroup of $S$.
        Then $H$ is an $(\ell,m,n)$-group if and only if $H\cong \SL(2,5)$ and $(\ell,m,n)=(3,5,5)$.
        % {\color{blue}\cite[Section 8]{Mac}}
        \item For any $(b,w)\in E_q(\alpha,\beta,\gamma)$, $\langle \overline{b},\overline{w}\rangle$ is an affine subgroup if and only if $E_q(\alpha,\beta,\gamma)$ contains a commutative pair.
        % {\color{blue}\cite[Corollary 1]{Mac}}
        \item If $E_q(\alpha,\beta,\gamma)$ contains no commutative pairs, then $E_q(\alpha,\beta,\gamma)$ contains exactly two conjugacy classes of element pairs in $\SL(2,q)$; and exactly one conjugacy class of element pairs in in $\SL(2,q^2)$.
        % {\color{blue}\cite[Theorem 3]{Mac}}
    \end{enumerate}
\end{lemma}

The following lemma is a criterion for the triple in $E_q(\alpha,\beta,\gamma)$ to generate a projective subgroup.

\begin{lemma}\label{lem:subfield}
    Let $\ell,m,n\in \mathrm{Spec}(G)$, and let $(b,w)$ be an $(\ell,m,n)$-pair in $G$.
    Assume that $\langle\overline{b},\overline{w}\rangle$ is not an affine subgroup of $S$.
    For any divisor $e$ of $f$, $\langle b,w\rangle\lesssim\SL(2,p^{e})$ if and only if $\{\ell,m,n\}\subseteq \mathrm{Spec}(\SL(2,p^{e}))$.
\end{lemma}
\begin{proof}
    The sufficiency is obvious.
    We now assume that $\{\ell,m,n\}\subseteq \mathrm{Spec}(\SL(2,p^{e}))$.

    Let $\alpha$, $\beta$ and $\gamma$ be traces of $b$, $w$ and $bw$, respectively.
    Since $\{\ell,m,n\}\subseteq \mathrm{Spec}(\SL(2,p^{e}))$, it is easy to see that $\alpha,\beta,\gamma\in\FF_{p^e}$.
    Then there exists $b_0,w_0\in \SL(2,p^e)$ such that $(b_0,w_0)\in E_{p^e}(\alpha,\beta,\gamma)$ by Lemma~\ref{lem:Mac}\,(1).
    It follows that $\langle b_0,w_0\rangle\lesssim\SL(2,p^{e})$.
    Since $\langle \overline{b},\overline{w}\rangle $ is not an affine subgroup, $E_{q}(\alpha,\beta,\gamma)$ contains no commutative pairs by Lemma~\ref{lem:Mac}\,(3).
    Hence, $E_{p^e}(\alpha,\beta,\gamma)$ also contains no commutative pairs, and then $\langle \overline{b_0},\overline{ w_0}\rangle$ is not an affine subgroup of $S$.
    By Lemma~\ref{lem:Mac}\,(4), $\langle b,w\rangle$ and $\langle b_0,w_0\rangle$ are conjugate in $\SL(2,q^2)$.
    Therefore, we have that $\langle b,w\rangle\cong \langle b_0,w_0\rangle\leqslant\SL(2,p^e)$.
\end{proof}

\begin{lemma}\label{lem:33p-ppp}
    Assume that $p\ge 5$.
    \begin{enumerate}[{\rm(1)}]
        \item The quotient image of each $(3,3,p)$~$($or $(p,p,p))$~pair of $G=\SL(2,q)$ generates an affine subgroup of $S$.
        \item $S\cong \PSL(2,q)$ is a $(3,3,p)$~$($or $(p,p,p))$-group if and only if $q=p$.
    \end{enumerate}
\end{lemma}
\begin{proof}
    (1).
    We only give a proof for the case $(3,3,p)$ for part (1) since the arguments for $(p,p,p)$ is similar.
    Let $(b,w)$ be a $(3,3,p)$-pair of $G$.
    Then $\mathrm{Tr}(b)=\mathrm{Tr}(w)=\lambda+\lambda^{-1}$ and $\mathrm{Tr}(bw)=2$, where $\lambda\in\FF_{q^2}^\times$ of order $3$.
    Hence, $(b,w)\in E_q(\lambda+\lambda^{-1},\lambda+\lambda^{-1},2)$.
    Since $(b,b^{-1})$ is an commutative pair in $E_q(\lambda+\lambda^{-1},\lambda+\lambda^{-1},2)$, $\langle\overline{b},\overline{w}\rangle$ is an affine subgroup of $S$ by Lemma~\ref{lem:Mac}\,(3).
    Thus, part~(1) holds.

    (2).
    We prove the sufficiency first.
    Let $(\overline{b},\overline{w})$ be a $(3,3,p)$-(or $(p,p,p)$)-pair in $S$ with $b,w\in G$.
    Then $\{|b|,|w|,|bw|\}\subseteq \{3,6,p,2p\}\subseteq\mathrm{Spec}(\SL(2,p))$.
    By Lemma~\ref{lem:subfield}, we obtain that $\langle b,w\rangle=\SL(2,q)$ only if $q=p$.

    Now we show the necessity, and assume that $q=p$ and $\lambda\in\FF_{p^2}^\times$ with $|\lambda|=3$.

    We claim that each pair $(b,w)\in E_p(\lambda+\lambda^{-1},\lambda+\lambda^{-1},-2)$ is a $(3,3,2p)$-pair.
    It is clear that $|b|=|w|=3$.
    Since $\mathrm{Tr}(bw)=-2$, the order of $bw$ is either $2$ or $2p$.
    If $|bw|=2$, then $bw=-I$ and $\langle b,w\rangle$ is abelian.
    This is impossible as $|b|=|w|=3$.
    Hence, we have that $|bw|=2p$.

    Now, we claim that $\langle b,w\rangle=G$ for $(b,w)\in E_p(\lambda+\lambda^{-1},\lambda+\lambda^{-1},-2)$.
    Note that $\langle b,w\rangle$ is a $(3,3,2p)$-group.
    Since any abelian group cannot be a $(3,3,2p)$-group, we have that $\langle \overline{b},\overline{w}\rangle$ is not an affine subgroup of $S$ by Lemma~\ref{lem:Mac}\,(3).
    Note that $\langle\overline{b},\overline{w}\rangle$ is not a projective subgroup as $q=p$; and it is not a finite triangular subgroup by Lemma~\ref{lem:Mac}\,(2).
    Thus, we have that $\langle\overline{b},\overline{w}\rangle=S$, and then $\langle b,w\rangle=G$.

    Let $(b,w)\in E_p(\lambda+\lambda^{-1},\lambda+\lambda^{-1},-2)$.
    We have that $G=\langle b,w\rangle$ and $(|b|,|w|,|bw|)=(3,3,2p)$.
    Then $S=\langle \overline{b},\overline{w}\rangle$ and $(|\overline{b}|,|\overline{w}|,|\overline{b}\overline{w}|)=(3,3,p)$.
    Thus, $S$ is a $(3,3,p)$-group.

    With similar arguments by choosing $(b,w)\in E_p(-2,2,2)$, we can obtain that $S$ is a $(p,p,p)$-group.
\end{proof}

Now, we deal with the case where $(\ell,m,n)=(3,5,5)$.
\begin{lemma}\label{lem:35}
    \begin{enumerate}[{\rm(1)}]
        \item If $p\equiv \pm 2\pmod{5}$, then $\SL(2,p^f)$ is a $(3,5,5)$-group if and only if $p\neq 3$ and $f=2$.
        \item If $p\equiv \pm 1\mbox{ or }0\pmod{5}$, then $\SL(2,p^f)$ is a $(3,5,5)$-group if and only if $f=1$.
    \end{enumerate}
\end{lemma}
\begin{proof}
    The proofs for two parts are similar, we provide a proof for the case $p\equiv \pm 2\pmod{5}$ in part~(1).

    Assume that $\SL(2,p^f)$ is a $(3,5,5)$-group.
    Then $f$ is even as $5\notin\mathrm{Spec}(\SL(2,p^f))$ when $f$ is odd.
    Notice that $\{3,5\}\subseteq\mathrm{Spec}(\SL(2,p^2))$, and hence $f=2$ deduced by Lemma~\ref{lem:subfield}.
    It can be verified by~Magma that $\SL(2,9)$ is not a $(3,5,5)$-group.
    Hence, $p\neq 3$ and $f=2$.

    On the other hand, we assume that $G=\SL(2,p^2)$ with $p\ge 7$.
    We claim that there exists a $(3,5,5)$-pair $(b,w)$ in $G$ such that $\langle b,w\rangle\not\cong\SL(2,5)$.
    Set $S_{3,5,5}$ the set of $(3,5,5)$-pair in $G$.
    Let $\lambda_3$ and $\lambda_5$ be elements of order $3$ and $5$ in $\FF_q^\times$, respectively.
    Then $S_{3,5,5}$ is a union of sets $E_q(\lambda_3+\lambda_3^{-1},\lambda_5^{k_1}+\lambda_5^{-k_1},\lambda_5^{k_2}+\lambda_5^{k_2})$ for $k_1,k_2\in\{1,2\}$.
    By Lemma~\ref{lem:Mac}\,(4), $S_{3,5,5}$ contains $8$ conjugacy classes of element pairs of $G$.
    We remark the following facts:
    \begin{enumerate}[(a)]
        \item $G=\SL(2,p^2)$ contains two conjugacy classes of subgroups isomorphic to $\SL(2,5)$, and these are both maximal subgroups of $G$;
        \item $\SL(2,5)$ contains exactly $2$ conjugacy classes of $(3,5,5)$-pairs.
    \end{enumerate}
    Hence, there exists a pair $(b,w)\in S_{3,5,5}$ such that $\langle b,w\rangle\not\cong \SL(2,5)$, as we claimed.

    Let $(b,w)$ be a $(3,5,5)$-pair in $G$ such that $\langle b,w\rangle\not\cong\SL(2,5)$.
    Clearly that $\langle \overline{b},\overline{w}\rangle$ is not a projective subgroup of $S$ as $5\notin\mathrm{Spec}(\SL(2,p))$.
    Let $\alpha=\mathrm{Tr}(b)$, $\beta=\mathrm{Tr}(w)$ and $\gamma=\mathrm{Tr}(bw)$.
    Then any pair in $E_q(\alpha,\beta,\gamma)$ is a $(3,5,5)$-pair in $G$.
    Since an abelian group is not a $(3,5,5)$-group, $E_q(\alpha,\beta,\gamma)$ contains no commutative pairs.
    Lemma~\ref{lem:Mac}\,(3) implies that $\langle \overline{b},\overline{w}\rangle$ is not an affine subgroup of $S$.
    Thus, $\langle \overline{b},\overline{w}\rangle=S$ and $\langle b,w\rangle=G$.
    Therefore, $G$ is a $(3,5,5)$-group, which completes the proof.
\end{proof}

\begin{lemma}\label{lem:not35}
    Let $\ell,m,n\in\mathrm{Spec}(G)$ be odd integers such that $\frac{1}{\ell}+\frac{1}{m}+\frac{1}{n}<1$ and $(\ell,m,n)$ is not equal to the followings triples
    \[(3,3,p),\ (3,p,3),\ (p,3,3),\ (p,p,p),\ (3,5,5),\ (5,3,5)\mbox{ and }(5,5,3).\]
    Then the followings are equivalent.
    \begin{enumerate}[{\rm(1)}]
        \item $G=\SL(2,q)$ is an $(\ell,m,n)$-group;
        \item $S=G/N\cong \PSL(2,q)$ is an $(\ell,m,n)$-group;
        \item $\{\ell,m,n\}\not\subseteq\mathrm{Spec}(\SL(2,p^e))$ for any proper divisor $e$ of $f$.
    \end{enumerate}
\end{lemma}
\begin{proof}
    Since $(\ell,m,n)$ is a triple of odd integers and $N\cong\ZZ_2$, part~(1) obviously implies part~(2).
    Lemma~\ref{lem:subfield} deduces that part~(2) implies part~(3).

    Now we assume that part~(3) holds.
    Let $(b,w)$ be an $(\ell,m,n)$-pair of $G$.
    Then $\langle \overline{b},\overline{w}\rangle$ is not a projective subgroup of $S$.
    By Lemma~\ref{lem:Mac}\,(2), $\langle\overline{b},\overline{w}\rangle$ is not a finite triangular subgroup.
    Hence, it suffices to show that there exists an $(\ell,m,n)$-pair of $G$ whose quotient image does not generate an affine subgroup of $S$.

    Assume that $\ell\neq p$ and $m=n=p$.
    Let $\lambda_\ell\in\FF_q^\times$ of order $\ell$, and let $\alpha=\lambda_\ell+\lambda_\ell^{-1}$.
    For $(b,w)\in E_q(\alpha, 2,2)$, it is easy to see that neither $w$ or $bw$ is trivial.
    This yields that $(|b|,|w|,|bw|)=(\ell,p,p)$.
    Since $p\nmid \ell$, any abelian group is not an $(\ell,p,p)$-group.
    Then $\langle\overline{b},\overline{w}\rangle$ is not an affine subgroup of $S$ by Lemma~\ref{lem:Mac}\,(3), as desired.

    Assume that $\ell=m>3$, $p\nmid \ell m$ and $n=p$.
    Let $\lambda_\ell\in\FF_q^\times$ of order $\ell$.
    Set
    \[\alpha=\lambda_\ell+\lambda_\ell^{-1}\mbox{ and }\beta=\lambda_\ell^2+\lambda_\ell^{-2}.\]
    For $(b,w)\in E_q(\alpha, \beta,2)$, it is easy to see that $bw\neq 1$, and then $(|b|,|w|,|bw|)=(\ell,\ell,p)$.
    Since $p\nmid \ell m$, any abelian group is not an $(\ell,\ell,p)$-group.
    Then $\langle\overline{b},\overline{w}\rangle$ is not an affine subgroup of $S$ by Lemma~\ref{lem:Mac}\,(3), as desired.

    The case $\ell\neq m$, $p\nmid \ell m$ and $n=p$ is similar with the previous case by letting
    \[\alpha=\lambda_\ell+\lambda_\ell^{-1}\mbox{ and }\beta=\lambda_m+\lambda_m^{-1},\]
    where $\lambda_\ell$ and $\lambda_m$ are elements in $\FF_q^\times$ of order $\ell$ and $m$, respectively.

    The last case is that $p\nmid \ell mn$.
    We may assume that $n>3$.
    Let $\lambda_\ell$, $\lambda_m$ and $\lambda_n$ be elements in $\FF_q^\times$ of order $\ell$, $m$ and $n$, respectively.
    Set
    \[\alpha=\lambda_\ell+\lambda_\ell^{-1},\ \beta=\lambda_m+\lambda_m^{-1},\ \gamma_1=\lambda_n+\lambda_n^{-1}\mbox{ and }\gamma_2=\lambda_n^2+\lambda_n^{-2}.\]
    Let $(b,w)\in E_q(\alpha,\beta,\gamma_1)$.
    If $\langle \overline{b},\overline{w}\rangle$ is not an affine subgroup of $S$, then we are done.
    Suppose that $\langle \overline{b},\overline{w}\rangle$ is an affine subgroup of $S$.
    Lemma~\ref{lem:Mac}\,(3) deduces that $E_q(\alpha,\beta,\gamma_1)$ contains a commutative pair $(b',w')$.
    Then we obtain that $\lambda_n=\mathrm{Tr}(b'w')=\lambda_\ell\lambda_m$.
    Hence, we have that $E_q(\alpha,\beta,\gamma_2)$ does not contain a commutative pair.
    Thus, for any $(b_0,w_0)\in E_q(\alpha,\beta,\gamma_2)$, $\langle\overline{b_0},\overline{w_0}\rangle$ is not an affine subgroup of $S$ by Lemma~\ref{lem:Mac}\,(3), which completes the proof.
\end{proof}

Now we can prove the following theorem.
\begin{theorem}\label{thm:generater}
    Let $q=p^f\geqslant 5$ for odd prime $p$, and let $\ell\le m\le n$ be odd integers in $\mathrm{Spec}(\SL(2,q))$ such that $\frac{1}{\ell}+\frac{1}{m}+\frac{1}{n}<1$.
    \begin{enumerate}[{\rm(i)}]
        \item $\SL(2,q)$ is an $(\ell,m,n)$-group if and only if the followings hold.
        \begin{enumerate}[{\rm(a)}]
            \item  $\{\ell,m,n\}\not\subseteq\mathrm{Spec}(\SL(2,p^e))$ for any proper divisor $e$ of $f$;
            \item $(\ell,m,n,q)\notin\{(3,3,p,p),(p,p,p,p),(3,5,5,9)\}$.
        \end{enumerate}
        \item $\PSL(2,q)$ is an $(\ell,m,n)$-group if and only if $\{\ell,m,n\}\not\subseteq\mathrm{Spec}(\PSL(2,p^e))$ for any proper divisor $e$ of $f$.
    \end{enumerate}
\end{theorem}
\begin{proof}
   (i). The equivalency for the case where $(\ell,m,n)\notin\{(3,3,p),(p,p,p),(3,5,5)\}$ has been proved by Lemma~\ref{lem:not35}.

    Note that $\{3,p\}\subset \mathrm{Spec}(\SL(2,p))$.
    Lemmas~\ref{lem:subfield} and~\ref{lem:33p-ppp} imply that $\SL(2,q)$ is neither a $(3,3,p)$-group nor a $(p,p,p)$-group.

    For $(\ell,m,n)=(3,5,5)$.
    Lemma~\ref{lem:35} implies that $\SL(2,q)$ is a $(3,5,5)$-group if and only if $q\neq 9$ and part~(a) holds.
    Thus, we complete the proof of (i).

  (ii). Assume that $\{\ell,m,n\}\subseteq\mathrm{Spec}(\PSL(2,p^e))$ for some proper divisor $e$ of $f$.
    Let $(\overline{b},\overline{w})$ be an $(\ell,m,n)$-pair in $S=\PSL(2,q)$ with $b,w\in G=\SL(2,q)$.
    Then $(b,w)$ is a $(k_1\ell,k_2m,k_3n)$-pair in $G$ for some $k_1,k_2,k_3\in\{1,2\}$.
    Since $\ell,m,n$ are odd integers, $\{k_1\ell,k_2m,k_3n\}\subseteq \mathrm{Spec}(\SL(2,q))$, and hence $\langle b,w\rangle\lesssim\SL(2,p^e)$ by Lemma~\ref{lem:subfield}.
    Hence, $\langle \overline{b},\overline{w}\rangle\lesssim \PSL(2,p^e)<S$, a contradiction.

    Now, we assume that $\{\ell,m,n\}\not\subseteq\mathrm{Spec}(\PSL(2,p^e))$ for any proper divisor $e$ of $f$.

    If $(\ell,m,n,q)\notin\{(3,3,p,p),(p,p,p,p),(3,5,5,9)\}$, then $G=\SL(2,q)$ is generated by the pair $b,w\in G$ such that $(|b|,|w|,|bw|)=(\ell,m,n)$.
    Hence, we have that $(|\overline{b}|,|\overline{w}|,|\overline{b}\overline{w}|)=(\ell,m,n)$ and $\PSL(2,q)=\langle \overline{b},\overline{w}\rangle$.
    Hence, $\PSL(2,q)$ is an $(\ell,m,n)$-group.

    It can be checked by Magma that $\PSL(2,9)$ is a $(3,5,5)$-group.

    Let $b=\begin{pmatrix}1&1\\0&1\end{pmatrix}$ and $w=\begin{pmatrix}1&0\\-4&1\end{pmatrix}$.
    Then $|bw|=2p$ as $bw=\begin{pmatrix}-3&1\\-4&1\end{pmatrix}$ has a unique eigenvalue $-1$.
    Note that $(b,w)$ is a $(p,p,2p)$-pair in $\SL(2,p)$ and $$\langle b,w\rangle=\left\langle \begin{pmatrix}1&1\\0&1\end{pmatrix},\begin{pmatrix}1&0\\1&1\end{pmatrix}\right\rangle=\SL(2,p).$$
    Thus, we have that $\PSL(2,p)=\langle\overline{b},\overline{w}\rangle$ is a $(p,p,p)$-group.

    Finally, assume that $(\ell,m,n,q)=(3,3,p,p)$.
    We only need to show that $G=\SL(2,p)$ is a $(3,3,2p)$-group.
    Let $\lambda\in\FF_{q^2}^\times$ of order $3$, and let $\alpha=\lambda+\lambda^{-1}$.
    Remark that each element in $G$ of trace $\alpha$ has order $3$; and each element in $G$ of trace $-2$ has order $2$ or $2p$.
    Since an abelian group is neither a $(3,3,2)$-group nor a $(3,3,2p)$-group, we have that $E_q(\alpha,\alpha,-2)$ contains no commutative pairs.
    Let $(b,w)\in E_q(\alpha,\alpha,-2)$.
    Then $\langle \overline{b},\overline{w}\rangle$ is not an affine subgroup by Lemma~\ref{lem:Mac}\,(3).
    Hence, either $G=\langle b,w\rangle$ or $\langle\overline{b},\overline{w}\rangle$ is a finite triangular subgroup of $S$.
    Suppose the latter case holds.
    Then $\langle\overline{b},\overline{w}\rangle$ is a $(3,3,p)$-group.
    In~\cite[page 25]{Mac}, we have that $p=5$.
    Recall Lemma~\ref{lem:33p-ppp}\,(2) that $\PSL(2,5)$ is a $(3,3,5)$-group.
    Therefore, $G$ is a $(3,3,p)$-group for any $p\ge 5$, which completes the proof.
\end{proof}

With Magma, we obtain the following lemma.
\begin{lemma}\label{lem:schurpsl29}
    Let $G$ be a Schur covering of $S=\PSL(2,9)$, and let $\ell,m,n\in \mathrm{Spec}(S)$ such that $\ell\leqslant m\leqslant n$.
    \begin{enumerate}[{\rm(1)}]
        \item $S$ is an $(\ell,m,n)$-group if and only if $(\ell,m,n)$ equals one of
        \[\begin{array}{l}
            (2,4,5),\ (2,5,5),\ (3,3,4),\ (3,3,5),\ (3,4,5),\ (3,5,5),\\
            (4,4,4),\ (4,4,5),\ (4,5,5)\mbox{ and }(5,5,5).
        \end{array}\]
        \item Assume that $(\ell,m,n)$ is one of the triples listed above.
        Then there exists a regular dessin of $S$ which has a smooth Schur covering with automorphism group isomorphic to $G$ if and only if
        \begin{enumerate}[{\rm(i)}]
            \item $(\ell,m,n)\in\{(3,3,5),(5,5,5)\}$ when $|\Z(G)|=2$; or
            \item $(\ell,m,n)\notin\{(2,4,5),(2,5,5)\}$ when $|\Z(G)|=3$ or $6$.
        \end{enumerate}
    \end{enumerate}
\end{lemma}

We are ready to prove Theorem~\ref{thm:sl2q}.

\vskip0.1in
\noindent{\bf Proof of Theorem~\ref{thm:sl2q}:}
Let $G=\SL(2,q)$ and let $N=\Z(G)\cong \ZZ_2$.
We identify $S=\PSL(2,q)$ with $G/N$.

First, we assume that conditions in parts~(1) and (2) hold.
Then $G$ is an $(\ell,m,n)$-group by Theorem~\ref{thm:generater}\,(i).
Let $b,w\in G$ such that $G=\langle b,w\rangle$ and $(|b|,|w|,|bw|)=(\ell,m,n)$.
Then $S=\langle\overline{b},\overline{w}\rangle$.
Since $\ell,m,n$ are odd integers and $|N|=2$, we have that $(|\overline{b}|,|\overline{w}|,|\overline{b}\overline{w}|)=(\ell,m,n)$.
Define $\calD=\calD(S,\overline{b},\overline{w})$.
Then both of $\calD(G,b,w)$ and $\calD$ are of type $(\ell,m,n)$.
Therefore, $\calD(G,b,w)$ is a smooth Schur covering of $\calD$.

Now we assume that $\calD(G,b,w)$ is a smooth covering of $\calD=\calD(S,\overline{b},\overline{w})$ for some $b,w\in G$, where $\calD$ is of type $(\ell,m,n)$.
Then $\calD(G,b,w)$ is also of type $(\ell,m,n)$.
It follows that both $G$ and $S$ are $(\ell,m,n)$-groups, and hence part~(2) by Theorem~\ref{thm:generater}\,(ii).

Suppose that one of $\ell mn$ is even.
We may assume that $\ell=|\overline{b}|$ is even.
Then $|b|=2\ell$ and $\calD(G,b,w)$ is not of type $(\ell,m,n)$.
Thus, we have that $\ell mn$ is odd.

Moreover, $(\ell,m,n)\neq (3,3,3)$ since it is known that finite $(3,3,3)$-groups are solvable.
Then we have that $\frac{1}{\ell}+\frac{1}{m}+\frac{1}{n}<1$.
By Theorem~\ref{thm:generater}\,(i), we have that $(\ell,m,n,q)\notin\{(3,3,p,p),(p,p,p,p)\}$, as in part~(1).
We complete the proof.
\qed

\subsection{Smooth coverings of regular dessins on \texorpdfstring{$\PSL(2,p)$}{PSL(2,p)}}\label{sec:sl2p}

We now concentrate on regular dessins of two-dimensional (projective) linear groups.
We fix the following notations.
\begin{enumerate}
    \item $G=\SL(2,p),\ N=\Z(G)\mbox{ and } S=G/N\cong \PSL(2,p)$;
    \item $b=\begin{pmatrix}1 & 0\\1 & 1\end{pmatrix}$ and $w= \begin{pmatrix}1 & 1\\0 & 1\end{pmatrix}$ with $\ol b,\ol w$ being images of $b,w$ in $S$, respectively.
\end{enumerate}

It is clear that $|b|=|w|=p$ and $G=\l b,w\r$.
Thus the coset graph
\[\Ga=\Cos(G,\l b\r,\l w\r)\]
is connected, of valency $p$, and $G$-edge-regular.
We remark that $G$ acts $2$-transitively on $\calP$, the set of all Sylow $p$-subgroups of $G$, by conjugation and $|\calP|=p+1$.
Then each pair of Sylow $p$-subgroups of $G$ is conjugate to the pair $(\l b\r,\l w\r)$.
These facts yield that $\Ga$ is the unique (up to isomorphism) connected $G$-edge-regular graph of valency $p$.

\begin{lemma}\label{graph-val-p}
With the notations defined above, we have
\begin{itemize}
\item[(i)] every connected $G$-edge-regular graph of valency $p$ is isomorphic to $\Ga$, and
\item[(ii)] $\Ga$ is a covering of the quotient graph $\Ga_N$.
\end{itemize}
\end{lemma}
\proof
Let $\Sig$ be a graph of valency $p$ which is $G$-edge-regular, where $G\leq\Aut\Sig$.
If $\Sig$ is $G$-vertex-transitive, then $\Sigma$ is $G$-half-arc-transitive with even valency, which contradicts that $p$ is an odd prime.
Hence $\Sig$ is not $G$-vertex-transitive, which implies that $\Sig$ is a bipartite graph.
Let $\b,\o$ be two adjacent vertices of $\Sig$.
Then both $G_\b^{\Sig(\b)}$ and $G_\o^{\Sig(\o)}$ are regular, and $G_\b\cong G_\b^{\Sig(\b)}\cong G_\o^{\Sig(\o)}\cong G_\o\cong\ZZ_p$.
In particular, $\Sig=\Cos(G,G_\b,G_\o)$.

Since $G$ is 2-transitive on the set of subgroups of $G$ of order $p$, there exists an element $g\in G$ such that
\[(G_\b,G_\o)^g=(\l b\r,\l w\r).\]
It follows that $g$ induces an isomorphism from $\Sig$ to $\Ga$, as claimed in part~(i).

It is clear that the quotient graph $\Ga_N$ is of valency $p$, and hence $\Ga$ is a covering of $\Ga_N$ by definition, as in part~(ii).
\qed

The above lemma deduces that every regular dessin $\calD$ with automorphism group $G$ and valency $p$ has underlying graph isomorphic to $\Ga$.
In addition, such a regular dessin $\calD$ is always a quasi-smooth covering of its quotient $\calD_N$.
To determine whether $\calD$ is a smooth covering of $\calD_N$, we need the following lemma which classifies regular dessins with underlying graph $\Ga$.

\begin{lemma}\label{Gamma-embeddings}
Let $\calD$ be a regular dessin with $\Aut\calD=G=\SL(2,p)$ and underlying graph $\Ga$.
Then the following statements hold:
\begin{itemize}
    \item[(i)] $\calD$ is isomorphic to $\calD(G,b,w^i)$ for some $i$ with $1\leqslant i\leqslant p-1$;
    \item[(ii)] $\calD(G,b,w^i)\not\cong\calD(G,b,w^j)$ if $i\neq j$ for $1\leqslant i,j\leqslant p-1$;
    \item[(iii)] $\calD$ is a smooth covering of $\calD_N$ if and only if
    $\calD\cong\calD(G,b,w^i)$ with $|bw^i|$ odd.
\end{itemize}
\end{lemma}
\proof
By Lemma~\ref{graph-val-p}, we may assume that $\calD=\calD(G,b^k,w^i)$ for some integers $k,i$ with $1\leqslant k,i\leqslant p-1$.
Let $\lambda\in\mathbb{F}_p$ be such that $k\lambda \equiv 1\pmod p$, and let $x=\begin{pmatrix}\lambda & 0\\0&1\end{pmatrix}\in \GL(2,p)$.
Then $x$ induces an automorphism $\sigma$ of $G$ such that $g^\sigma=x^{-1}gx$ for $g\in G$.
Clearly, $(b^k)^\sigma=b^{(\lambda k)}=b$ and $(w^i)^\sigma=w^{ik}$.
Thus, $\calD=\calD(G,b^k,w^i)\cong \calD(G, b, w^{ik})$, as in part~(i).

Suppose that $\calD(G,b,w^i)\cong\calD(G,b,w^j)$ where $i\neq j$ for $1\leqslant i,j\leqslant p-1$.
Then $(w^i,b)^\s=(w^j,b)$ for some element $\s\in\Aut(G)$ by Lemma~\ref{iso}.
Thus $\s$ centralizes $b$ and normalizes $\l w\r$.
Note that $\Aut(G)=\PGL(2,p)$ and $\C_{\Aut(G)}(b)\cap \N_{\Aut(G)}(\l w\r)=1$.
So $\s=1$ and $i=j$, a contradiction.
Thus, part~(ii) holds.

We may assume that $\calD\cong \calD(G,b,w^i)$ for some $1\leqslant i\leqslant p-1$ by part~(i).
Remark that $G=\SL(2,p)$ has a unique involution which generates $N=\Z(G)$.
Hence, for any subgroup $H\leqslant G$, $H\cap N=1$ if and only if $|H|$ is odd.
Note that $|b|$ and $|w|$ are equal to the odd prime $p$.
By Lemma~\ref{smooth?}, $\calD$ is a smooth covering of $\calD_N$ if and only if $|bw^i|$ is odd, as in part~(iii).
\qed

Let $\calD=\calD(G,b,w^i)$ with $1\le i\le p-1$.
To determine whether $\calD$ is a smooth covering of $\calD_N$, by Lemma~\ref{Gamma-embeddings}\,(iii), we need to determine whether the order of the generator $bw^i$ of a face stabilizer is odd.
Note that the element $bw^i$, as a $(2\times 2)$-matrix, of $\SL(2,p)$ is conjugate to a Jordan form in $\GL(2,p^2)$.
So we need to calculate its eigenvalues.
The element $bw^i$ has the matrix form
\[bw^i=\begin{pmatrix}1 & 0\\1 & 1\end{pmatrix}\begin{pmatrix}1 & i\\0 & 1\end{pmatrix}=\begin{pmatrix}1 & i\\1 & i+1\end{pmatrix}.\]
The characteristic polynomial of $bw^i$ is
\begin{eqnarray}\label{char-poly}
\det(XI_2-bw^i)=X^2-(i+2)X+1=(X-\mu_1)(X-\mu_2),
\end{eqnarray}
where $\mu_1,\mu_2\in\FF_{p^2}$ are the two eigenvalues of $bw^i$.
Then
\begin{itemize}
\item[(a)] $\mu_1+\mu_2=i+2$, and $\mu_1\mu_2=1$;
\item[(b)] $\mu_1=\mu_2$ if and only if $(i+2)^2-4=0$.
\end{itemize}
Straightforward calculation obtains the following conclusions.
\begin{itemize}
\item[(i)] If $\mu_1=\mu_2$, then $i=p-4$, $\mu_1=\mu_2=-1$, and $bw^i$ is conjugate to the Jordan form
\[\begin{pmatrix} -1&1\\ 0&-1\end{pmatrix},\]
which has order equal to $2p$.
\item[(ii)] If $\mu_1\not=\mu_2$, then $\mu_2=\mu_1^{-1}$ as $\mu_1\mu_2=1$, and $bw^i$ is conjugate to the diagonal matrix
\[\begin{pmatrix} \mu_1 &\\ &\mu_1^{-1}\end{pmatrix},\]
which has order equal to the multiplicative order of $\mu_1$ in $\FF_{p^2}^\times$.
\end{itemize}
This leads to the following lemma.

\begin{lemma}\label{face-length}
For each integer $i$ with $1\leqslant i\leqslant p-1$, one of the following holds.
\begin{itemize}
\item[(i)] When $i=p-4$, the matrix $bw^i$ has a unique eigenvalue $-1$ and $|bw^i|=2p$.
\item[(ii)] When $i\neq p-4$, the matrix $bw^i$ has two distinct eigenvalues $\mu,\mu^{-1}\in\FF_{p^2}^\times$, and $|bw^i|$ is equal to the order of $\mu$ in $\FF_{p^2}^\times$.
\item[(iii)] Let $\mu\in\FF_{p^2}^\times\setminus\{\pm 1\}$ of order dividing $p+1$ or $p-1$.
Then $\mu$ and $\mu^{-1}$ are distinct eigenvalues of $bw^i$ when $i\equiv\mu+\mu^{-1}-2\pmod p$.
\end{itemize}
\end{lemma}
\proof
We only need to prove part~(iii).
Each element  $\mu\in\FF_{p^2}^\times\setminus\{1,-1\}$ is such that $\mu\not=\mu^{-1}$.
Assume that $i\equiv\mu+\mu^{-1}-2\pmod p$.
Then $(X-\mu)(X-\mu^{-1})=X^2-(\mu+\mu^{-1})X+\mu\mu^{-1}=X^2-(i+2)X+1$ is the characteristic polynomial of $bw^i$, and hence $\mu,\mu^{-1}$ are the two eigenvalues of $bw^i$.
\qed

Assume that $\calD=\calD(G,b,w^i)$ is a smooth covering of $\calD_N$.
By Lemma~\ref{Gamma-embeddings}\,(iii), the order $|bw^i|$ is odd, and hence $i\neq p-4$, and $|bw^i|$ equals to the multiplicative order of $\mu\in\FF_{p^2}$, by Lemma~\ref{face-length}.
Thus $|bw^i|$ is an odd divisor of $p+1$ or $p-1$ with $i\not\equiv -4\pmod{p}$.

\begin{lemma}\label{1-1}
Let $n$ be an odd divisor of $p+1$ or $p-1$ with $n>1$, and let $i$ be an integer with $1\leqslant i\leqslant p-1$.
Then the following statements hold:
\begin{itemize}
    \item[(i)] $|bw^i|=n$ if and only if $i\equiv\mu+\mu^{-1}-2\pmod p$ such that $\mu\in\FF_{p^2}^\times$ is of order $n$;
    \item[(ii)] there are exactly $\frac{\phi(n)}{2}$ values of $i$ such that $bw^i$ of order $n$;
    \item[(iii)] there are exactly $\frac{(p+1)_{2'}+(p-1)_{2'}}{2}-1$ values of $i$ such that $bw^i$ of odd order.
    %different values of $1\leqslant i\leqslant p-1$ such that $|bw^i|$ is odd.
\end{itemize}
\end{lemma}
\proof
Recall the characteristic polynomial of $bw^i$ given in (\ref{char-poly}):
\[\det(XI_2-bw^i)=X^2-(i+2)X+1.\]

(i). Suppose that $bw^i$ has order $n$. By Lemma~\ref{face-length}, $bw^i$ has two distinct eigenvalues $\mu,\mu^{-1}\in\FF_{p^2}^\times$, and and which has trace $\mu+\mu^{-1}=i+2$.
Thus $i\equiv\mu+\mu^{-1}-2\pmod p$.

Conversely, assume that $i\equiv\mu+\mu^{-1}-2\pmod p$, with $\mu\in\FF_{p^2}^\times$ being of order $n\geqslant 3$.
Then $\mu,\mu^{-1}$ are distinct eigenvalues of the $(2\times 2)$-matrix $bw^i$ by Lemma~\ref{face-length}(iii).
This yields $|bw^i|=n$, the multiplicative order of $\mu\in\FF_{p^2}$.

(ii). Noticing that the cyclic group $\FF_{p^2}^\times$ has a unique subgroup of order $n$, there are exactly $\phi(n)$ elements $\mu$ of order $n$ in $\FF_{p^2}^\times$.
Since the order $n$ of $\mu$ is odd, $\mu^{-1}\neq\mu$, and thus there are exactly $\frac{\phi(n)}{2}$ pairs $\{\mu,\mu^{-1}\}$.
So there are exactly $\frac{\phi(n)}{2}$ elements $bw^i$ which have order $n$, as in part~(ii).

(iii). We note that, if $|bw^i|=n$ is odd, then $n$ is an odd divisor of $p+1$ or $p-1$.
Let $A,B$ be the subgroups of $\FF_{p^2}^\times$ such that $A\cong\ZZ_{(p+1)_{2'}}$ and $B\cong\ZZ_{(p-1)_{2'}}$.
Then $\FF_{p^2}^\times$ has exactly $|A|+|B|-2$ non-identity elements $\mu$ of odd order dividing $p+1$ or $p-1$, and thus there are exactly $(p+1)_{2'}+(p-1)_{2'}-2$ possibilities for eigenvalues of $bw^i$.
By part~(i), such $bw^i$ is uniquely determined by $\mu+\mu^{-1}$.
Since $\mu\not=\mu^{-1}$, we conclude that
\[\begin{aligned}
    |\{bw^i \mid |bw^i|>1\ \mbox{odd},1\leqslant i\leqslant p-1\}|&=&|\{\{\mu,\mu^{-1}\}\mid 1\not=\mu\div(p\pm1)_{2'}\}|\\
    &=&\frac{1}{2}\left((p+1)_{2'}+(p-1)_{2'}-2\right).
\end{aligned}\]
This proves part~(iii).
\qed

We notice that, for each prime $p\ge5$, the number $\frac{(p+1)_{2'}+(p-1)_{2'}}{2}-1$ is such that
\[0<\frac{(p+1)_{2'}+(p-1)_{2'}}{2}-1<\frac{(p+1)+(p-1)}{ 2}-1,\]
and so we obtain the following conclusion for regular dessins with underlying graph $\Ga=\Cos(G,\l b\r,\l w\r)$, by Lemmas~\ref{Gamma-embeddings} and \ref{1-1}\,(iii).

\begin{corollary}\label{Orders}
For any prime $p\geqslant 5$, the following statements hold:
\begin{itemize}
\item[(i)] for each divisor $\ell$ of $p\pm 1$ with $\ell\neq 1,2$, there exists a regular dessin with underlying graph $\Ga$, automorphism group $\SL(2,p)$, and face length $2\ell$;
\item[(ii)] there are exactly $\frac{(p+1)_{2'}+(p-1)_{2'}}{2}-1$ regular dessins (up to isomorphism) $\calD$ with underlying graph $\Ga$ such that $\Aut\calD\cong\SL(2,p)$ and $\calD$ is a smooth covering of $\calD_N$.
\end{itemize}
\end{corollary}

Part~(ii) of this corollary tells us that there always exist regular dessins $\calD=\calD(G,b,w^i)$ which are smooth coverings of $\calD_N$.
However, for a given prime $p$, it is not clear which positive integer $i$ is such that $\calD$ is a smooth covering of $\calD_N$ since it is not clear that $|bw^i|$ is odd or not.
We next introduce a method to do this by analyzing polynomials over integers.

In \cite{Lehmer}, Lehmer shows that $2\cos(\frac{2k\pi}{n})$ is an algebraic integer of degree $\frac{\phi(n)}{2}$.
Let $\psi_n(X)$ be a minimal polynomial of $2\cos(\frac{2k\pi}{n})$, which is
% \[\psi_n(X)=\prod_{0\leqslant k< n/2\atop \gcd(k,n)=1}\mbox{$\left(X-2\cos(\frac{2k\pi}{n})\right)$}.\]
\[\psi_n(X)=\prod_{\genfrac{}{}{0pt}{3}{0\leqslant k< n/2}{\gcd(k,n)=1}}\mbox{$\left(X-2\cos(\frac{2k\pi}{n})\right)$}.\]
This polynomial has integer coefficients, refer to \cite{shortlist} for a list of $\psi_n(2X)$ for $n\leqslant 22$, and \cite{psi} for more details about calculations of $\psi_n(X)$.
Set $\psi_n^*(X)=\psi_n(X-2)$.
The next lemma gives a relation between the order of $bw^i$ and the value $\psi_n^*(i)$.

\begin{lemma}\label{lem:psi}
    Let $n\geqslant 3$ be a divisor of either $p+1$ or $p-1$.
    Then $bw^i$ has order $n$ if and only if $\psi_n^*(i)$ is divisible by $p$.
\end{lemma}
\proof
    Let $f\in \ZZ[X]\mapsto \overline{f}\in\ol{\FF_p}[X]$ and $x\in\ZZ\mapsto\overline{x}\in\overline{\FF_p}$ be the natural moduli maps.
    It is suffices to show that $\overline{I}$ is the set of all roots of $\ol{\psi_n^*}$ where $I:=\{1\leqslant i\leqslant p-1\mid|bw^i|=n\}$.

    From \cite[page 165]{Lehmer}, we have the following equality
    \[\Phi_n(X)=\psi_n(X+X^{-1})X^{\frac{\phi(n)}{2}},\]
    where $\Phi_n(X)$ is the $n$-th \textit{cyclotomic polynomial}.
    By Lemma~\ref{1-1}\,(i), $\ol{i+2}=\mu+\mu^{-1}$ for some $\mu\in\FF_{p^2}^\times$ of order $n$.
    Then we have
    \[\ol{\psi_n^*(i)}\mu^{\frac{\phi(n)}{2}}=\ol{\psi_n(i-2)}\mu^{\frac{\phi(n)}{2}}=\ol{\psi_n}(\mu+\mu^{-1})\mu^{\frac{\phi(n)}{2}}=\overline{\Phi_n}(\mu).\]
    Recall that $\Phi_n(X)$ divides $X^n-1$, it follows that $\overline{\Phi_n(X)}$ divides $\overline{X^n-1}$.
    Since $\mu$ has order $n$, we have $\overline{\Phi_n}(\mu)=0$, and hence $\ol{\psi_n^*(i)}\mu^{\frac{\phi(n)}{2}}=0$.
    It follows that $\overline{i}$ is a root of $\ol{\psi_n^*}$ for each $i\in I$.
    Recall that $\deg\psi_n^*=\frac{\phi(n)}{2}$ and $|I|=\frac{\phi(n)}{2}$ by Lemma~\ref{1-1}\,(ii).
    Therefore, $\overline{I}$ is exactly the set of all roots of $\overline{\psi_n^*}$.
\qed

We list $\psi_n^*(X)$ in Table~\ref{tablehni} for $3\leqslant n\leqslant 10$, and provide all pairs $(p,i)$ for $p\leqslant 19$ such that $\psi_n^*(i)\equiv 0\pmod p$ in the last column of the table.
In particular, we can easily observe from the table that $\psi_n^*$ has a unique root $p-3$ (or $p-2$, $p-1$, respectively) when $n=3$ (or $n=4$, $n=6$, respectively).
This leads to the following corollary.
\begin{corollary}
    For $p\geqslant 5$ and $n\in\{3,4,6\}$, there exists a unique (up to isomorphism) regular dessin $\calD$ with automorphism group $\SL(2,p)$ and underlying graph $\Ga$ of face size $2n$.
    In particular, $\calD\cong\calD(G,b,w^i)$ with $i=p-3$ (or $p-2$, $p-1$, respectively) when $n=3$ (or $n=4$, $n=6$, respectively).
\end{corollary}
\begin{table}[htb]
    \begin{tabular}{lll}
        \Xhline{2pt}
        $n$	&$\psi_n^*(X)$& $(p,i)$ with $p\leqslant 19$	\\
        \Xhline{1pt}
        \vspace{0.1cm}
        $3$		&	$X+3$&$(5,2),(7,4),(11,8),(13,10),(17,14),(19,16)$\\
        \vspace{0.1cm}
        $4$		&	$X+2$	&$(5,3),(7,5),(11,9),(13,11),(17,15),(19,17)$\\
        \vspace{0.1cm}
        $5$		&	$X^2+5X+5$	&$(11,1),(11,5),(19,2),(19,12)$\\
        \vspace{0.1cm}
        $6$		&	$X+1$	&$(5,4),(7,6),(11,10),(13,12),(17,16),(19,18)$\\
        \vspace{0.1cm}
        $7$		&	$X^3+7X^2+14X+7$	&$(13,5),(13,6),(13,8)$\\
        \vspace{0.1cm}
        $8$		&	$X^2+4X+2$	&$(7,1),(7,2),(17,4),(17,9)$\\
        \vspace{0.1cm}
        $9$		&	$X^3+6X^2+9X+3$	&$(17,5),(17,11),(17,12),(19,1),(19,5),(19,7)$\\
        \vspace{0.1cm}
        $10$	&	$X^2+3X+1$  &$(11,2),(11,6),(19,3),(19,13)$\\
        \Xhline{2pt}
    \end{tabular}
    \caption{$\psi_n^*(X)$ for  $3\le n \le 10$}\label{tablehni}
\end{table}

\subsection{Fibonacci coverings}\label{sec:Fibonacci}
Let $G=\SL(2,p)$, and let $\calD=\calD(G,b,w^i)$ with matrices $b,w$ given in the previous section.
Lemma~\ref{Gamma-embeddings}(iii) shows that $\calD$ is a smooth covering of $\calD_N$ if and only if $n:=|bw^i|$ is odd.
For a given odd integer $n\geqslant 3$, we can apply Lemma~\ref{lem:psi} to calculate roots of $\psi_n^*$ modulo $p$.
If we fix the integer $i$, what can we say about the relations between $p$ and $|bw^i|$?

It would be notably interesting to consider the case $i=1$, since we have
\[
    (bw)^k=\begin{pmatrix}1&1\\1&2\end{pmatrix}^k=\begin{pmatrix}F_{2k-1} & F_{2k}\\F_{2k} & F_{2k+1}\end{pmatrix},
\]
where $F_1=1$, $F_2=1$, and $F_k=F_{k-2}+F_{k-1}$, that is, $F_k$ is the $k$-th \textit{Fibonacci number}.
Thus the order $|bw|$ is highly related to the Fibonacci sequence modulo $p$.
It is well-known that the Fibonacci sequence is periodic under any modulus.
The study of the periods of the Fibonacci sequence under moduli was initiated by Wall in 1960~\cite{Wall} (see~\cite{Alex} for some recent achievements).

\begin{lemma}
    Let $p\geqslant5$ be a prime.
    Then one of the following holds:
    \begin{itemize}
        \item[(i)] if $p=5$, then $bw$ is of order $10$;

        \item[(ii)] if $p\equiv \pm 1 \pmod 5$, then $|bw|\div{\frac{p-1}{2}}$;
        in particular, $bw$ is of odd order if $p\equiv3\pmod 4$;

        \item[(iii)] if $p\equiv\pm2 \pmod 5$, then $|bw|\div(p+1)$ but $|bw|\notdiv{\frac{p+1}{2}}$, so that $bw$ is of even order.
    \end{itemize}
\end{lemma}
\proof
As noticed above, we have the $k$-th power
\[(bw)^k=
\begin{pmatrix}
F_{2k-1} & F_{2k}\\
F_{2k} & F_{2k+1}
\end{pmatrix},
\]
where $F_1=F_2=1$, and $F_i$ is a Fibonacci number.
Thus $|bw|=k$ if and only if $k$ is the least positive integer such that
\[(bw)^k=
\begin{pmatrix}
F_{2k-1} & F_{2k}\\
F_{2k} & F_{2k+1}
\end{pmatrix}\equiv I_2\pmod p.
\]
This is equivalent to that $k$ is the least positive integer such that
\[(F_{2k},F_{2k+1})\equiv(0,1)\pmod p,\]
and so $2k$ is the period of the Fibonacci sequence modulo $p$.

When $p=5$, then $bw$ is of order $2p=10$ by Lemma~\ref{face-length}.
Assume that $p>5$.
Then $p\equiv 1,2,3$ or 4 $(\mod 5)$.
By Theorem 3 and Lemma 3 in~\cite{Fibonacci}, we obtain that
\begin{itemize}
\item[(i)] $|bw|~|~\frac{p-1}{2}$, if $p\equiv\pm 1~(\mbox{mod}~5)$; or
\item[(ii)] $|bw|~|~(p+1)$ but $k~\notdiv~\frac{p+1}{2}$, if $p\equiv\pm 2~(\mbox{mod}~5)$.
\end{itemize}
For the former, if $p\equiv3$ $(\mod 4)$, then $bw$ is of odd order.
For the latter, $bw$ is of even order.
This completes the proof.
\qed

The well-known Dirichlet's theorem says that there are infinitely many primes $p\equiv a\pmod b$ for any positive coprime integers $a$ and $b$.
The following corollary is immediately held by Lemma~\ref{Gamma-embeddings}(iii).
\begin{corollary}
    Let $\calD=\calD(G,b,w)$ for $p>5$.
    Then
    \begin{itemize}
        \item[(i)] $\calD$ is not a smooth covering of $\calD_N$ when $p\equiv \pm 2\pmod 5$;
        \item[(ii)] there are infinitely many primes $p$ such that $\calD$ is not a smooth covering of $\calD_N$;
        \item[(iii)] $\calD$ is a smooth covering of $\calD_N$ when $p\equiv 11\mbox{ or }19\pmod {20}$;
        \item[(iv)] there are infinitely many primes $p$ such that $\calD$ is a smooth covering of $\calD_N$.
    \end{itemize}
\end{corollary}

We remark the following facts:
\begin{itemize}
    \item[(i)] $|bw|=25$ is odd for $p=101$ and $|bw|=20$ is even for $p=41$, while $41\equiv 101\equiv 1\pmod{20}$;
    \item[(ii)] $|bw|=7$ is odd for $p=29$ and $|bw|=54$ is even for $p=109$, while $29\equiv 109\equiv 9\pmod{20}$.
\end{itemize}
We cannot simply determine the parity of $|bw|$ for primes $p\equiv 1\mbox{ or }9\pmod{20}$.
As far as we know, the sufficient and necessary condition prime $p\equiv 1\mbox{ or }9\pmod{20}$ such that the period of Fibonacci sequence modulo $p$ is divisible by $4$ is still unknown.

\vskip0.2in

\noindent{\bf Data availability}
\vskip0.1in
No data was used for the research described in the article.

\vskip0.2in

\noindent{\bf Declaration of competing interest}
\vskip0.1in
We have no conflict of interest.

\vskip 3mm

 \noindent{\bf Acknowledgement}:\ We thank Yuhan Cai for his interest in and suggestions for the work of Section~\ref{sec:Fibonacci}.

\end{document}